\documentclass{article}
\usepackage[utf8]{inputenc}
\usepackage{amsfonts}
\usepackage{amsmath}
\usepackage{amsthm}
\usepackage{graphicx}
\usepackage{float}
\usepackage{xypic}
\usepackage{color}   
\usepackage{hyperref}
\usepackage{tikz-cd}
\usepackage{chngcntr}
\usepackage{url}
\usepackage{lipsum}
\usepackage{xfrac}
\usepackage{mathrsfs}
\usepackage{enumerate}
\usepackage{amssymb}
\usepackage{faktor}
\usepackage{enumitem}
\usepackage[small]{titlesec}

\setenumerate{fullwidth,itemindent=\parindent,listparindent=\parindent,itemsep=0ex,partopsep=0pt,parsep=0ex}
\newlist{steps}{enumerate}{1}
\setlist[steps, 1]{label = \textbf{Step \arabic*}:}

\counterwithout{figure}{section}
\newcommand{\mathleft}{\@fleqntrue\@mathmargin0pt}
\newtheorem{lemma}[subsubsection]{Lemma}
\newtheorem*{definition*}{Definition}
\newtheorem*{proposition*}{Proposition}
\newtheorem*{remark*}{Remark}
\newtheorem*{theorem*}{Theorem}
\newtheorem*{observation*}{Observation}
\newtheorem*{assumption*}{Assumption}
\newtheorem*{lemma*}{Lemma}
\newtheorem*{claim*}{Claim}
\newtheorem{theorem}[subsubsection]{Theorem}
\newtheorem{proposition}[subsubsection]{Proposition}
\newtheorem{definition}[subsubsection]{Definition}

\newtheorem{corollary}[subsubsection]{Corollary}
\newtheorem{remark}[subsubsection]{Remark}

\newtheorem{thmx}{Theorem}
\newtheorem{corx}{Corollary}

\newtheorem{propx}{Proposition}

\newtheorem*{lem*}{Key Lemma}
\newtheorem*{thm*}{Main Theorem}
\newtheorem*{assum*}{Assumption}

\usepackage[textwidth=16.5cm,textheight=21cm]{geometry}

\title{On Dynamical Parameter Space of Cubic Polynomials with a Parabolic Fixed Point}
\author{Runze Zhang}
\date{}

\begin{document}

\maketitle
\begin{abstract}
    This article focus on the connected locus of the cubic polynomial slice $Per_1(\lambda)$ with a parabolic fixed point of multiplier $\lambda=e^{2\pi i\frac{p}{q}}$. We first show that any parabolic component, which is a parallel notion of hyperbolic component, is a Jordan domain. Moreover, a continuum $\mathcal{K}_\lambda$ called the central part in the connected locus is defined. This is the natural analogue to the closure of the main hyperbolic component of $Per_1(0)$. We prove that $\mathcal{K}_\lambda$ is almost a double covering of the filled-in Julia set of the quadratic polynomial $P_{\lambda}(z) = \lambda z+z^2$.
\end{abstract}

\section{Introduction}
The dynamics of iteration of degree $d\geq 2$ complex polynomials acting on the complex plane is very rich. As a consequence, the connected locus $\mathcal{C}^{d} :=\{f\in\mathcal{P}^d;\,\text{ the Julia set of }f \text{ is  connected.}\}$ presents complicate fractal structures, where $\mathcal{P}^d\cong\mathbb{C}^{d-1}$ is the parameter space of degree $d$ polynomials. When $d=2$, $\mathcal{C}^2\subset\mathbb{C}$ is known as the Mandelbrot set $\textbf{M}$. The famous MLC conjecture asserting that $\textbf{M}$ is locally connected is a central problem in the study of dynamics of quadratic polynomials, since it is equivalent to the combinatorial rigidity conjecture and will imply the density of hyperbolicity conjecture. It also implies that $\mathbf{M}$ has a topological model of a certain quotient of the closed unit disk. However when $d$ augments, $\mathcal{C}^d$ exhibits much more sophisticated structures and the analogue of MLC conjecture for $\mathcal{C}^3$ turns out to be false \cite{lavaurs}. For this reason, Milnor suggested to investigate $\mathcal{C}^3$ by restricting the parameter space to complex 1-dimensional algebraic curves by adding dynamical conditions. In fact, when restricting to 1-dimensional slices, the local connectivity usually holds for non-renormalisable or finitely renormalisable parameters (\cite{Roesch1}, \cite{WANG2021107554}, \cite{misha}). This paper will focus on $Per_1(\lambda)$, the 1-dimensional slice of the cubic polynomials which have a parabolic fixed point with multiplier $\lambda =e^{2\pi i\frac{p}{q}}$ ($p,q$ coprime). We aim at investigating local connectivity at certain parameters in $\mathcal{C}_\lambda := \mathcal{C}^3\cap Per_1(\lambda)$ and giving a global description of $\mathcal{C}_\lambda$. As a byproduct, we prove that $\mathcal{C}_\lambda$ presents a big patch called "the central part" which is almost a double covering of the filled-in Julia set of $P_\lambda(z) =\lambda z+z^2$.

More precisely, consider the space of unitary cubic polynomials fixing the origin 0:
\begin{equation}\label{eq.family_a}
   f_{\lambda,a}(z) = \lambda z+az^2 + z^3,\,\,(\lambda,a)\in\mathbb{C}^2.
\end{equation}
Notice that every cubic polynomial is affinely conjugated to one of the polynomials of form (\ref{eq.family_a}). By fixing $\lambda$, one gets the slice
\[Per_{1}(\lambda) := \{f_{\lambda,a};\,a\in \mathbb{C}\}\cong\mathbb{C}\]
and its corresponding connected locus $\mathcal{C}_{\lambda} := \{a\in\mathbb{C};\,J_{\lambda,a} \text{ is connected}\}$, where $J_{\lambda,a}$ is the Julia set of $f_{\lambda,a}$. It is a classical result that the Julia set of a polynomial is connected if and only if none of its critical points escapes to infinity. Thus one also has
\begin{equation*}
    \mathcal{C}_\lambda
       = \{a\in \mathbb{C};\text{both critical points of $f_{\lambda,a}$ do not escape to } \infty\}.
\end{equation*}

When $|\lambda|\textless 1$. (resp. $\lambda = e^{\frac{2\pi ip}{q}}$), $0$ is an attracting (resp. parabolic) fixed point. By classical results in holomorphic dynamics, 0 attracts at least one of the two critical points of $f_{\lambda,a}$. For each $\lambda$ fixed, it is natural to consider the following attracting locus:

\begin{equation}\label{eq.attracting_locus}
    \mathcal{H}^{\lambda} := \{a\in\mathbb{C};\,\text{both critical points of }f_{\lambda,a} \text{ are attracted by }0\}.
\end{equation}

\begin{definition*}
    Let $|\lambda|\textless 1$. resp. $\lambda = e^{\frac{2\pi ip}{q}}$. The \textbf{central part} $\mathcal{K}_\lambda$ of $\mathcal{C}_\lambda$ is defined to be the connected component of $\overline{\mathcal{H}_\lambda}$ containing $a=0$. 
\end{definition*}

Our main reult is the following:
\begin{thm*}
Let $\lambda = e^{2\pi i\frac{p}{q}}$. Every connected component of $\mathring{\mathcal{H}^\lambda}$ is a Jordan domain. The central part $\mathcal{K}_\lambda$ is a full continuum and is locally connected. Moreover, there exists a dynamically defined double covering: 
\[\mathfrak{G}:\mathcal{K}_\lambda\setminus\mathcal{I}\longrightarrow K_\lambda\setminus\bigcup^{q-1}_{i=0} \overline{P_i}\] 
where $\mathcal{I}\subset\mathcal{K}_\lambda$ is a curve passing $a=0$; $K_\lambda$ the filled-in Julia set of $P_\lambda(z) = \lambda z+z^2$ and $(P_i)_i$ are $q$ petals contained respectively in the $q$ immediate basins of $P_\lambda(z)$.
\end{thm*}

The Main Theorem can also be stated for $|\lambda|\textless 1$ and is proved in \cite{Roesch1} ($\lambda$ = 0) and \cite{Tanlei} ($0\textless |\lambda|\textless 1$). Our result generalizes theirs to all parabolic slices. In a recent paper by A. Blokh, L. Oversteegen and V. Timorin \cite{blokh}, a global description of $\mathcal{C}_\lambda$ for $|\lambda|\leq 1$ is given by decomposing it into a full continuum $\mathcal{CU}_\lambda$ called "main cubioid" and the limbs attached to it, where they define $\mathcal{CU}_\lambda$ to be the collection of all $f_{\lambda,a}$ verifying the following property:
\begin{itemize}
    \item[~] $f_{\lambda,a}$ has a non-repelling fixed point, $f_{\lambda,a}$ has no repelling periodic cutpoints
in $J_{\lambda,a}$, and all non-repelling periodic points of $f_{\lambda,a}$, except at most one fixed point,
have multiplier 1.
\end{itemize}

As a byproduct of the proof of our Main Theorem, we are able to obtain a more visualisable description of $\mathcal{CU}_\lambda$:
\begin{corx}\label{cor.blokh}
   For $\lambda = e^{2\pi ip/q}$ we have $\mathcal{K}_\lambda = \mathcal{CU}_\lambda$.
\end{corx}

Before illustrating the strategy of the proof, let us say a few more words about $\mathcal{H}^\lambda$. Let $c_1,c_2$ be the two critical points of $f_{\lambda,a}$. Following Milnor \cite{Milnor}, parameters in $\mathcal{H}^{\lambda}$ are divided into four different types: 
\begin{enumerate}[label=(\Alph*)]
\item Adjacent: $c_1,c_2$ are contained in the same Fatou component.
\item Bitransitive: $c_1,c_2$ belong to two different periodic Fatou components in the same cycle.
\item Capture: one of $c_1,c_2$ eventually hits a periodic Fatou component. 
\item Disjoint: $c_1,c_2$ belong to different cycles of periodic Fatou components.
\item[(MP)] Misiurewicz parabolic: one of $c_1,c_2$ eventually hits the parabolic fixed point $z=0$.
\end{enumerate}

\begin{definition*}
Let $\lambda = e^{2\pi i\frac{p}{q}}$. Type (D) parameters are exactly those $a\in\mathcal{C}_\lambda$ such that $z=0$ becomes parabolic degenerate. We also call them \textbf{double parabolic} parameters, and denote by $\mathcal{A}_{p/q}$ the collection of them. A connected component of Type (A) (B) or (C) is usually called a \textbf{parabolic component}.
\end{definition*}

\paragraph{Strategy of the proof.} Let $\lambda = e^{2\pi i\frac{p}{q}}$. To prove the Main Theorem, it is natural to begin with building an identification between $\mathring{\mathcal{K}_\lambda}$ and $\mathring{K_{\lambda}}$: intuitively, Type (A) and (B) corresponds to the immediate basins of $P_{\lambda}$; Type (C) in $\mathring{\mathcal{K}_\lambda}$ corresponds to infinite towers of preimages of the immediate basins attached on their boundaries (see Figure \ref{fig.firstpicture} to have a global picture of $\mathcal{K}_\lambda$ in mind). This can be done by giving dynamical parametrisations to $\mathring{\mathcal{H}_\lambda}$.
The hard part of the Main Theorem is to extend this identification continuously to $\partial\mathcal{K}_{\lambda}$, where the need of local connectivity naturally arises. The strategy for the proof of local connectivity of $\partial\mathcal{K}_\lambda$ is to use the "dynamical-parameter puzzle" technique: transfer the shrink property of dynamical puzzles to that of parameter puzzles. However, comparing to the super-attracting case (\cite{Faught}, \cite{Roesch1}) where this technique is applied, there are several difficulties that we need to overcome:
\begin{itemize}
    \item A priori, the existence of Type (A) (B) (D) parameters with given combinatorics (namely, the dynamical gap between two critical points) is not obvious. Based on Cui-Tan's pinching deformation theory \cite{CuiTan} (Appendix \ref{subsec.appendix_pinching}), we are able to fully solve this question. 
 
    \item The parametrisation of Type (A) (B) components is more complicate. We need to deal with $f_{\lambda,a}$ such that the boundary of its maximal Fatou petal contains two critical points of the return map $f^q_{\lambda,a}$, since at such $a$, the parametrisation by locating the free critical value fails. We will show that the collection of such parameters in each Type (A) (B) components is a curve with end points at Type (D) parameters. Then the parametrisation will work in the complement of this curve.
    \item $\mathcal{C}_\lambda$ can be viewed as a "pinched model" of $\mathcal{C}_0$, where the pinched points are exactly Type (D) (MP) parameters. At such parameters there are no longer nested puzzles surrounding it and thus Grötzsch inequality can not be applied. Instead, we will use an argument based on holomorphic motion to conclude local connectivity. See \ref{subsec.loc.connec.Misiurewicz-parabolic}. Moreover, in order to construct puzzle pieces adapted to Type (D) parameters, we need to investigate the landing external rays at them, which is a rather subtle subject, see \ref{subsec.landing_double_parabolic}.
\end{itemize}

\begin{figure}[H] 
\centering 
\includegraphics[width=0.75\textwidth]{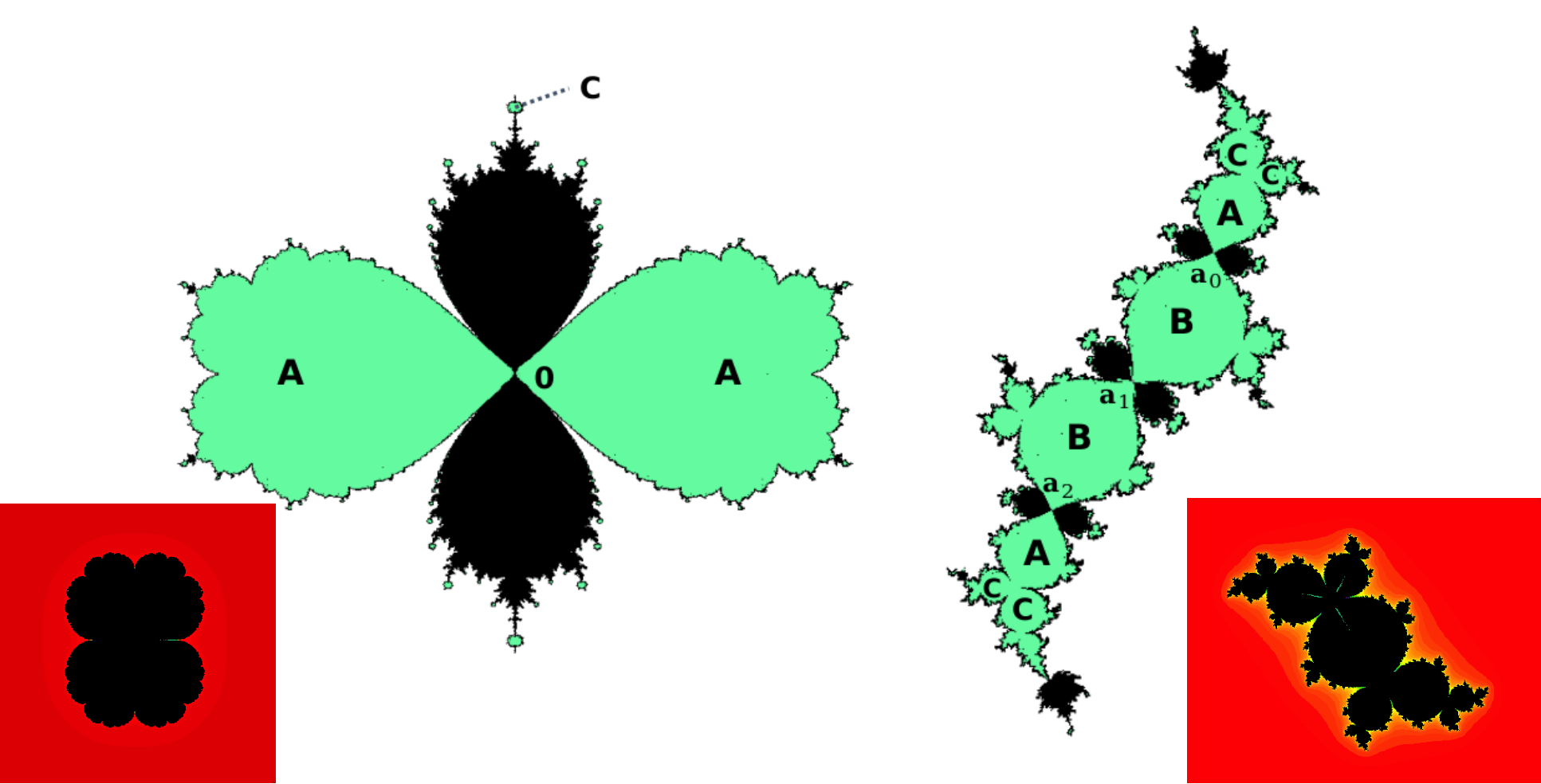}
\caption{$\mathcal{C}_{\lambda}$ and the Julia set of $\lambda z+z^2$, with $\lambda  = 1$ on the left and $e^{\frac{2\pi i}{3}}$ on the right. Different types of parameters in $\mathcal{H}^\lambda$ are marked out. Type (A) (B) and (C) are in green. $0, \boldsymbol{\mathrm{a}}_0,\boldsymbol{\mathrm{a}}_1,\boldsymbol{\mathrm{a}}_2$ are of Type (D). Black parts are copies of the Mandelbrot set or the parabolic Mandelbrot set, correspond respectively to polynomial-like or parabolic-like maps \cite{lomonaco}.} 
\label{fig.firstpicture} 
\end{figure}

\paragraph{Outline of the paper.}
In Section \ref{sec.prepa} we recall some known results of parameter external rays, and the description of $\mathcal{C}_0$ given by \cite{Roesch1}. In Section \ref{sec.prop1_and_prop2}, we first prove the existence of Type (A) (B) components in \ref{subsec.existence.compo} and prove that there are exactly $q$ Type (D) parameters  in \ref{subsec.double_parabolic}. Then we parametrize Type (A) (B) components in \ref{subsec.parame.component} and prove their uniqueness (Corollary \ref{cor.unique.B_m}). We describe relative positions among Type (A) (B) (D) parameters and investigate landing parameter external rays at Type (D) parameters. (Proposition \ref{Prop.grand1} and \ref{Prop.grand2}). These results have their own interests and will simplify the construction of parameter puzzles later. Section \ref{sec.dym.graph} focus on the construction of admissible dynamical graphs, so that one can apply Yoccoz's Theorem to obtain shrinking puzzle pieces. In Section \ref{sec.paragraph}, we construct parameter puzzles and show that dynamical graphs moves holomorphically when the parameter varies in parameter puzzles. Finally Section \ref{sec.loc} is devoted to the proof of the Main Theorem and Corollary \ref{cor.blokh}.
\paragraph{Acknowledgements.}
I would like to thank Pascale Roesch for comments on the the manuscript. The pictures are generated by It of Mannes Technology and a personnal computer program of Arnaud Chéritat.

\section{Preparations}\label{sec.prepa}

\subsection{Families with marked out critical points}
Consider the following two families
\begin{equation}\label{eq.para.c}
    g_{\lambda,c}(z) = \lambda z\left(1-\frac{1+1/c}{2}z + \frac{1/c}{3}z^2\right),\,\,(\lambda,c)\in (\mathbb{C}^*)^2
\end{equation}
\begin{equation}\label{eq.para.s}
    \hat{g}_{\lambda,s}(z) = \lambda z\left(1-\frac{s+1/s}{2}z + \frac{1}{3}z^2\right),\,\,(\lambda,s)\in (\mathbb{C}^*)^2
\end{equation}
The advantage of $g_{\lambda,c}$ (or $\hat{g}_{\lambda,s}$) is that the two critical points $1,c$ (resp. $s,1/s$) are marked out. Fix $\lambda$-slice, denote by $\check{\mathcal{C}}_\lambda$ $\hat{\mathcal{C}}_\lambda$ respectively the connected locus for these two families. The attracting locus $\check{\mathcal{H}}^\lambda,\hat{\mathcal{H}}^\lambda$ are defined likewise as in (\ref{eq.attracting_locus}). For $\lambda\not=0$, the relation between $f_{\lambda,a}$, $g_{\lambda,c}$ and $\hat{g}_{\lambda,s}$ is given by 
\begin{lemma}\label{lem.relation}
$\hat{g}_{\lambda,s}$ is conjugate to $f_{\lambda,a}$ by $z\mapsto \sqrt{3/\lambda}\cdot z$ with $a = \sigma(s) = - \sqrt{{3}\lambda}\cdot\frac{s+1/s}{2}$ and is conjugate to $g_{\lambda,c}$ by $z\mapsto \frac{1}{s}\cdot z$ with $c = \iota(s) = s^2$.
\end{lemma}

\subsection{Parameter external rays in $Per_1(\lambda)$, $|\lambda|\leq 1$}

Let $f$ be a polynomial with connected Julia set and $z$ a (pre-)periodic point. It is a classical result by Yoccoz that $z$ admits an external ray landing (cf. \cite{DH}). The collection of all the angles of external rays landing at $z$ is called the \textbf{portrait} at $z$. For family (\ref{eq.family_a}), the Böttcher coordinate $\phi^\infty_{\lambda,a}$ at $\infty$ depends analytically on $(\lambda,a)$ by taking the normalization $\phi^\infty_{\lambda,a}(z) = z+{o}(1)$. So there is no ambiguity of angles when $(\lambda,a)$ varies.

Now suppose $|\lambda|\leq 1$. It is a classical result (cf. \cite{Zakeri1999DynamicsOC}) that $\mathcal{C}_\lambda$ is a full continuum containing $\pm\sqrt{3\lambda}$. Thus one can define analytically the two critical points of $f_{a,\lambda}$ for $a\in\mathbb{C}\setminus\mathcal{C}_\lambda$: $c^\pm_{\lambda,a} = \frac{-a\pm\sqrt{a^2-3\lambda}}{3}$ such that $c^+_{\lambda,a}\in K_{\lambda,a}$ and $c^-_{\lambda,a}\in\mathbb{C}\setminus K_{\lambda,a}$. Set $v^{\pm}_{\lambda,a} = f_a(c_{\pm}(a))$. One can parametrize $\mathbb{C}\setminus\mathcal{C}_\lambda$ by looking at the position of $v^{-}_{\lambda,a}$ in the Böttcher coordinate: 
\begin{proposition}[\cite{Zakeri1999DynamicsOC}]\label{prop.parametri.H_infini}
Let $|\lambda|\leq 1$. The mapping $\Phi^\lambda_{\infty}:\mathbb{C}\setminus{\mathcal{C}_{\lambda}}\longrightarrow\mathbb{C}\setminus\mathbb{\overline{D}}$ defined by $a\mapsto \phi^{\infty}_{a}(v^{-}_{\lambda,a})$ is a degree 3 covering.
\end{proposition}

A parameter external ray with angle $t$ is defined to be a connected component of $(\Phi^\lambda_{\infty})^{-1}(\{re^{2\pi it};\,r\textgreater 1\})$. In most cases, we denote a parameter external ray by $\mathcal{R}^\lambda_\infty$ without precising the component. When we say "$\mathcal{R}^\lambda_\infty$ lands at $a_0\in\mathcal{C}_\lambda$", it means that one of the three components of $\Phi^\lambda_{\infty})^{-1}(\{re^{2\pi it};\,r\textgreater 1\})$ accumulates at $a_0$. We use the notation $\mathcal{R}^\lambda_\infty(t)*(r_1,r_2)$ to represent the set $(\Phi^\lambda_{\infty})^{-1}(\{re^{2\pi it};\,r\in(r_1,r_2)\})$.

\begin{remark}
Let us just mention that in \cite{Zakeri1999DynamicsOC} the proposition above is stated only for $\lambda = e^{2\pi i\theta}$ with $\theta$ of Brjuno type. However the proof there works without any change for all $\lambda\in\overline{\mathbb{D}}\setminus\{0\}$.
\end{remark}

Proposition \ref{prop.zakeri.parametri} can be passed to family (\ref{eq.para.s}) by Lemma \ref{lem.relation}: 

\begin{proposition}\label{prop.zakeri.parametri}
For $\lambda\in\overline{\mathbb{D}}\setminus\{0\}$, $\mathbb{C}^*\setminus\hat{{\mathcal{C}}}_{\lambda}$ has exactly two connected components ($\tau:s\mapsto1/s$)
\[\hat{\mathcal{H}}_{\infty}=\{s\in\mathbb{C}^*;\,\hat{g}^n_{\lambda,s}(s)\to\infty\text{ as }n\to\infty\}\]
\[\tau\hat{\mathcal{H}}_{\infty} = \{s\in\mathbb{C}^*;\,\hat{g}^n_{\lambda,s}(1/s)\to\infty\text{ as }n\to\infty\}\]
which are punctured neighorhoods of $\infty,0$ respectively and are homeomorphic to punctured disk. Moreover the mapping $\hat{\Phi}_{\infty}:\hat{\mathcal{H}}_{\infty,0}\longrightarrow\mathbb{C}\setminus\overline{D}$ given by $\hat{\Phi}_{\infty}(s) = \hat{\phi}^{\infty}_s(\hat{g}_{\lambda,s}(s))$ is a degree 3 covering, where $\hat{\phi}^{\infty}_{s}(z) = \phi^\infty_{\sigma(s)}(\frac{\lambda}{3}\cdot z)$.
\end{proposition}

By Lemma \ref{lem.relation} we can write $\mathbb{C}^*\setminus \check{\mathcal{C}}_{\lambda} = \check{\mathcal{H}}^\lambda_{\infty}\cup\tau\check{\mathcal{H}}^\lambda_\infty$, where $\tau:c\mapsto \frac{1}{c}$ and
\[ \check{\mathcal{H}}^\lambda_{\infty} = \{c\in\mathbb{C}^*;\,g^n_{\lambda,c}(c)\to\infty\text{ as }n\to\infty\}\]
\[ \tau\check{\mathcal{H}}^\lambda_{\infty} = \{c\in\mathbb{C}^*;\,g^n_{\lambda,c}(1)\to\infty\text{ as }n\to\infty\}\]
by noticing that $cg_{\lambda,c}(\frac{z}{c}) = g_{\lambda,1/c}$. These two escaping regions are simply connected.

\begin{lemma}\label{lem.boundary.components}
Let $\check{\mathcal{U}}\subset\check{\mathcal{H}}^{\lambda}$ be an open component. Then $\check{\mathcal{U}}$ is simply connected and $\partial\check{\mathcal{U}}\subset\partial\check{\mathcal{C}}_{\lambda}$. In particular, if $\partial\check{\mathcal{U}}$ is locally connected, then it is a Jordan curve.
 \end{lemma}
 \begin{proof}
 Simply connectivity can be shown easily by applying maximum principle to $g^n_{\lambda,c}(c)$ or $g^n_{\lambda,c}(1)$; for $\partial\check{\mathcal{U}}\subset\partial\check{\mathcal{C}}_{\lambda}$, see \cite[Lem. 2.1.5]{runze}. We prove the last statement. Suppose $\partial\check{\mathcal{U}}$ is locally connected, then any conformal representation $\Psi:\mathbb{D}\longrightarrow\check{\mathcal{U}}$ can be extended continuously and surjectively to the boundary. Moreover $\Psi:\partial\mathbb{D}\longrightarrow\partial\check{\mathcal{U}}$ is injective: if not, then there exists $a\in\partial\check{\mathcal{U}}$ accessible by two rays $\Psi(re^{it_1}),\Psi(re^{it_2})$ which bound a simply connected region containing part of $\partial\check{\mathcal{U}}$. This contradicts $\partial\check{\mathcal{U}}\subset\check{\mathcal{C}}_\lambda$
 \end{proof}

\begin{definition}
Let $|\lambda|\leq 1$. $a_0\in\mathcal{C}_\lambda$ is a \textbf{Misiurewicz parameter} if one of the critical points is repelling pre-periodic. For $\lambda = e^{\frac{2\pi ip}{q}}$, $a_0\in\mathcal{C}_\lambda$ is a \textbf{Misiurewicz parabolic parameter} if one of the critical points is pre-periodic to $z=0$. The same notions are defined for $g_{\lambda,c}$ and $\hat{g}_{\lambda,s}$.
\end{definition}

\begin{lemma}\label{lem.rational_rays_property}
Let $|\lambda|\leq 1$, $t\in\mathbb{Q}/\mathbb{Z}$. Then $\mathcal{R}^\lambda_\infty(t)$ land at some $a_0\in\mathcal{C}_\lambda$ which is geometrically finite.
\end{lemma}
\begin{proof}
It suffices to prove the accumulation set of $\mathcal{R}^\lambda_\infty(t)$ is finite, since the accumulation set is connected. It will be convenient to work in family (\ref{eq.para.s}) since the critical points are marked out and by Proposition \ref{prop.zakeri.parametri}, external rays with angle $t$, denoted by $\hat{\mathcal{R}}_\infty^\lambda(t)$, in $\hat{\mathcal{H}}_\infty$ are well-defined. Suppose $s_0$ is accumulated by $\hat{\mathcal{R}}_\infty^\lambda(t)$. Moreover if $\lambda = e^{2\pi i\frac{p}{q}}$ suppose $s_0$ is not double parabolic (it is not hard to see that $\#\mathcal{A}_{p/q}\textless\infty$). Suppose $\hat{R}^\infty_{\lambda,s_0}(t)$ (the dynamical external ray of $\hat{g}_{\lambda,s_0}$) lands at $x_{\lambda,s_0}$. 
\begin{itemize}
    \item If $x_{\lambda,s_0}$ is repelling pre-periodic or if $\lambda = e^{\frac{2\pi ip}{q}}$ and $x_{\lambda,s_0}$ is pre-periodic to $z=0$, then by stability of repelling Koenigs coordinate or repelling Fatou coordinate one concludes that $x_{\lambda,s_0} = \hat{v}(s_0) := \hat{g}_{\lambda,s_0}(s_0)$. Therefore $s_0$ satisfies a non-trivial algebraic equation: $\hat{g}^{lk}(\hat{v}(s_0)) = \hat{g}^l(\hat{v}(s_0))$ which has only finitely many solutions.
    \item If $x_{\lambda,s_0}$ is parabolic pre-periodic and when $\lambda = e^{\frac{2\pi ip}{q}}$ it is not pre-periodic to $0$. Then there exists $l\geq 0,k\geq 1$ (only depend on $t$) such that $\hat{g}_{\lambda,s_0}^l(\hat{R}^\infty_{\lambda,s_0}(t))$ is fixed by $\hat{g}_{\lambda,s_0}^k$. By the snail lemma, $(s_0,x_{\lambda,s_0})$ verifies \begin{equation}\label{eq.parabolic1}
        \hat{g}_{\lambda,s}^{k+l}(z) = \hat{g}_{\lambda,s}^l(z),\,\,(\hat{g}_{\lambda,s}^k)'(\hat{g}_{\lambda,s}^l(z)) = 1.
    \end{equation} 
     (\ref{eq.parabolic1}) defines a non-trivial algebraic variety (not equal to $\mathbb{C}^*\times\mathbb{C}$) and hence consists of only finitely many irreducible components of dimension 1 and 0 (i.e. points). We claim that when $\lambda\not= e^{\frac{2\pi ip}{q}}$, there are no dimension 1 component; when $\lambda\not = e^{\frac{2\pi ip}{q}}$, the only dimension 1 component is $z = 0$. Indeed, suppose there is another such component $X$. Then $X$ is unbounded (goes to the boundary of $\mathbb{C}^*\times\mathbb{C}$). Consider the projections $\pi_1: X\longrightarrow \mathbb{C}$, $\pi(s,z) = s$ and $\pi_2: X\longrightarrow \mathbb{C}$, $\pi(s,z) = z$. Then at least one of $\pi_i(X),i=1,2$ is unbounded. If $\pi_1(X)$ is unbounded, let $(s_n,z_n)\subset X$ with $s_n\to \infty$ or 0, then $\lambda$ has to be $e^{\frac{2\pi ip}{q}}$ and for $n$ large enough, $z_n = 0$, so $X$ has to be $z = 0$. If $\pi_2(X)$ is unbounded, let $(s_n,z_n)\subset X$ with $z_n\to \infty$. If $(s_n)$ also tends to $\infty$ or $0$, then we are in the precedent case; if not, then for all $n$, the basin at infinity of $\hat{g}_{\lambda,s_0}$ contains a common neighborhood of $\infty$. Hence for $n$ large enough, $z_n$ escapes to $\infty$, contradicting the assumption that $(s_n,z_n)$ verifies (\ref{eq.parabolic1}).
\end{itemize}
To conclude, the above analysis shows that the accumulation set of $\mathcal{R}_0(t)$ is finite and the possible accumulations are geometrically finite maps
\end{proof}

For our purpose, we extract from the lemma above the case of $\lambda = e^{2\pi i\frac{p}{q}}$:

\begin{lemma}\label{lem.landing.rational.external}
Let $\lambda = e^{2\pi i\frac{p}{q}}$, $t\in\mathbb{Q}/\mathbb{Z}$. Then $\mathcal{R}^\lambda_\infty(t)$ lands at some $a_0$. In the dynamical plan of $f_{a_0}$, $R^\infty_{\lambda,a_0}(t)$ lands at a (pre-)periodic point $x(a_0)$. Moreover
\begin{itemize}
    \item If $x(a_0)$ is repelling, then $x(a_0)$ is the free critical value. 
    \item If $a_0\not\in\mathcal{A}_{p/q}$, and $x(a_0)$ is in the inverse orbit of $0$, then $x(a_0)$ is the free critical value. 
    \item If $a_0 = \boldsymbol{\mathrm{a}}_m\in\mathcal{A}_{p/q}$ and $t$ is not pre-periodic to the portrait at $z=0$ of $f_{\lambda,a_0}$, then $\mathcal{R}^\lambda_\infty(t)$ does not land at $a_0$. 
\end{itemize}
\end{lemma}

\begin{lemma}\label{lem.landing.Misiurewicz}
Suppose $|\lambda|\leq 1$. Let $a_0$ be a Misiurewicz parameter or parabolic Misiurewicz parameter, $t\in\mathbb{Q}/\mathbb{Z}$. Then $R^{\infty}_{\lambda,a_0}(t)$ lands at one of the critical values of $f_{\lambda,a}$ if and only if $\mathcal{R}^\lambda_{\infty}(t)$ lands at $a_0$.
\end{lemma}
\begin{proof}
First of all $a_0 \not=\pm\sqrt{3\lambda}$, i.e. the two critical points are not the same. This is clear when $|\lambda|\textless1$ or $\lambda = e^{\frac{2\pi ip}{q}}$ since $0$ attracts at least one critical point. If $\lambda$ is Siegel, then the boundary of the Siegel disk is contained in the accumulation set of critical orbits; if $\lambda$ is Cremer, then the Julia set is not locally connected. Therefore $c^\pm_{\lambda,a}$ can be analytically defined near $a_0$. Then one uses stability of repelling Koenigs coordinate or repelling Fatou coordinate to conclude the proof. 
\end{proof}

\begin{lemma}\label{lem.landing.stable-region}
Let $t\in\mathbb{Q}/\mathbb{Z}$. 
\begin{itemize}
    \item Suppose $|\lambda|\leq 1$ but $\lambda \not= e^{2\pi i\frac{p}{q}}$. Let $\mathcal{O}\subset\mathbb{C}$ be the set of parameters $a$ such that $R^\infty_{\lambda,a}(t)$ lands at a repelling (pre-)periodic point, Then $\mathbb{C}\setminus\bigcup_{k\geq 1}\overline{\mathcal{R}_\infty^\lambda(3^kt)}\subset\mathcal{O}$. Moreover $\bigcup_{k\geq 0}\overline{R^\infty_{\lambda,a}(3^kt)}$ admits a dynamical holomorphic motion for $a$ in any component of $\mathcal{O}$.
    
    \item Suppose $\lambda = e^{2\pi i\frac{p}{q}}$. Let $\mathcal{O}\subset\mathbb{C}$ be the set of parameters $a$ such that $R^\infty_{\lambda,a}(t)$ lands at a repelling (pre-)periodic point or the inverse orbit of 0. Then $(\mathbb{C}\setminus\bigcup_{k\geq 1}\overline{\mathcal{R}_\infty^\lambda(3^kt)})\setminus\mathcal{A}_{p/q}\subset\mathcal{O}$. Moreover $\bigcup_{k\geq 0}\overline{R^\infty_{\lambda,a}(3^kt)}$ admits a dynamical holomorphic motion for $a$ in any component of $\mathcal{O}\setminus\mathcal{A}_{p/q}$ ($\mathcal{A}_{p/q}$ is the set of double parabolic parameters).
\end{itemize}
\end{lemma}
\begin{proof}
For the first point, see \cite[Lem. 3.8]{Roesch1}. For the second point, for $a\in\mathbb{C}\setminus\bigcup_{k\geq 1}\overline{\mathcal{R}_\infty^\lambda(3^kt)}$, ${R^\infty_{\lambda,a}(3^kt)}$ land for $k\geq 0$ since they do not crash on the free critical point. It suffices to prove that the landing point $x(a)$ of $R^\infty_{\lambda,a}(t)$ is either repelling pre-periodic or in the inverse orbit of $0$. Suppose the contrary that it is parabolic (pre-)periodic beyond $z=0$, then the portrait at $z=0$ of $f_{\lambda,a}$ does not intersect $\{3^kt\}_{k\geq 0}$. Since the portrait at $z=0$ is stable by the stability of repelling Fatou coordinate at $z=0$ (notice that $a\not\in\mathcal{A}_{p/q}$), so for $a'$ in a neighborhood of $a$, the portrait at $z=0$ of $f_{\lambda,a'}$ does not intersect $\{3^kt\}_{k\geq 0}$. However on the other hand, from the second point in the proof of Lemma \ref{lem.rational_rays_property}, the set of parabolic parameters with a given period is discrete. So for $a'$ in a punctured neighborhood of $a$, $x(a')$ is either repelling or sent to $z=0$ with portrait containing the cycle in $\{3^kt\}_{k\geq 0}$. So the second case is impossible. The first case contradicts the maximum principle for $a\mapsto 1/(f^n)'({x(a)})$ for $n$ large enough.
\end{proof}

\subsection{Results for $\mathcal{C}_0$}
When $\lambda = 0$, $z=0$ is a supper-attracting fixed point of $f_{0,a}$, $-\frac{2a}{3}$ is the free critical point. $\mathcal{C}_0$ has a unique hyperbolic component $\mathcal{H}^0_0$ containing $a = 0$. (The upper index means $\lambda = 0$, the lower means it is of depth 0, i.e. $\frac{-2a}{3}$ is in the immediate basin of 0). Similarly as $\mathbb{C}\setminus\mathcal{C}_0$, one can parametrize $\mathcal{H}_0^0$:

\begin{proposition}[\cite{Roesch1}]\label{prop.parametrisation.H00}
The mapping $\Phi^0_0:\mathcal{H}^0_0\longrightarrow\mathbb{D}$ given by $\Phi^0_0(a) = \phi^{0}_{0,a}(f_{0,a}(-\frac{2a}{3}))$ is a degree 4 holomorphic covering ramified at 0. 
\end{proposition}

The internal ray of angle $t$ in $\mathcal{H}^0_0$ is defined to be a connected component of \[(\Phi^0_0)^{-1}(\{re^{2\pi i t};\,r\in(0,1)\}).\]

\begin{remark}
It is easy to verify by symmetric that \[\Phi^0_0(-a) = \Phi^0_0(a), \Phi^0_0(\overline{a}) = \overline{\Phi^0_0({a})};\,\, \Phi^0_{\infty}(-a) = -\Phi^0_{\infty}(a), \Phi^0_{\infty}(\overline{a}) = \overline{\Phi^0_{\infty}({a})}\]
\end{remark}

\begin{remark}
By Proposition \ref{prop.parametrisation.H00} and \ref{prop.parametri.H_infini}, there are 4 (resp. 3) internal (resp. external) rays associated to a given angle. Let ${\mathcal{S}} = \{x+iy,x\geq0,y\textgreater0\}$. These four rays are contained in four sectors $\mathcal{S},i\mathcal{S},-\mathcal{S},-i\mathcal{S}$ respectively. Hence we precise a ray contained in a sector by adding this sector in the above index, for example $\mathcal{R}^{0,\mathcal{S}}_0(t)$ (resp. $\mathcal{R}^{0,\mathcal{S}}_\infty(t)$) denotes the internal (resp. external) ray with angle $t$ contained in $\mathcal{S}$. Nevertheless in most time we omit for simplicity this index if the sector to which this ray belong is clear or if there is no need to precise, according to the context.
\end{remark}

The main results in \cite{Roesch1} can be summarized as the following proposition and theorem:

\begin{proposition}\label{prop.pascale.landingH0}\cite[Lem 2.29]{Roesch1}.
For all $t\in\mathbb{R}/\mathbb{Z}$, $\mathcal{R}^{0,\mathcal{S}}_0(t)$ lands at some $a(t)\in \partial\mathcal{H}^0_0$. In the dynamical plan, $R^0_{a(t)}(t)$ lands at a parabolic $k$-periodic point if and only if $\frac{t+1}{2}$ is $k$-periodic under multiplication by 2.
\end{proposition}

\begin{theorem}\label{thm.pascale.descriptionC0}
For all $n\geq 0$, $\partial\mathcal{H}^0_n$ is a union of Jordan curve. If $a_0\in\partial\mathcal{H}^0_n$ is renormalizable, then there are exactly two external rays landing at it; otherwise there is only one. 

If $a_0\in\partial\mathcal{H}^0_0$ is renormalizable, then the two corresponding rays landing at it separate $\mathbb{C}$ into two connected components $U_{a_0}$ and $W_{a_0}$ with $\mathcal{H}^0_0\subset U_{a_0}$ and $\overline{\mathcal{W}_{a_0}}$ containing a copy of Mandelbrot set rooted at $a_0$. For $a\in\partial\mathcal{H}^0_0$, define limbs by $\mathcal{L}_{a}:=\mathcal{W}_{a}\cap\mathcal{C}_0$ if $a$ is renormalisable; otherwise $\mathcal{L}_{a}:=\{a\}$. Then $\mathcal{C}_0 = \mathcal{H}^0_0\cup(\bigcup\limits_{a\in\partial\mathcal{H}^0_0}\mathcal{L}_{a})$.
\end{theorem}

Let us give three lemmas which will be used later in the proof of Proposition \ref{prop.landing.four}:
\begin{lemma}\label{lem.thurston.rigid}
Let $a_1,a_2\in\partial\mathcal{H}^0_0$ be two different renormalizable parameters. By Theorem \ref{thm.pascale.descriptionC0}, suppose that $\mathcal{R}^0_{0,\infty}(t_i),\mathcal{R}^0_{0,\infty}(t'_i)$ land at $a_i$, $i=1,2$. Then $\{t_1,t'_1\}\not=\{t_2,t'_2\}$.
\end{lemma}
\begin{proof}
Suppose the contrary that $\{t_1,t'_1\}=\{t_2,t'_2\}$. Then the portrait of the co-critical point $\tilde{c}_{a_i}$ will be the same, i.e. the angles of the two external rays bounding $\tilde{c}_{a_1}$ equal those for $\tilde{c}_{a_1}$. On the other hand, since the parametrisation of $\mathcal{H}^0_\infty$ defined by $a\mapsto\phi^\infty_{0,a}(\tilde{c}_a)$ is an isomorphism, we conclude that the angles of parameter rays landing at $a_i$ are different. Thus in the dynamical plan, the portrait at $\tilde{c}_{a_i}$ should also be different, a contradiction.
\end{proof}

\begin{lemma}\label{lem.landing.boundary.H00}
There are exactly $4q$ external rays whose angles are of rotation number $p/q$ (under multiplication by 3) landing at $\partial\mathcal{H}^0_0$.
\end{lemma}
\begin{proof}
By Theorem \ref{thm.goldberg}, the periodic cycle of rotation number $p/q$ under multiplication by 2 is unique. Denote this cycle by $T_{\frac{p}{q}}$. 
\begin{claim*}
Let $t\in T_{\frac{p}{q}}$ such that $\frac{t+1}{2}$ is \textbf{not} $q$-periodic. Then $1-t\in T_{1-\frac{p}{q}}$ and $\frac{(1-t)+1}{2}$ is $q-$periodic.
\end{claim*}
\begin{proof}[proof of the claim]
Let $t = \frac{m}{2^q-1}$. By hypothesis $2^q\cdot\frac{t+1}{2} \not= \frac{t+1}{2}$, which by an elementary computation is equivalent to $\frac{m}{2} \not= \frac{1}{2}$, i.e. $m$ is even. Hence $1-t =  \frac{\Tilde{m}}{2^q-1}$ where $\Tilde{m} = 2^q-1-m$ is odd. Hence $\frac{(1-t)+1}{2}$ is $q-$periodic.
\end{proof}
 Let $a(t)$ be the landing point of $\mathcal{R}^{0,\mathcal{S}}_0(t)$. Suppose there are $k$ element in $T_{\frac{p}{q}}$ such that $R^0_{a(t)}(t)$ lands at a parabolic periodic point, then from Proposition \ref{prop.pascale.landingH0} and the claim we see that there are $q-k$ elements in $T_{\frac{p}{q}}$ such that $R^0_{a(1-t)}(1-t)$ lands at a parabolic periodic point. Moreover, if we name these $k$ elements and $q-k$ elements by $t_1,...,t_k$, $\Tilde{t}_1,...,\Tilde{t}_{q-k}$, the claim tells us that $T_{p/q} = \{t_1,...,t_k,1-\Tilde{t}_1,...,1-\Tilde{t}_{q-k}\}$. Denote by $a(t_i)$ the landing point of $\mathcal{R}^{0,\mathcal{S}}_0(t_i)$ for $1\leq i\leq k$ and $a(\Tilde{t}_i)$ the landing point of $\mathcal{R}^{0,-i\mathcal{S}}_0(\Tilde{t}_i)$ for $1\leq i\leq q-k$. In the dynamical plan, let $(x_i)_{1\leq i\leq k}$ and $(\Tilde{x}_i)_{1\leq i\leq q-k}$ be the landing point of $R^0_{a(t_i)}(t_i)$ and $R^0_{{a}(\Tilde{t}_i)}(\Tilde{t}_i))$ respectively. Then $({x}_i),(\Tilde{x}_i)$ are all parabolic periodic points with rotation number $p/q$. By Theorem \ref{thm.pascale.descriptionC0}, for any $t\in T_{p/q}$, there are  exactly two parameter external rays $\mathcal{R}_{\infty}(\eta),\mathcal{R}_{\infty}(\eta')$ landing at $a(t)$ with $\eta,\eta'\in\bigcup_k\Theta_k$. Therefore we have found $2q$ external rays landing at parameters on $\partial\mathcal{H}_0^0$ in the right-half plane. By symmetric, there are in total $4q$ rays.
\end{proof}

\begin{lemma}\label{lem.landing.mandelbrot.Per0}
Suppose $t\in\mathbb{Q}/\mathbb{Z}$ has rotation number $p/q$ under multiplication by 3. Suppose that $a_0$ the landing point of $\mathcal{R}^{0}_{\infty}(t)$ is a parabolic parameter, i.e. $R^{\infty}_{0,a_0}(t)$ lands at a parabolic periodic point. Then there are exactly two external rays $\mathcal{R}^{0}_{\infty}(\theta^+),\mathcal{R}^{0}_{\infty}(\theta^-)$ landing at $a_0$ ($t$ is one of $\theta^+,\theta^-$). Moreover, $R^\infty_{0,a_0}(\theta^+),R^\infty_{0,a_0}(\theta^-)$ bound the critical value $v_{0,a_0}$, separating it from the immediate basin of 0.
\end{lemma}
\begin{proof}
Let $\mathcal{P}^{0}_n({a_0})$ be the puzzle piece ($a_0\in\mathcal{P}^{0}_n({a_0})$) constructed in \cite{Roesch1}. Since $a_0$ is renormalisable, by \cite[Prop. 3.26]{Roesch1}, $\textbf{M}_{a_0} := \bigcap_n\overline{\mathcal{P}^{0}_n({a_0})}$ is a copy of Mandelbrot set and let $\chi:\textbf{M}_{a_0}\longrightarrow\textbf{M}$ be the homeomorphism. Let $\theta^-_n$ (resp. $\theta^+_n$) be the smallest/largest angle of the external rays involved in $\partial\mathcal{P}^{0}_n({a_0})$. \\\vspace{-2.5mm}

\noindent\textbf{Step 1}. For $n$ large enough, $3^k\theta^{\pm}_{n} = \theta^{\pm}_{n-k}$.

By \cite[Lem. 3.17, Prop. 3.22]{Roesch1}, for $n$ large enough, there is a natural homeomorphism preserving equipotentials and rays between $\partial\mathcal{P}^{0}_n({a_0})$ and $\partial P^{a_0}_{0,n}$. Thus $\theta^{\pm}_n$ are also the largest/smallest angle involved in $\partial P^{a_0}_{0,n}$. Notice that $\mathcal{R}^0_{\infty}(\theta^{\pm}_n)$ land at the same hyperbolic component since adjacent and capture components have disjoint boundaries (cf. \cite[Lem. 1]{Roesch1}). Thus the two corresponding dynamical rays $R^{\infty}_{0,a_0}(\theta^{\pm}_n)$ land at the boundary of some Fatou component $U$. Let $x^{\pm}$ be their landing points and $y^{\pm}$ their intersection with the equipotential in $\partial P^{a_0}_{0,n}$. Then $y^{+},y^{-}$ are linked by a curve $L$ consisting of two segments of $R^{\infty}_{0,a_0}(\theta^{\pm}_n)$ linking $y^{\pm}$ to $x^{\pm}$ and a segment contained in $\overline{U}\cap\partial P^{a_0}_{0,n}$ linking $x^+,x^-$. Suppose the contrary that $3^k\theta^{\pm}_{n} \not= \theta^{\pm}_{n-k}$. Then there is a segment of equipotential $\gamma$ with angles between $(\theta^-_n,\theta^+_n)$ linking $f^k_{0,a_0}(y^+),f^k_{0,a_0}(y^-)$. Thus $\gamma\cup f^{k}_{0,a_0}(L)$ bounds a simply connected domain $P$ which do not contain the free critical value of $f_{0,a_0}$. Let $Q$ be the component of $f^{-k}_{0,a_0}(P)$ whose boundary contains $L$. Then $f^{k}_{0,a_0}:\overline{Q}\longrightarrow \overline{P}$ is injective. Hence $\partial Q$ contains the segment of equipotential consisting of angles in $\mathbb{S}^1\setminus(\theta^-_n,\theta^+_n)$, which in particular contains the angle 0. While $R^{\infty}_{0,a_0}(0)$ is fixed by $f_{0,a_0}$, this leads to a contradiction.

\noindent\textbf{Step 2}. Two external rays land at $\chi^{-1}(\frac{1}{4})$.

Since $\overline{P^{a_0}_{0,n}}\subset P^{a_0}_{0,n-k}$, we obtain a sequence of decreasing/increasing angles $\{\theta^{\pm}_{n_0+lk}\}_{l\geq 0}$ for some $n_0$ large enough. Let $l$ tend to infinity we get two limits $\theta^{\pm}$ which satisfy $3^k\theta^{\pm} = \theta^{\pm}$ by Step 1. Let $a^{\pm}$ be the landing point of $\mathcal{R}_{0,\infty}(\theta^{\pm})$. Clearly $a^{\pm}\in\textbf{M}_{a_0}$. We claim that $a^{+}=a^{-}=\chi^{-1}(\frac{1}{4})$. We prove for $a^{+}$ and the left one will be the same. If not, then $R^{\infty}_{0,a^+}(\theta^+)$ land at a repelling fixed point for $f^k_{0,a^+}$. By \cite[Lem. 2.24]{Roesch1}, this point is the free critical value of $f_{0,a^+}$. Let $P_{c^+}(z) = z^2+c^+$ be the corresponding quadratic polynomial to which $f^k_{0,a^+}|_{P^{a^+}_{0,n}}$ is conjugate. Then we have $P_{c^+}(c^+) = c^+$, i.e, $c^+ = 0$. While $c^+\in\partial\textbf{M}$, a contradiction. Thus we have proved that $\mathcal{R}^0_{\infty}(\theta^{\pm})$ land at $\chi^{-1}(\frac{1}{4})$. \\\vspace{-2.5mm}

\noindent\textbf{Step 3}. Two external rays land at $\chi^{-1}(c)$ for $c\not=\frac{1}{4}$ a parabolic parameter on the cardioid.

Let $a = \chi^{-1}(c)$ and $R^{\infty}_{0,a}(t),R^{\infty}_{0,a}(3^kt),...,R^{\infty}_{0,a}(3^{(s-1)k}t)$ be the cycle of external rays landing at the parabolic fixed point of $f^k_{0,a}$. Notice that $c$ is satellite, hence it is the intersection of the closure of the main hyperbolic component $\textbf{H}_0$ and the closure of the periodic hyperbolic component $\textbf{H}_1$ attached at $c$. Thus there are two pinching paths $\gamma_0\subset\chi^{-1}(\textbf{H}_0)$ and $\gamma_1\subset\chi^{-1}(\textbf{H}_1)$ converging to   $f_{0,a}$ such that
\begin{enumerate}
    \item if $a'\in\gamma_0$, then $\{R^{\infty}_{0,a'}(3^{kl}t)\}_{l\geq 0}$ land at a repelling $s$-periodic cycle of $f^k_{0,a'}$.
    \item if $a'\in\gamma_1$, then $\{R^{\infty}_{0,a'}(3^{kl}t)\}_{l\geq 0}$ land at a common repelling fixed point of $f^k_{0,a'}$.
\end{enumerate}
Since landing at a repelling periodic cycle is an open property, there must exist two external rays among $\{\mathcal{R}^0_{\infty}(3^{l}t)\}_{l\geq 0}$ landing at $a$.\\\vspace{-2.5mm}

\noindent\textbf{Step 4}. The two rays landing are unique. The proof of \cite[Thm. 3]{Roesch1} can be adapted.\\\vspace{-2.5mm}

\noindent\textbf{Step 5}. $a_0$ is on the cardioid of $\textbf{M}_{a_0}$.

Let $k$ be the period of renormalisation of $a_0$. Since the cycle $t,3t,...,3^{q-1}t$ has rotation number $p/q$ under multiplication by 3, then so does the cycle $t,3^kt,...,3^{sk}t$ under multiplication by $3^k$. By \cite[Prop. 3.22]{Roesch1}, $f^k_{0,a_0}:P^{a_0}_{0,n}\longrightarrow P^{a_0}_{0,n-k}$ is hybrid conjugate to a quadratic polynomial $P_c(z) = z^2+c$ for $n$ large enough, where $P^{a_0}_{0,n}$ is the dynamical puzzle piece containing the free critical value of $f_{0,a_0}$. Thus the cycle of the parabolic periodic point of $P_c$ has a cycle of access, which also admits a rotation number. This cycle of access is homotopic to a cycle of external rays with the same rotation number. Since by Theorem \ref{thm.goldberg}, there is only one cycle of angles for a given rotation number under multiplication by 2, this implies that the parabolic periodic point of $P_c$ is actually fixed, hence $c$ is on the cardioid, so is $a_0$.

From Step 1 and 3 we see that $R^\infty_{0,a_0}(\theta^+),R^\infty_{0,a_0}(\theta^-)$ bound a sector containing $v_{0,a_0}$, separating it from $z=0$.
\end{proof}

\begin{corollary}\label{cor.two_angles}
Under the same hypothesis and notations of Lemma \ref{lem.landing.mandelbrot.Per0}, the interval $(\theta^-,\theta^+)$ does not contain any angle in $\{3^k\theta^+;\,k\geq 0\}\cup\{3^k\theta^-;\,k\geq 0\}$.
\end{corollary}
\begin{proof}
It suffices to prove for the case when $a_0$ is the root of $\textbf{M}_{a_0}$.

First we prove that in the dynamical plan of $f_{0,a_0}$, the wake $W$ bounded by ${R}^\infty_{0,a_0}(\theta^{\pm})$ (that is also the wake containing the critical value $v_{0,a}$) does not contain the critical point $-\frac{2a}{3}$. Let $\tilde{W}$ be the wake bounding $-\frac{2a}{3}$ and bounded by ${R}^\infty_{0,a_0}(\tilde{\theta}^{\pm})$. Suppose the contrary, then we have $\tilde{W}\subset W$. Hence $f^{k}_{0,a}|_{W\setminus\overline{\tilde{W}}}$ is injective and proper, where $k$ is the period of the parabolic cycle. Since ${3^{k-1}\theta}^{\pm} = \tilde{\theta}^{\pm}$, we conclude that  $f^{k}_{0,a}({W\setminus\overline{\tilde{W}}})\subset\tilde{W}$. However on the other hand since $\tilde{W}\subset W$, the wake bounded by ${R}^\infty_{0,a_0}(3^{k-1}\theta^{\pm})$ contains the wake bounded by ${R}^\infty_{0,a_0}(3^{k-1}\tilde{\theta}^{\pm})$. Thus $f_{0,a}(f^{k-1}_{0,a}(W\setminus\overline{\tilde{W}}))$ will not intersect $W$, a contradiction.

Next we prove that $W$ contains no other point in the parabolic periodic cycle. Suppose the contrary, then $W$ contains another wake $W'$ bounded by some ${R}^\infty_{0,a_0}(3^l{\theta}^{\pm})$ with $1\leq l\leq k-2$. Thus $f^l_{0,a_0}(W) = W'$. Apply Denjoy-Wolff theorem we obtain $f^{nl}_{0,a_0}(W)$ converges to $\infty$, a contradiction.
\end{proof}

\section{Description of $\mathcal{H}^\lambda$, $\lambda = e^{\frac{2\pi ip}{q}}$}\label{sec.prop1_and_prop2}
The main goal of this section is to prove the following:
\begin{propx}\label{Prop.grand1}
There are exactly $q$ Type (D) parameters, i.e. $\#\mathcal{A}_{p/q} = q$. The different $\boldsymbol{\mathrm{a}}_m\in\mathcal{A}_{p/q}$ ($0\leq m\leq q-1$) are characterized by the portrait at $z=0$. Moreover each $\boldsymbol{\mathrm{a}}_m$ admits 4 parameter external rays landing, cutting $Per_1(\lambda)$ with respect to the portrait at $z=0$. Moreover, these four rays are unique among all rays with rational angle.
\end{propx}

\begin{propx}\label{Prop.grand2}
$\mathcal{H}^{\lambda}$ has exactly 2 adjacent components $\mathcal{B}_0,\mathcal{B}_q$ symmetric with respect to $a\mapsto-a$ and $q-1$ bitransitive components $\mathcal{B}_1,...,\mathcal{B}_{q-1}$. These components are characterized by the portrait at $z=0$. Moreover, \[\overline{\mathcal{B}_m}\cap{\mathcal{A}_{p/q}} = \boldsymbol{\mathrm{a}}_m\cup\boldsymbol{\mathrm{a}}_{m-1},\,\,1\leq m\leq q-1;\quad 
\overline{\mathcal{B}_0}\cap{\mathcal{A}_{p/q}} = \boldsymbol{\mathrm{a}}_0,\,\, \overline{\mathcal{B}_q}\cap{\mathcal{A}_{p/q}} = \boldsymbol{\mathrm{a}}_{q-1}.\]
\end{propx}

Our construction of $\boldsymbol{\mathrm{a}}_m$ and $\mathcal{B}_m$ is based on the technique of pinching deformation developed by Cui-Tan \cite{CuiTan}. Its advantage is that the combinatorical information is explicitly prescribed. See a different approach using transversality method applied to the same problem for quadratic rational maps \cite{buff}. Notice that the uniqueness of the 4 external rays in Proposition \ref{Prop.grand1} is not obvious: this is the cubic analogue to the ray landing problem at parabolic parameters in the Mandelbrot set, which has been solved by Douady-Hubbard \cite{DH}. See also \cite{lei_2000} for an easier presentation. Essentially, the "Tour de valse" argument should work, but should also be more delicate since the fixed point is parabolic degenerate. Instead, here we use a "ray counting" argument based on pinching deformation to avoid the complicate analysis in \cite{DH}. Let us mention that the landing problem at double parabolic parameters has been solved recently by \cite[Thm 11.7]{blokh}. However our method are different from theirs.

\subsection{Existence of adjacent and bitransitive components}\label{subsec.existence.compo}
 
In this subsection we prove existence of adjacent and bitransitive components satisfying any combinatorics given. The main tool is pinching deformation.
 
For $(\lambda,a)\in(0,1)\times(-\sqrt{3\lambda},\sqrt{3\lambda})$, $z=0$ is an attracting fixed point for $f_{\lambda,a}$. Denote by $\Omega_{\lambda,a}$ the maximal linearization domain at $z=0$. Since $\overline{f_{\lambda,\overline{a}}}=f_{\lambda,a}$, $\Omega_{\lambda,a}$ is symmetric with respect to $x-$axis and therefore $\partial\Omega_{\lambda,a}$ contains both critical points of $f_{\lambda,a}$. Denote by $c^{+}_{\lambda,a}$ (resp. $c^{-}_{\lambda,a}$) the one contained in the upper-half (resp. lower-half) plane. Let $\phi_{\lambda,a}:\overline{\Omega_{\lambda,a}}\longrightarrow\overline{\mathbb{D}}$ be the Koenigs coordinate normalised by $\phi_{\lambda,a}(c^{+}_{\lambda,a}) = 1$. Take a branch of $\log(\cdot)$ such that $\log(1) = 0$. Set $\varphi_{\lambda,a} = -\log(\phi_{\lambda,a})$. Then clearly $\varphi_{\lambda,a}(\partial\Omega_{\lambda,a}) = i(-2\pi,0]$.  Define $I_{\lambda}(a) =\mathfrak{Im}(-\varphi_{\lambda,a}(c^{-}_{\lambda,a}))$. Then by definition $I$ is strictly positive. It is not hard to prove the following lemma by quasiconformal deformation:

 \begin{lemma}\label{lem.descrip.I_lambda}
 For any $x\in (0,2\pi)$ there exists $(\lambda,a)\in(0,1)\times(-\sqrt{3\lambda},\sqrt{3\lambda})$ such that $x=I_{\lambda}(a)$.
 \end{lemma}

\begin{figure}[H] 
\centering 
\includegraphics[width=0.4\textwidth]{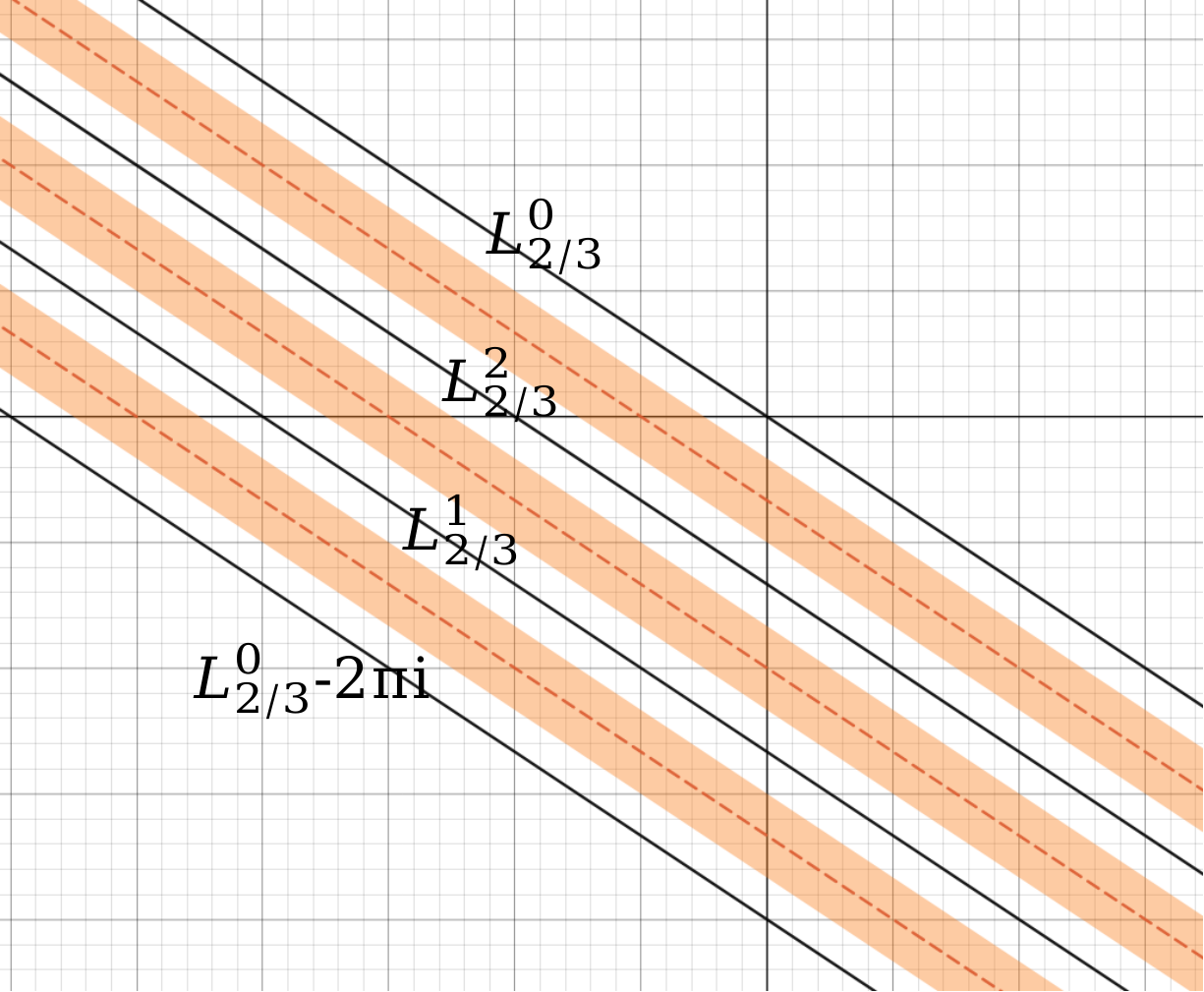} 
\caption{An illustration for $\frac{p}{q}=\frac{2}{3}$. The bands in orange are $\tilde{S}_k$.} 
\label{fig.main.parabolic} 
\end{figure}
 
Let $L^0_{p/q}$ be the line passing 0 and $q\log\lambda+2\pi (q- p) i$. For $k=0,1,...,q-1$, let $L^k_{p/q} = L^0_{p/q}-k\log\lambda\pmod {2\pi i}$ be the line intersecting $i(-2\pi,0]$. The open Strip $S$ bounded by $L^0_{p/q},L^0_{p/q}-2\pi i$ is divided into $q$ sub-strips by $L^k_{p/q}$. Let $S_k$ be the open sub-strip whose upper boundary is $L^k_{p/q}$. Let $\Tilde{L}^k_{p/q}$ be the central line of $S_k$ and suppose that it intersects $i\mathbb{R}$ at $i\Tilde{y}_k$. Notice that $\Tilde{y}_k$ does not depend on $(\lambda,a)$. For any $0\leq m\leq \lfloor\frac{q}{2}\rfloor$, choose $\lambda$ and $a_m\in(-\sqrt{3\lambda},\sqrt{3\lambda})$ such that $-I_{\lambda}(a)\not\in\{\Tilde{y}_0,...,\Tilde{y}_{q-1}\}$ and the interval $(-I_{\lambda}(a),0)$ contains exactly $m$ elements in $\{\Tilde{y}_0,...,\Tilde{y}_{q-1}\}$. Such $\lambda$ and ${a}_m$ necessarily exist by Lemma \ref{lem.descrip.I_lambda}. For each $k$ pick an open strips $\Tilde{S}_k$ centered at $\Tilde{L}^k_{p/q}$ such that  $\forall k,\, \varphi_{\lambda,a_m}(c^\pm_{\lambda,a_m})\not\in\overline{\Tilde{S}_k}$. Then $\{\Tilde{S}_k\}_{0\leq k\leq q-1}$ defines a non-separating multi-annulus $\mathscr{A}$ in the quotient space of $f_{\lambda,a_m}$ by setting 

\[\mathscr{A} = \pi\phi^{-1}_{\lambda,a_m}(\exp(-\bigcup_{k=0}^{q-1}\Tilde{S}_k))\]
which by Theorem \ref{thm.pinching} gives a converging pinching path $\psi_t\circ{f_{\lambda,a_m}}\circ{\psi_t^{-1}}\in Per_1(\lambda_t)$. Here we choose a different normalisation for $\psi_t$ by setting $\psi_t(0) = 0$ and $\psi_t(z) = z+{o}(1)$ at $\infty$. Therefore the resulting limit when $t\to\infty$
\[f_{\Tilde{a}_m}:=\psi_\infty\circ{f_{\lambda,a}}\circ{\psi_\infty^{-1}}\in Per_1(e^{2\pi i p/q})\]
is of adjacent type if $m=0$ and bitransitive type if $0 \textless m\leq \lfloor\frac{q}{2}\rfloor$. 

Now we give a more detailed description for the pinching limit in terms of the portrait at $z=0$. Notice that the $0$-level skeleton for $\mathscr{A}$ is a $q$-cycle of rays with rotation number $p/q$ starting from 0 and landing at $\partial B^*_{\lambda,a}(0)$. Denote by $\{x_i(a_m)\}$ this corresponding landing cycle of points. Then the period of $\{x_i(a_m)\}$ is $q$ since the cycle of external rays landing at $\{x_i(a_m)\}$ has rotation number $p/q$. Denote by $\{\theta_i(a_m)\}$ the angles for these external rays, then these angles form a cycle with rotation number $p/q$ under multiplication by 3. The external rays with the same angles, denoted by $R^\infty_{\Tilde{a}_m}{(\theta_i(a_m)})$, for the pinching limit $f_{\Tilde{a}_m}$ land at $z=0$.

\begin{lemma}\label{lem.different.angles}
   The mapping $m\mapsto\{\theta_i(a_m)\}$ is injective ($0 \leq m\leq \lfloor\frac{q}{2}\rfloor$).
\end{lemma}
\begin{proof}
   Rearranging the index we may assume that $f^m_{\Tilde{a}_m}(c^+_m)$ belongs to the immediate basin containing $c^-_m$, where $c^\pm_m$ are the critical points of $f^m_{\Tilde{a}_m}$. Let $\{\theta_0(a_m),...,\theta_{q-1}(a_m)\}$ be such that $3\cdot\theta_i(a_m) = \theta_{i+1}(a_m)$ and $c_m^+$ is in the sector bounded by $R^\infty_{\Tilde{a}_m}{(\theta_0(a_m)}),R^\infty_{\Tilde{a}_m}{(\theta_1(a_m)})$. Let $d_i$ be the length of $[\theta_{i}(a_m),\theta_{i+1}(a_m)]$. Then $(d_i)_i$ satisfy

\begin{equation}
   \left\{
\begin{aligned}
&3d_0 = \sum_{k=0}^{q}d_k+d_1,\,\,3d_m = \sum_{k=0}^{q}d_k+d_{m+1}\\
&3d_i=d_{i+1},\,\,i\not=0,m
\end{aligned}
\right. \quad\quad\text{if } m\not=0,
\end{equation}

\begin{equation}
3d_0 = 2\sum_{k=0}^{q}d_k+d_1,\,\,3d_i=d_{i+1},\,\,i\not=0
\quad\quad\text{if } m=0.
\end{equation}
Notice that $\sum d_i = 1$, one can solve the equations above to get $d_0 = \frac{3^q}{3^q-1}\cdot(\frac{1}{3^{m+1}}+\frac{1}{3})$ and $d_m = \frac{3^q}{3^q-1}\cdot(\frac{1}{3^{q+1-m}}+\frac{1}{3})$ if $m\not=0$; $d_0 = \frac{3^q}{3^q-1}\cdot\frac{2}{3}$ if $m=0$. Thus the $\{\theta_i{(a_m)}\}$ are different for $0\leq m\leq \lfloor\frac{q}{2}\rfloor$.
   
\end{proof}

\begin{figure}[H] 
\centering 
\includegraphics[width=0.45\textwidth]{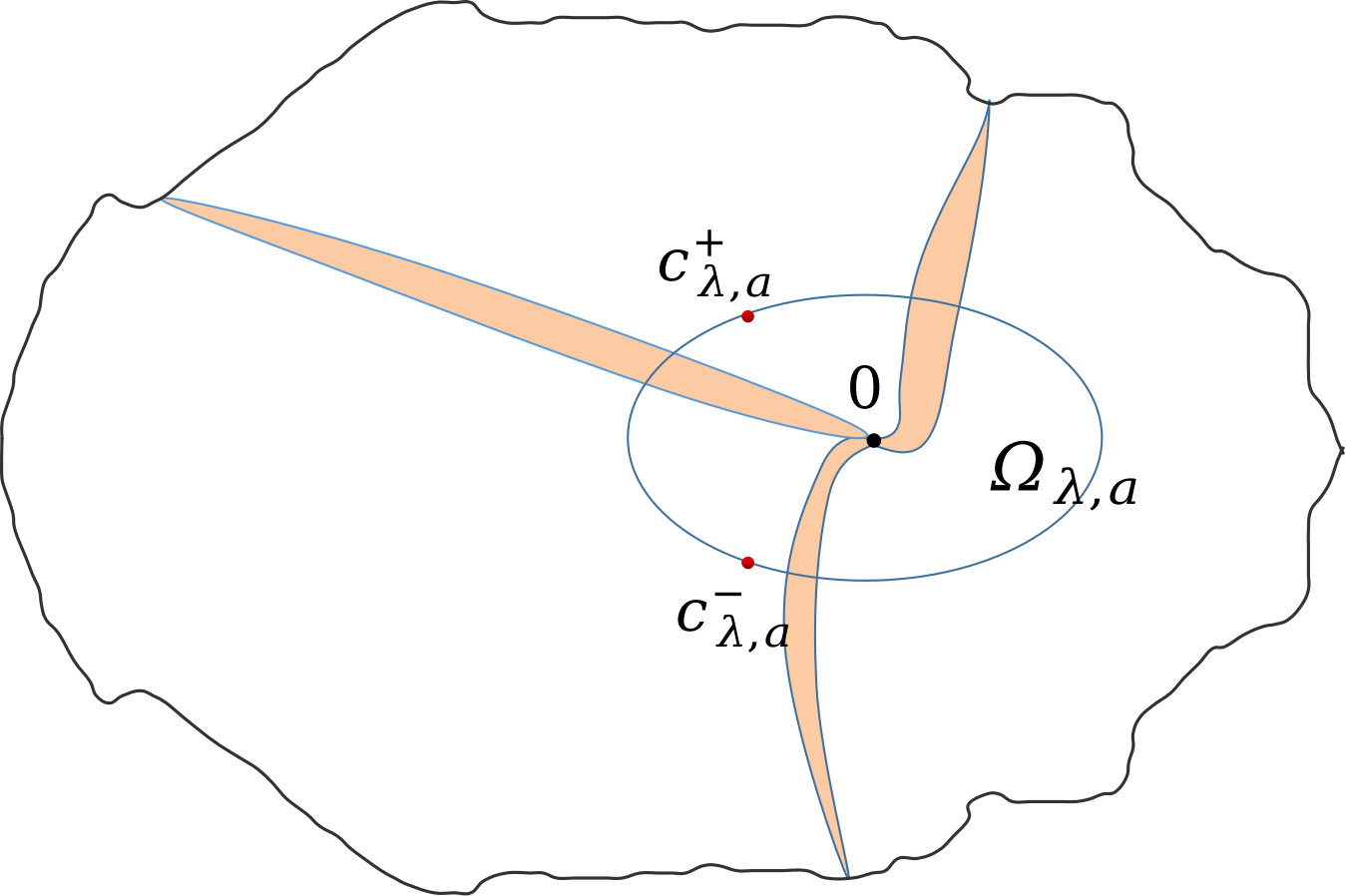} 
\caption{An illustration for pinching bands of depth 0 in $B^*_{\lambda,a}$ for $\frac{p}{q}=\frac{2}{3}$. The resulting pinching limit is in an 1-component.} 
\label{fig.main.parabolic} 
\end{figure}

\begin{definition}\label{def.rotation.angle}
By Lemma \ref{lem.different.angles}, we get $\lfloor\frac{q}{2}\rfloor+1$ different cycles of angles $\{\theta_i(a_m)\}$ ($0\leq m\leq \lfloor\frac{q}{2}\rfloor$) with rotation number $p/q$ under multiplication by 3. Notice that $\{\theta_i(a_m)+\frac{1}{2}\}$ is also a cycle with rotation number $p/q$ and is different from $\{\theta_i(a_m)\}$ if $m\not = \frac{q}{2}$. Therefore we get $q+1$ different cycles of rotation number $p/q$. By Theorem \ref{thm.goldberg}, there are exactly $q+1$ cycles of rotation number $p/q$ under multiplication by 3. Hence $\{\theta_i(a_m)\}, \{\theta_i(a_m)+\frac{1}{2}\}$ only depend on $m$. Denote by $\Theta_m$ the cycle $\{\theta_i(a_m)\}$ if $0\leq m\leq \lfloor\frac{q}{2}\rfloor$; the cycle $\{\theta_i(a_{q-m})+\frac{1}{2}\}$ if $\lfloor\frac{q}{2}\rfloor\textless m\leq q$.
\end{definition}

\begin{remark}\label{rem.angle.notdepending}
    The condition that $\Tilde{L}^k_{p/q}$ is the central line of $S_k$ is not essential: suppose $\Tilde{L}^0_{p/q}\subset S_0$ is any line parallel to $L^0_{p/q}$ and $\Tilde{L}^k_{p/q}=\Tilde{L}^k_{p/q} = L^0_{p/q}-k\log\lambda\pmod {2\pi i}$ be the line intersecting $i(-2\pi,0]$ at $\Tilde{y}_k$. For each $0\leq m\leq \lfloor\frac{q}{2}\rfloor$, similarly choose $(\lambda,a_m)$ so that $(-I_{\lambda}(a_m),0)$ contains $m$ elements in $\{\Tilde{y}_0,...,\Tilde{y}_{q-1}\}$. Then the portrait at $z=0$ for the resulting pinching limit is still $\Theta_m$.
\end{remark}

\begin{definition}\label{def.m-compo.f_a}
Let $0\leq m\leq q$. We call an adjacent or bitransitive component of the family $f_a$ a \textbf{$\textbf{m}$-component} if in which all the parameters have portrait $\Theta_m$ at $z=0$. In particular, an adjacent component is either a $0$-component or a $q$-component.
\end{definition}

\begin{definition}
Let $0\leq m\leq q-1$. For the family $g_{\lambda,c}$, an adjacent or bitransitive component $\check{\mathcal{B}}$ is called a \textbf{$\textbf{m}$-component} if for all $g_{\lambda,c}\in\check{\mathcal{B}}$, there are $m$ repelling axis between $1$ and $c$ in the counterclockwise direction. In particular, an adjacent component is a $0$-component.
\end{definition}

From the discussion above, we see that $m$-component exists with every $0\leq m\leq q$ for the family $f_{\lambda,a}$ and with every $0\leq m\leq q-1$ for the family $g_{\lambda,c}$.

\subsection{Double parabolic parameters}\label{subsec.double_parabolic}

Let $\lambda = e^{2\pi i \frac{p}{q}}$. By Fatou-Leau Theorem we have Taylor expansions near 0:

\begin{equation}\label{eq.formula.f^q}
    f^q_{\lambda,a}(z) = z+A_{p/q}(a)z^{q+1}+\mathcal{O}(z^{q+2}).
\end{equation}

\begin{equation}\label{eq.formula.g^q}
    g^q_{\lambda,c}(z) = z+C_{p/q}(c)z^{q+1}+\mathcal{O}(z^{q+2}).
\end{equation}

\begin{definition}
We say that $f_{\lambda,a}$ (resp. $g_{\lambda,c}$) is \textbf{double parabolic} if $A_{p/q}(a) = 0$ (resp. $C_{p/q}(c) = 0$). For the family $f_{\lambda,a}$, denote by $\mathcal{A}_{p/q}$ the collection of these parameters. 
\end{definition}

In this subsection we will show that there are exactly $q$ double parabolic parameters for families $f_{\lambda,a}$ and $g_{\lambda,c}$. 

\begin{lemma}[\cite{arnaud}]\label{lem.deg.C_pq}
$C_{p/q}(\frac{1}{c})$ is a polynomial in $c$ of degree $q$.
\end{lemma}
\begin{proof}
Since $C_{p/q}(c)$ is a polynomial in $\frac{1}{c}$, $C_{p/q}(\frac{1}{c})$ is a polynomial in $c$. When $c$ tends to 0, the map $g_c(z) = c^{-1}g_{\lambda,c}(cz)$ converges uniformly on compact subsets of $\mathbb{C}$ to $P_{\lambda}(z) = \lambda z(1-\frac{z}{2})$. Let $P^q_{\lambda}(z) =  z+C_0z^q+1+\mathcal{O}(z^{q+2})$. Then $C_0\not = 0$ because $P_{\lambda}$ has only one parabolic basin. While from (\ref{eq.formula.g^q}) we see that $g^q_c(z) = z + c^qC_{p/q}(c)z^{q+1}+\mathcal{O}(z^{q+2})$. Therefore $c^qC_{p/q}(c)$ converges to $C_0$ when $c\to 0$. Hence $C_{p/q}(\frac{1}{c})$ has degree $q$.
\end{proof}

By the above lemma and the relation between $f_{\lambda,a},g_{\lambda,c}$ given by Lemma \ref{lem.relation}, it suffices to find $q$ different double parabolic parameters for the family $f_{\lambda,a}$. Similarly we construct such parameters by pinching deformation.

Consider the family $f_{\lambda,a}$ for $0\textless\lambda\textless1$. Recall in Subsection \ref{subsec.existence.compo} the construction of lines $L^k_{p/q}$, the corresponding strips $S_k$ and the corresponding central lines $\Tilde{L}^k_{p/q}$, $0\leq k\leq q-1$. Recall that $\Tilde{L}^k_{p/q}$ intersect $i\mathbb{R}$ at $\Tilde{y}_k$. For $0\leq m\leq \lfloor\frac{q}{2}\rfloor$, choose $(\lambda,{a}'_m)\in(0,1)\times(-\sqrt{3\lambda},\sqrt{3\lambda})$ such that $-I_{\lambda}({a}'_m) = \Tilde{y}_m$ (Lemma \ref{lem.different.angles}). Let $S^+_k,S^-_k$ be the two components of $S_k\setminus \Tilde{L}^k_{p/q}$. Let $(\Tilde{L}^k_{p/q})^+,(\Tilde{L}^k_{p/q})^-$ be the central line respectively. For each $k$ pick two narrow strips $\Tilde{S}^+_k,\Tilde{S}^-_k$ centred at $(\Tilde{L}^k_{p/q})^+,(\Tilde{L}^k_{p/q})^-$. Define a non-seperating annulus for $f_{\lambda,\Tilde{a}_m}$: \[\Tilde{\mathscr{A}} = \pi\phi^{-1}_{\lambda,\tilde{a}_m}(\exp(-\bigcup_{k=0}^{q-1}\Tilde{S}^+_k\cup\Tilde{S}^-_k))\]
For $0\leq m\leq \lfloor\frac{q}{2}\rfloor$, the corresponding pinching limit yields a double parabolic parameter $\boldsymbol{\mathrm{a}}_m\in Per_1(e^{2\pi i\frac{p}{q}})$. Its portrait at $z=0$ is $\Theta_m\cup\Theta_{m+1}$ by Remark \ref{rem.angle.notdepending}. Therefore the portrait for $-\boldsymbol{\mathrm{a}}_m$ at $z=0$ is $\Theta_{q-m}\cup\Theta_{q-m-1}$. Thus for every $0\leq m\leq q-1$, we obtain a double parabolic parameter $\boldsymbol{\mathrm{a}}_m$ with portrait $\Theta_m,\Theta_{m+1}$ at $z=0$.
\begin{definition}\label{def.m-doubleparabolic}
Let $0\leq m\leq q-1$. $\boldsymbol{\mathrm{a}}_m$ is called the double parabolic parameter of {$\textbf{m}$}\textbf{-type}.
\end{definition}

\begin{lemma}\label{lem.double.para.boudary}
A $m-$component for $f_a$ can only have double parabolic parameters of type $m,m-1$ on its boundary.
\end{lemma}
\begin{proof}
Suppose the contrary that on the boundary there is a double parabolic parameter $a_0$ of type $n$, $n\not = m,m-1$. Then the cycle with angle $\Theta_m$ will land a repelling $q$-cycle of $f_{a_0}$. Since the landing property is stable (cf.\cite{DH}), we conclude that for all $a$ in this $m$-component, $R^\infty_{\lambda,a}(t)$ with $t\in\Theta_m$ lands at a repelling periodic point, a contradiction since these rays should land at $z=0$.
\end{proof}


 \subsection{Parametrisation of parabolic components}\label{subsec.parame.component}
 In this subsection always fix $\lambda = e^{2\pi i\frac{p}{q}}$. We parametrize $m$-components $\check{\mathcal{B}}_m$ of family (\ref{eq.para.c}) by locating the free critical value in the immediate basins of the quadratic model $P_{\lambda}(z) = \lambda z+ z^2$.  \\
 
The critical point of $P_{\lambda}$ is $-\frac{\lambda}{2}$. Denote by $B_0^*(0)$ the immediate basin of $P_{\lambda}$ containing $-\frac{\lambda}{2}$. There are exacty $q$ immediate basins attached at 0 in the cyclic order $B^*_0(0),...,B^*_{q-1}(0)$. Let $\Omega^0_0$ be an maximal admissible petal of $P^q_{\lambda}|_{B^*_0(0)}$. Let $\phi:\overline{\Omega^0_0}\longrightarrow\overline{\mathbb{H}}$ ($\mathbb{H}$ is the right half plan) be the Fatou coordiante normalised by $\phi({-\frac{\lambda}{2}}) = 0$. Let $1\leq m\leq q$. Define $\Tilde{\Omega}^0_{c,m} := g_{\lambda,c}^s(\Omega^0_{c,0})$. For any $0\leq k\leq q-1$, $n\in \mathbb{Z}$ such that $np+k=0 \pmod q$, define $\Omega^n_{k}$ and $\tilde{\Omega^n_{k}}$ as in  (\ref{eq.sequence}). For simplicity we will omit index $\lambda$ for all terms related to family $g_c := g_{\lambda,c}$.

Now consider family $g_c$. Let $\Omega^0_{c,0}\subset B^*_{c,0}(0)$ be the standard maximal petal of $g_c^q|_{B^*_{c,0}(0)}$. For $1\leq m\leq q-1$, let $\Tilde{\Omega}^0_{c,m} := g_c^s(\Omega^0_{c,0})$ where $1\leq s\leq q-1$ is the smallest integer such that $sp=m\pmod q$. For any $0\leq k\leq q-1$, $n\in \mathbb{Z}$ such that $np+k=0 \pmod q$, $\tilde{\Omega}^n_{c,k}$ is defined as in (\ref{eq.sequence}). The degree of $g^q_c|_{B^*_{c,m}(0)}$ is 3 if $m=0$ and is 4 if $1\leq m\leq q-1$.  In the latter case, denote by $1,cr_1,cr_2$ respectively the three critical points of $g^q_c|_{B^*_c(0)}$; denote by $c,\tilde{cr}_1,\tilde{cr}_2$ respectively the three critical points of $g^q_c|_{B^*_{c,m}(0)}$. Notice that $g^{r}_c(cr_{1,2}) = c$, $g^{q-r}_c(\tilde{cr}_{1,2}) = 1$.

We introduce the following loci in order to distinguish which critical point is on the maximal petal:
\begin{equation}\label{eq.special.locus}
\begin{split}
&{{{{{{\check{\mathcal{D}}}}}}}}_0 = \{c\in{\check{\mathcal{B}}}_0;\,1\in\partial\Omega_{c,0}^0 \text{ but } c\not\in\partial\Omega_{c,0}^0\}\\
&{\check{\mathcal{I}}}_0 = \{c\in{\check{\mathcal{B}}}_0;\,1,c\in\partial\Omega_{c,0}^0\}\\
&{{{{{{\check{\mathcal{D}}}}}}}}_m = \{c\in{\check{\mathcal{B}}}_m;\,1\in\partial\Omega_{c,0}^0 \text{ but } cr_1,cr_2\not\in\partial\Omega_{c,0}^0\},\text{ for }1 \leq m\leq q-1\\
&{\check{\mathcal{I}}}_m = \{c\in{\check{\mathcal{B}}}_m;\,1 \text{ and one of }cr_1,cr_2\in\partial\Omega_{c,0}^0\},\text{ for }1 \leq m\leq q-1
\end{split}
\end{equation}

Let $\Tilde{\check{\mathcal{D}}}_m = {\check{\mathcal{B}}}_m\setminus({\check{\mathcal{I}}}_m\cup{{{{{{\check{\mathcal{D}}}}}}}}_m)$. Clearly for $c\in \Tilde{\check{\mathcal{D}}}_m$, $1\not\in\partial\Omega^0_{c,0}$. In particular, if $1\leq m\leq q$, then $1\not\in\partial \Omega^0_{c,0}$, hence one of $cr_1,cr_2\in\partial \Omega^0_{c,0}$, and $\tilde{\Omega}^0_{c,m}\subset B^*_{c,m}(0)$ is a maximal petal for $g^q_c|_{B^*_{c,m}(0)}$, having $c$ on its boundary but not containing $\tilde{cr}_{1,2}$. Thus we have
\begin{equation}
    \begin{split}
         &\Tilde{{{{{{\check{\mathcal{D}}}}}}}}_0 = \{c\in{\check{\mathcal{B}}}_0;\,c\in\partial\Omega_{c,0}^0 \text{ but } 1\not\in\partial\Omega_{c,0}^0\}\\
         &\Tilde{{{{{{\check{\mathcal{D}}}}}}}}_m \subset \{c\in{\check{\mathcal{B}}}_m;\,c\in\partial\Tilde{\Omega}_{c,m}\text{ but }\Tilde{cr}_1,\Tilde{cr}_2\not\in\partial\Tilde{\Omega}_{c,m}\},\text{ for }1 \leq m\leq q-1
    \end{split}
\end{equation}

\begin{remark}
By Proposition \ref{prop.first_critical_point}, $\check{\mathcal{D}}_m,\tilde{\check{\mathcal{D}}}_m$ are open for $0\leq m\leq q-1$.
\end{remark}

\begin{lemma}\label{lem.cr1=cr2}
Let $c\in\check{\mathcal{B}}_m$ and $\Tilde{\phi}_c$ any Fatou coordiante for $g^q_c|_{B^*_{c,0}(0)}$. For $1\leq m\leq q-1$, if $g^q_c(cr_1,2),g^q_c(1)$ are contained $\Omega^0_{c,0}$ and $\Tilde{\phi}_c(cr_{1,2}) = \Tilde{\phi}_c(1)$, then $1=cr_1=cr_2$; for $m=0$, $g^q_c(c),g^q_c(1)$ are contained $\Omega^0_{c,0}$ and $\Tilde{\phi}_c(c) = \Tilde{\phi}_c(1)$, then $1=c$.
\end{lemma}
\begin{proof}
We only do the proof for $1\leq m\leq q-1$, the case $m=0$ is similar. By hypothesis, since $\tilde{\phi}_c$ is injective on $\Omega^0_{c,0}$, we have $g^q_c(cr_1,2)=g^q_c(1)$. If $cr_1,cr_2,1$ are not distinct, then $g^q_c(1)$ has at least $5$ preimages counting multiplicity, contradicting $deg(g^q_c|_{B^*_{c,0}(0)}) = 4$.
\end{proof}

For $0\leq m\leq q-1$ and $c\in{{{{{\check{\mathcal{D}}}}}}}_m$ (resp. $c\in\Tilde{{{{{{{\check{\mathcal{D}}}}}}}}}_0$), let
$\phi_c:\overline{\Omega^0_{c,0}}\longrightarrow\overline{\mathbb{H}}$ be the Fatou coordiante of $g^q_c|_{B^*_{c,0}(0)}$ normalised by $\phi_c({1}) = 0$ (resp. $\phi_c({c}) = 0$). Define $h_c:\Omega^0_{c,0}\longrightarrow\Omega_0^0$ by $h_c = \phi^{-1}\circ\phi_c$. Pull back $h_c$ by $g_c$ and $P_{\lambda}$ until we reach the critical value $g_c(c)$ (resp. $g_c(1)$):
 \begin{equation}\label{diag.commu1}
\begin{tikzcd}
\Omega_{c,\overline{m+p}}^{n_{c}} \ar[r]{}{g_{c}}\ar[d]{}{h_{c}} & \Omega_{c,\overline{m+2p}}^{n_{c}-1} \ar[d]{}{h_{c}} \ar[r]{}{g_{c}} & ... \ar[d]{}{h_{c}}  \ar[r]{}{g_{c}} & \Omega_{c,q-p}^1 \ar[d]{}{h_{c}} \ar[r]{}{g_{c}} & \Omega_{c,0}^0 \ar[d]{}{h_{c}}\\
\Omega^{n_c}_{\overline{m+p}} \ar[r]{}{P_{\lambda}} & \Omega^{n_c-1}_{\overline{m+2p}}  \ar[r]{}{P_{\lambda}} & ...  \ar[r]{}{P_{\lambda}} & \Omega^1_{q-p}  \ar[r]{}{P_{\lambda}} & \Omega^0_0
\end{tikzcd}
\end{equation}
where $n_c$ is the smallest integer such that $\Omega_{c,\overline{m+p}}^{n_{c}}$ contains $g_{c}(c)$ (resp. $g_c(1)$). Moreover at each step $h_c$ is conformal.

Similarly, for $1\leq m\leq q-1$ and $c\in\Tilde{{{{{{\check{\mathcal{D}}}}}}}}_m$, let $\Tilde{\phi_c}:\overline{\Tilde{\Omega}_{c,m}}\longrightarrow\overline{\Tilde{\mathbb{H}}}$ be the Fatou coordinate of $g^q_c|_{B^*_{c,m}(0)}$ normalised by $\Tilde{\phi}_{c}(c)=0$, where $\Tilde{\mathbb{H}} = \Tilde{\phi_c}(\Tilde{\Omega}_{c,m})$. Define $\Tilde{h}_c:\Tilde{\Omega}_{c,m} =: \Tilde{\Omega}^0_{c,m}\longrightarrow {\Omega}^0_{0}$ by $\Tilde{h}_c = \phi^{-1}\circ\Tilde{\phi}_c$. Pull back $\tilde{h}_c$ by $g_c$ and $P_{\lambda}$ until we reach the critical value $g_c({1})$.
 \begin{equation}\label{diag.commu2}
\begin{tikzcd}
\Tilde{\Omega}_{c,{p}}^{n_{c}} \ar[r]{}{g_{c}}\ar[d]{}{\tilde{h}_{c}} & \Tilde{\Omega}_{c,\overline{2p}}^{n_{c}-1} \ar[d]{}{\tilde{h}_{c}} \ar[r]{}{g_{c}} & ... \ar[d]{}{\tilde{h}_{c}}  \ar[r]{}{g_{c}} & \Tilde{\Omega}_{c,\overline{m-p}}^1 \ar[d]{}{\tilde{h}_{c}} \ar[r]{}{g_{c}} & \Tilde{\Omega}_{c,m}^0 \ar[d]{}{\tilde{h}_{c}}\\
{\Omega}^{n_c}_{\overline{p-m}} \ar[r]{}{P_{\lambda}} & {\Omega}^{n_c-1}_{\overline{2p-m}}  \ar[r]{}{P_{\lambda}} & ...  \ar[r]{}{P_{\lambda}} & {\Omega}^1_{q-p}  \ar[r]{}{P_{\lambda}} & {\Omega}^0_0
\end{tikzcd}
\end{equation}

Define four holomorphic mappings
\begin{equation}\label{eq.parametri}
\begin{split}
&\Phi^{adj}_{0}:\check{\mathcal{D}}_{0}\longrightarrow B^*_{{p}}(0),\,\,c\mapsto h_c(g_c(c))\\
&\Tilde{\Phi}^{adj}_{0}:\Tilde{\check{\mathcal{D}}}_{0}\longrightarrow B^*_{{p}}(0),\,\,c\mapsto {h}_c(g_c(1))\\
&\Phi^{bit}_{m}:{{{{{\check{\mathcal{D}}}}}}}_{m}\longrightarrow B^*_{\overline{m+p}}(0)\setminus \overline{{\Omega}^{-s}_{\overline{m+p}}},\,\,c\mapsto h_c(g_c(c))\\
&\Tilde{\Phi}^{bit}_{m}:\Tilde{\check{\mathcal{D}}}_{m}\longrightarrow B^*_{\overline{p-m}}(0)\setminus \overline{{\Omega}^{l}_{\overline{p-m}}},\,\,c\mapsto \Tilde{h}_c(g_c(1))
\end{split}
\end{equation}
where $2\leq s\leq q$, $0\leq l\leq q-1$ are the unique integers such that $-sp+m+p=0\pmod q$ and $lp+p-m=0\pmod q$.

\begin{lemma}\label{lem.proper.bit}
Let $\Phi$ be one of the four holomorphic maps in (\ref{eq.parametri}). Then $\Phi$ satisfies the following property: let $(c_n)\subset\check{\mathcal{B}}_m$ be a sequence contained in the domain of definition of $\Phi$ which converges to $\partial\check{\mathcal{B}}_m$ (resp. $\check{\mathcal{I}}_m$), then $\Phi(c_n)$ converges to the boundary of the corresponding immediate basin (resp. $\partial{\Omega^{-1}_{{p}}}\setminus\{0\},\partial{\Omega^{-s}_{\overline{m+p}}}\setminus\{0\}$ or $\partial{{\Omega}^{l}_{\overline{p-m}}}\setminus\{0\}$).
\end{lemma}
\begin{proof}
We only do the proof for $\Phi^{bit}_{m}$, the others are similar. First we justify that $\Phi^{bit}_{m}(\check{\mathcal{D}}_m)\subset B^*_{\overline{m+p}}(0)\setminus \overline{{\Omega}^{-s}_{\overline{m+p}}}$. Since $\Phi^{bit}_{m}$ is open, it suffices to prove that its image does not intersect ${\Omega}^{-s}_{\overline{m+p}}$. If not, then there exists $c\in\check{\mathcal{D}}_m$ such that $g_c(c)\in{\Omega}^{-s}_{c,\overline{m+p}}$, hence $g^q_c(cr_1) = g^q_c(cr_2)\in {\Omega}^{-q}_{c,0}$. Let $\Lambda= g^{-q}_c({\Omega}^{-q}_{c,{0}})\cap B^*_{c,0}(0)$. Then $g^q:\Lambda\longrightarrow{\Omega}^{-q}_{c,{0}}$ is of degree 4. Since $g_c$ is a polynomial, each component of $\Lambda$ is simply connected. Applying Riemann-Hurwitz formula one easily sees that $\Lambda$ should contain 3 critical points of $g_c^q$. However $1\not\in\Lambda$ since $1\in\partial{\Omega}^{0}_{c,{0}}$, so $g^q_c(1)\in\partial{\Omega}^{-q}_{c,{0}}$.

Next we verify properness of $\Phi^{bit}_m$. Clearly $\partial\check{\mathcal{D}}_m\subset{\check{\mathcal{I}}}_m\cup\partial{\check{\mathcal{B}}}_m$. Let $(c_n)\subset{{{{{\check{\mathcal{D}}}}}}}_m$ be a sequence converging to some $c_0\in\partial{{{{{\check{\check{\mathcal{D}}}}}}}}_m$. If $c_0\in{\check{\mathcal{I}}}_m$, then by Corollary \ref{cor.fatou.extend}, $g^k_{c_n}(c_n)$ is compactly contained in $\Omega^0_{c_0,0}$ for $n$ large enough, where $k$ is the smallest integer such that $g^k_{c_n}(c_n)\in B^*_{c_n,0}(0)$. Moreover $h_{c_n}(g^k_{c_n}(c_n))$ converges to $\partial{\Omega}^{-q}_{0}\setminus\{0\}$ since $\mathfrak{Re}\{\phi_{c_n}(g^k_{c_n}(c_n))-\phi_{c_n}(g^q_{c_n}(1))\}$ converges to 0 and $\mathfrak{Im}\{\phi_{c_n}(g^k_{c_n}(c_n))-\phi_{c_n}(g^q_{c_n}(1))\}$ remains bounded. (In fact we have $h_{c_n}(g^k_{c_n}(c_n))\to z_0\in\partial{\Omega}^{-q}_{0}\setminus\{0\}$ with $\phi(z_0) = I_m(c_0)$). This implies that $\Phi^{bit}_m({c_n}) = h_{c_n}(c_n) = P^{-k}_{\lambda}h_{c_n}(g^k_{c_n}(c_n))$ converges to $\partial{\Omega}^{-s}_{\overline{m+p}}\setminus\{0\}$.

So let $c_0\in\partial{\check{\mathcal{B}}}_m$. We want to prove that $\Phi^{bit}_m(c_n)$ converges to $\partial B^*_{\overline{m+p}}(0)$. Suppose the contrary that, up to taking a subsequence, $\Phi^{bit}_m(c_n)$ converges to $z_0\in B^*_{\overline{m+p}}(0)$. Clearly $\{(h_{c_n}|_{\Omega_{c_n,0}^0})^{-1}\}$ is a normal family on $\Omega_0^0$. Up to taking a subsequence, suppose $h^{-1}_{c_n}$ converges uniformly to $\psi$. We claim that $\psi(\Omega^0_0)$ does not intersect $J_{c_0}$, the Julia set of $g_{c_0}$. Indeed, if not, then there exists a repelling periodic point $x_{c_0}\in J_{c_0}\cap \psi(\Omega^0_0)$ (since the repelling cycles are dense in $J_{c_0}$). Notice that $x_{c_0}$ moves holomorphically for $c$ in a small neighborhood of $c_0$, which gives a repelling periodic point $x_c$ of $g_c$. However for $n$ large enough, $x_c\in h^{-1}_{c_n}(\Omega_0^0)\subset B^*_{c,0}(0)$, a contradiction. Therefore $\psi(\Omega^0_0)$ and $g^q_{c_0}(1)$ are contained in the same Fatou component of $g_{c_0}$. On the other hand, 
take $s\geq 0$ such that $P^s_{\lambda}(z_0)\in{\Omega}^0_0$, then by diagram (\ref{diag.commu1}), $g^s_{c_n}(c_n)\in \Omega^0_{c_n,0}$ and $h_{c_n}(g^s_{c_n}(c_n))$ converges to $P^s_{\lambda}(z_0)$. Therefore 
\[\psi(P^s_{\lambda}(z_0)) = \lim\limits_{n\to\infty}h^{-1}_{c_n}(h_{c_n}(g^s_{c_n}(c_n))) = g^s_{c_0}(c_0)\in\psi(\Omega^0_0)\]
which means that $1,c_0$ are eventually attracted by the same Fatou component, contradicts $c_0\in\partial{\check{\mathcal{B}}}_m$.
\end{proof}

\begin{proposition}\label{prop.injec.parame.bit}
The mappings $\Phi^{bit}_{m},\Tilde{\Phi}^{bit}_{m}$ in (\ref{eq.parametri}) are bijective.
\end{proposition}
\begin{proof}
We only do for $\Phi^{bit}_{m}$, the other is similar. By Lemma \ref{lem.proper.bit}, it suffices to verify injectivity. Suppose $\Phi^{bit}_{m}(c_1) = \Phi^{bit}_{m}(c_2)$. Starting from the conjugacy $g_{c_2}\circ\phi^{-1}_{c_2}\circ\phi_{c_1} = \phi^{-1}_{c_2}\circ\phi_{c_1}\circ g_{c_1}$ on $\Omega^0_{c_1,0}$, lift $\phi^{-1}_{c_2}\circ\phi_{c_1}$ by $g_{c_1},g_{c_2}$ to $\Omega^{n}_{c_1,k}\subset B^*_{c_1,k}$ (recall definition of $\Omega^{n}_{c,k}$ at the beginning of \ref{subsec.parame.component}). Denote by $\varphi$ the lifting of $\phi^{-1}_{c_2}\circ\phi_{c_1}\circ g_{c_1}$. The only possible issue that might stop the lifting is when $k=\overline{m+p}$ and $g_{c_1}(c_1)\in\Omega^{n}_{c_1,k}$. However $\Phi^{bit}_{m}(c_1) = \Phi^{bit}_{m}(c_2)$ implies $\varphi(g_{c_1}(c_1)) = g_{c_2}(c_2)$, which ensures that the lifting is still valid (also notice that $g_{c_1}|_{B^*{c_1,m}},g_{c_1}|_{B^*{c_1,m}}$ are of degree 2). So $\varphi$ is extended to $\bigcup_i B^*_{c_1,i}$ and hence to all filled-in Julia set $K_{c_1}$.

On the other hand, take a connected open set $\check{\mathcal{V}}\subset\check{B}_m$ linking $c_1,c_2$, on which the Böttcher coordinate at $\infty$ depends analytically on $c$. Therefore for $c\in\check{\mathcal{V}}$, $\psi_c := (\phi^\infty_{c})^{-1}\circ\phi^\infty_{c_1}$ is a dynamical holomorphic motion on $C\setminus\mathring{K}_{c_1}$, and can be quasiconformally extend to $\mathbb{C}$ (Slodkowski's theorem). In particular,  $\psi_{c_2}$ conjugates $g_{c_1}$ to $g_{c_2}$ on $\mathbb{C}\setminus\mathring{K}_{c_1}$, coincides with $\varphi$ on $\partial K_{c_1}$. Applying Rickman's lemma, the global conjugacy defined by sewing $\varphi$ and $\psi_{c_2}$ is conformal, i.e. identity, so $c_1=c_2$. 
\end{proof}

The injectivity for $\Phi^{adj}_0,\tilde{\Phi}^{adj}_0$ is more subtle:
\begin{proposition}\label{prop.injec.parame.adj}
Let $c_1,c_2\in\check{\mathcal{D}}_0$ be such that $\Phi^{adj}_{0}(c_1) = \Phi^{adj}_{0}(c_2)$. 
\begin{enumerate}
    \item If $\Phi^{adj}_{0}(c_1)\in\overline{\Omega^{-1}_p}$, then $c_1=c_2$;
    \item if not, then for $i=1,2$, $(g^{-1}_{c_i})^{-1}(\overline{{\Omega^{-1}_{c_i,p}}})$ has two connected components, one of which has the figure of a filled eight, denoted by $H_{c_i}$. $H_{c_i}$ cuts $B^*_{c_i,0}(0)$ into two connected components  $\Delta^+_{c_i},\Delta^-_{c_i}$ ($+$ is on the right-hand side of $-$). If $c_i\in\Delta^+_{c_i}$ for $i=1,2$ or $c_i\in\Delta^-_{c_i}$ for $i=1,2$, then $c_1=c_2$.
\end{enumerate}
The same results hold for $\tilde{\Phi}^{adj}_0$. 
\end{proposition}

\begin{proof}
We do the proof for the second point. The first is similar. Without loss of generality suppsoe $c_i\in\Delta^-_{c_i}$ for $i=1,2$. Like what we did at the begining of the proof of Proposition \ref{prop.injec.parame.bit}, lift $\varphi = \phi^{-1}_{c_2}\circ\phi_{c_1}$ by $g_{c_1},g_{c_2}$. Similarly, when we lift $\varphi$ to $\Omega^n_{c_1,0}$ with $n_0$ the smallest integer such that $g^q_{c_1}(c_1)\in\Omega^{n_0}_{c_1,0}$, it is not clear whether $\varphi$ can be lifted once more to $\Omega^{n_0+q}_{c_1,0}$, where $n_0=k_0q$ for some $k_0\geq 1$. Notice that for $i=1,2$, $1\leq k\leq k_0+1$, $\Omega^{kq}_{c_i,0}$ is simply connected with piecewise smooth boundary intersecting $\partial B^*_{c_i,0}(0)$ at $2^{k}$ points, and $\Omega^{kq}_{c_i,0}\setminus\Omega^{(k-1)q}_{c_i,0}$ has $2^{k-1}$ connected components $D^i_{\epsilon_0...\epsilon_{k-2}}$ with
\[
\epsilon_j=\left\{
\begin{aligned}
&0, \text{ if }g^{jq}_{c_i}(D^i_{\epsilon_0...\epsilon_{k-2}}\cap\overline{\Omega^{kq}_{c_i,0}})\subset\overline{\Delta^-_{c_i}}  \\
&1, \text{ if }g^{jq}_{c_i}(D^i_{\epsilon_0...\epsilon_{k-2}}\cap\overline{\Omega^{kq}_{c_i,0}})\subset\overline{\Delta^+_{c_i}}
\end{aligned}
\right.\quad\]

Since $\Phi^{adj}_0(c_1) = \Phi^{adj}_0(c_2)$, there exists $\tilde{\epsilon}_0...\tilde{\epsilon}_{k-2}$ such that $g_{c_i}(c_i)\in D^i_{\tilde{\epsilon}_0...\tilde{\epsilon}_{k-2}}$. So $c_i\in D^i_{\delta_i\tilde{\epsilon}_0...\tilde{\epsilon}_{k-2}}$. But by hypothesis $c_i\in\Delta^-_{c_i}$, so $\delta_i = 0$ for $i=1,2$. Therefore we can extend $\varphi$ to $\Omega^{(k_0+1)q}_{c_1,0}$ by assigning injectively $D^1_{\epsilon_0...\epsilon_{k_0-1}}$ onto $D^2_{\epsilon_0...\epsilon_{k_0-1}}$. Notice that $\Omega^{(k_0+1)q}_{c_1,0}$ contain the two critical points and the two cocritical points of $g^q_{c_1}|_{B^*_{c_1,0}(0)}$, thus the lifting process is valid for all $k\geq k_0+1$, so $\varphi$ can be extended to $B^*_{c_1,0}(0)$, and hence to the filled-in Julia set $K_{c_1}$. Now apply the same strategy as in the second paragraph of the proof of Proposition \ref{prop.injec.parame.bit}, we conclude that $c_1=c_2$.
\end{proof}

\subsection{Describing the special locus ${\check{\mathcal{I}}}_m$}

Recall the definition of ${\check{\mathcal{I}}}_m$ in (\ref{eq.special.locus}). For $m=0$, $c\in {\check{\mathcal{I}}}_0$, define $I_0(c) = \mathfrak{Im}(\phi_c(c)-\phi_c(1))$; for $1\leq m\leq q-1$, $c\in {\check{\mathcal{I}}}_m$, define $I_m(c) = \mathfrak{Im}(\phi_c(cr)-\phi_c(1))$ where $cr = cr_1$ or $cr_2$, since $\phi_c(cr_1) = \phi_c(cr_2)$.

\begin{lemma}\label{cor.0in.partialN }
For every $m$-component $\check{\mathcal{B}}_m$, $I^{-1}_m(0)\not=\emptyset$.
\end{lemma}
\begin{proof}
Let $\Phi$ be the corresponding mapping in (\ref{eq.parametri}) defined on ${\check{\mathcal{D}}}_m$. Clearly $P^s_{\lambda}(-\frac{\lambda}{2})\in\partial\Phi(\mathcal{N})\cap B^*_{\overline{m+p}}(0)$. Thus by Lemma \ref{lem.proper.bit}, there exists a sequence of parameters $({c}_k)\subset\mathcal{N}$ converging to ${c}_0\in \check{\mathcal{I}}_m$ such that $\Phi({c}_k)$ converges to $P^s_{\lambda}(-\frac{\lambda}{2})$. Recall that $\Phi(c) = h_c(g_c(c))$, thus $P^{q-s}_{\lambda}\circ h_{{c}_k}(g_{{c}_k}({c}_k)) = h_c(g^{q-s+1}_{{c}_k}({c}_k))$ converges to $P^{q}_{\lambda}(-\frac{\lambda}{2})$, which implies that $\phi_{{c}_k}(g^q_{{c}_k}(c))-\phi_{{c}_k}(g^q_{{c}_k}(1))\to 0$ if $m=0$ and  $\phi_{{c}_k}(g^q_{{c}_k}(cr_{1,2}))-\phi_{{c}_k}(g^q_{{c}_k}(1))\to 0$ if $1\leq m\leq q-1$. Thus at ${c}_0$ we have respectively $\phi_{{c}_0}(c)=\phi_{{c}_0}(1)$ and $\phi_{{c}_0}(cr_{1,2})=\phi_{{c}_0}(1)$. By Lemma \ref{lem.cr1=cr2}, $c_0=1$ resp. $cr_{1,2} =1$, i.e. $I_m({c}_0) = 0$.
\end{proof}

\begin{corollary}\label{cor.unique.B0}
For family $g_c$ the $0$-component is unique. It contains $c=1$ and is symmetric with respect to $\tau:c\mapsto1/c$. 
\end{corollary}
\begin{proof}
By Lemma \ref{cor.0in.partialN }, $I^{-1}_0(0)\subset \check{\mathcal{B}}_0$, while by Lemma \ref{lem.cr1=cr2}, $I_0^{-1}(0) = 1$. Thus $\check{\mathcal{B}}_0$ is unique. Since $\tau(\check{\mathcal{B}}_0)$ is also a 0-component, hence $\tau(\check{\mathcal{B}}_0) = \check{\mathcal{B}}_0$.
\end{proof}

\begin{lemma}\label{lem.injectivity_Im}
Let $0\leq m\leq q-1$, ${\check{\mathcal{B}}}_m$ be a $m$-component,  ${\check{\mathcal{I}}}_m$ be as in (\ref{eq.special.locus}). Then $I_m:{\check{\mathcal{I}}}_m\longrightarrow (-\infty,+\infty)$ is injective.
\end{lemma}
\begin{proof}
For $1\leq m\leq q-1$, adapt the proof of Proposition \ref{prop.injec.parame.bit}; for $m=0$ adapt that of Proposition \ref{prop.injec.parame.adj}.
\end{proof}

\begin{lemma}\label{lem.curve.smooth}
If there exists $c_{t_0}\in\check{\mathcal{I}}_m$ such that $t_0 := I_m(c_0)\textgreater 0$ (resp. $\textless 0$), then for any $t\textgreater 0$ (resp. $\textless0$) there is $c_t\in \check{\mathcal{I}}_m$ with $I_m(c_t) = t$. Moreover $I^{-1}_m((-\infty,0)),I^{-1}_m({(0,+\infty)})$ are curves parametrized by $I^{-1}_m$.
\end{lemma}
\begin{proof}
This can be shown by quasiconformal deformation.
\end{proof}

\begin{lemma}\label{lem.limit.Im}
Let $(c_k)\subset\check{\mathcal{I}}_m$ be a sequence of parameters. If
\begin{enumerate}
      \item $I_m(c_k)\to\pm\infty$, then $c_k$ converges to a double parabolic parameter;
    \item $I_m(c_k)\to0$, then $c_k$ converges to $I^{-1}_m(0)$.
\end{enumerate}

\end{lemma}
\begin{proof}
Let $c_0$ be any accumulation point of $c_k$.
\begin{enumerate}
    \item If $I_m(c_k)\to\pm\infty$, then clearly $c_0\in\partial\check{\mathcal{B}}_m$. If $c_0$ is not double parabolic, then only one critical point is in the immediate basins, say 1. One can then apply Proposition \ref{prop.first_critical_point}, which implies that for $c$ near $c_0$, $1$ is always on the boundary of the maximal petal of $g^q_c$ while the other critical points of $g^q_c$ are not, contradicting the definition of $\check{\mathcal{I}}_m$.
    
    \item If $I_m(c_k)\to0$. It suffices to prove that $c_0\in\check{\mathcal{B}}_m$, since then by taking limit in $k$, we get $I_m(c_0) = 0$. For the same reason as above, if $c_0$ is not double parabolic, then $c_0\not\in\partial\check{\mathcal{B}}_m$ and $c_0\in\check{\mathcal{B}}_m$. So it remains to show that $c_0$ is not double parabolic. Suppose the contrary. Without loss of generality we assume that $I_m(c_k)\textgreater 0$. Then by the first point and Lemma \ref{lem.curve.smooth}, $I^{-1}_m((0,+\infty))$ is a curve separating $\check{\mathcal{B}}_m$ into two simply connected components ($\check{\mathcal{B}}_m$ is simply connected, Lemma \ref{lem.boundary.components}). Let $\mathcal{N}$ be the one not containing ${I}^{-1}_m(0)$. This $\mathcal{N}$ exists since ${I}^{-1}_m(0)$ is a single point by Lemma \ref{lem.injectivity_Im}. Then $\mathcal{N}\setminus\check{\mathcal{I}}_m$ is ${{{{\check{\mathcal{D}}}}}}_{m}$ or $\Tilde{{{{{\check{\mathcal{D}}}}}}}_m$. But recall that $\partial\check{\mathcal{D}}_m\cap\check{\mathcal{B}}_m = \partial\tilde{\check{\mathcal{D}}}_m\cap\check{\mathcal{B}}_m = \check{\mathcal{I}}_m$, so $I_m^{-1}(0)\in\partial(\mathcal{N}\setminus\check{\mathcal{I}}_m)$, a contradiction. 
\end{enumerate}
\end{proof}

\begin{proposition}\label{prop.endpoints_specialcurve}
For $0\leq m\leq q-1$, $\check{\mathcal{I}}_m$ is a curve parametrized by $I^{-1}_m:(-\infty,+\infty)\longrightarrow\check{\mathcal{I}}_m$. Moreover, when $q\textgreater 1$, $\check{\mathcal{I}}_m$ has two different end points; when $q=1$, $\check{\mathcal{I}}_0$ has only one end point $-1$.
\end{proposition}
\begin{proof}
First we prove that $I^{-1}_m$ is a parametrisation. By Lemma \ref{lem.curve.smooth} and \ref{lem.limit.Im}, it suffices to prove that $I^{-1}_m((-\infty,0))$ and $I^{-1}_m((0,+\infty))$ are not empty. Consider the case $m=0$: by Corollary \ref{cor.unique.B0}, $\check{\mathcal{B}}_0$ is unique and symmetric with respect to $\tau:c\mapsto1/c$. Moreover it is easy to see that $\tau({{{{\check{\mathcal{D}}}}}}_0) = \Tilde{{{{{\check{\mathcal{D}}}}}}}_0$, which are non empty. If both $I_0^{-1}((-\infty,0)),I_0^{-1}((0,+\infty))$ are empty, then one of ${{{{\check{\mathcal{D}}}}}}_0,\Tilde{{{{{\check{\mathcal{D}}}}}}}_0$ is empty, a contradiction. So suppose $I_0^{-1}((-\infty,0))\not=\emptyset$, hence $\tau (I_0^{-1}((-\infty,0))) = I_0^{-1}((0,+\infty)) \not =\emptyset$.

Next we prove for $1\leq m\leq q-1$. By Proposition \ref{prop.injec.parame.bit}, $\Phi^{bit}_{m}:{{{{{\check{\mathcal{D}}}}}}}_{m}\longrightarrow B^*_{\overline{m+p}}(0)\setminus \overline{{\Omega}^{-s}_{\overline{m+p}}}$ is conformal, so for any $z\in\partial{\Omega}^{-s}_{\overline{m+p}}\setminus\{0\}$, there exists $c_n\to c_0\in\check{\mathcal{I}}_m$ such that $\Phi^{bit}_{m}(c_n)\to z$. Thus $I_m(c_0)$ can take arbitrary value in $\mathbb{R}$ by choosing properly $z$.\\

Now we investigate the end points of $\check{\mathcal{I}}_m$. First consider the case $m=0$. Notice that $\check{\mathcal{I}}_0$ is symmetric with respect to $\tau:c\mapsto\frac{1}{c}$. If $q = 0$, it is easy to see that there is only one double parabolic parameter $-1$, hence by Lemma \ref{lem.limit.Im} the end point of $\check{\mathcal{I}}_0$ is $-1$. If $q\textgreater 1$, we prove that $\check{\mathcal{I}}_0$ ends at two different points. Suppose the contrary, then $\check{\mathcal{I}}_0$ must land at $-1$ since $\check{\mathcal{I}}_0 = \tau(\check{\mathcal{I}}_0)$. By Lemma \ref{lem.relation}, $\sigma\iota^{-1}(\check{\mathcal{I}}_0)$ is a curve symmetric with respect to $z\mapsto-z$, starting from $\sqrt{3e^{2\pi i\frac{p}{q}}}$ and ending at $-\sqrt{3e^{2\pi i\frac{p}{q}}}$. Moreover $0=\sigma\iota^{-1}(-1)$ is a double parabolic parameter for the family $f_{a}$. By Lemma \ref{lem.double.para.boudary}, $0$ is of type $0$ and $q-1$, i.e. $q=1$, a contradiction. 

Now we consider the case $1\leq m\leq q-1$. If not, then $\overline{\check{\mathcal{I}}_m}$ is a simple closed curve. We claim that $\overline{\check{\mathcal{I}}_m}$ must separate $0,\infty$. Suppose not, then $\overline{\check{\mathcal{I}}_m}$ will bound a simply connected region $\mathcal{O}\subset\check{\mathcal{B}}_m$, since by Proposition \ref{prop.zakeri.parametri} and Lemma \ref{lem.relation} there are only two connected components $\tau\check{\mathcal{H}}_\infty,\check{\mathcal{H}}_{\infty}$ of $\mathbb{C}^*\setminus\check{\mathcal{C}}_{\lambda}$, which are punctured neighborhoods of $0$ and $\infty$ respectively. Thus $\mathcal{O}$ is either ${{{{\check{\mathcal{D}}}}}}_m$ or $\Tilde{{{{{\check{\mathcal{D}}}}}}}_m$ and $\Phi^{bit}_{m}$ or $\Tilde{\Phi}^{bit}_{m}$ is well-defined on $\mathcal{O}$. But this contradicts Lemma \ref{lem.proper.bit} since $\partial\mathcal{O} = \check{\mathcal{I}}_m$. So $\overline{\check{\mathcal{I}}_m}$ is a closed curve separating $0,\infty$. Now we claim that $\check{\mathcal{B}}_m$ is invariant under $\tau:c\mapsto1/c$. Suppose not, then $\tau(\check{\mathcal{B}}_m)\cap\check{\mathcal{B}}_m = \emptyset$. In particular, $\tau(\check{\mathcal{I}}_m)\cap\check{\mathcal{I}}_m = \emptyset$ and both their closures separate $0,\infty$. Therefore $\mathbb{C}^*\setminus(\overline{\check{\mathcal{I}}_m\cup\tau(\check{\mathcal{I}}_m}))$ has a connected component $\mathcal{V}$ which do not intersect $\tau\check{\mathcal{H}}_{\infty},\check{\mathcal{H}}_{\infty}$. By MSS $J$-stability theorem, $g_c$ is stable on $\mathcal{V}$,  hence $\check{\mathcal{I}}_m,\tau(\check{\mathcal{I}}_m)$ are in fact in the same parabolic component $\check{\mathcal{B}}_m$, a contradiction. Since $\check{\mathcal{B}}_m$ is invariant under $\tau$, so is $\overline{\check{\mathcal{B}}_m}$. Write $\partial\check{\mathcal{B}}_m = c_0\cup\mathcal{K}_0\cup\mathcal{K}_{\infty}$, where $c_0$ is the end point of $\check{\mathcal{I}}_m$, $\mathcal{K}_0$ is the connected component of $\partial\check{\mathcal{B}}_m\setminus\{c_0\}$ contained in the bounded component of $\mathbb{C}\setminus\overline{\check{\mathcal{I}}_m}$ and $\mathcal{K}_\infty$ the one contained in the unbounded component of $\mathbb{C}\setminus\overline{\check{\mathcal{I}}_m}$. Clearly $\mathcal{K}_0\subset\partial\tau\check{\mathcal{H}}_{\infty}\setminus\partial\check{\mathcal{H}}_{\infty}$ and $\mathcal{K}_\infty\subset\partial\check{\mathcal{H}}_{\infty}\setminus\partial\tau\check{\mathcal{H}}_{\infty}$. By the relation $\frac{1}{c}g_c(cz) = g_{1/c}(z)$ we conclude that $\tau(\mathcal{K}_0) = \mathcal{K}_{\infty}$ and $\tau(\mathcal{K}_\infty) = \mathcal{K}_{0}$, hence $\tau(c_0) = c_0$, $c_0 = \pm1$, while $1\in\check{\mathcal{B}}_0$, so $c_0 = -1$. Thus by Lemma \ref{lem.relation}, $\sigma\iota^{-1}(-1) = 0$ and $\sigma\iota^{-1}(\check{\mathcal{I}}_m)$ is a closed curve separating $\mathbb{C}$ into two connected components, each of which intersects $\sigma(\hat{\mathcal{H}}_{\infty}\cup\tau\hat{\mathcal{H}}_{\infty})$, the complementary of the connected locus for the family $f_a$. But this contradicts Proposition \ref{prop.zakeri.parametri}, which says that $\sigma(\hat{\mathcal{H}}_{\infty}\cup\tau\hat{\mathcal{H}}_{0})$ has only one connected component.
\end{proof}

\begin{corollary}\label{cor.unique.B_m}
For $0\leq m\leq q-1$, the $m$-component for the family $g_{c}$ is unique.
\end{corollary}
\begin{proof}
The case $m=0$ has already been justified in Corollary \ref{cor.unique.B0}. We prove for $m\geq 1$. For every $0\leq k\leq q-1$ pick a $k$-component. By the above proposition and Lemma \ref{lem.double.para.boudary}, $\mathcal{Z} = \bigcup_{k=0}^{q-1}\overline{\check{\mathcal{I}}_k}$ is a simple closed curve surrounding 0. If for some $m\geq 1$ there exists another $m$-component $\check{\mathcal{B}}'_m$ with correspoinding $\check{\mathcal{I}}'_m$, then by the above proposition, $\mathbb{C}^*\setminus(\mathcal{Z}\cup\overline{\check{\mathcal{I}}'_m})$ has a component which do not intersect $\check{\mathcal{H}}_{\infty},\tau\check{\mathcal{H}}_\infty$ while it intersects $\partial\check{\mathcal{C}}_{\lambda}$, a contradiction.
\end{proof}

\subsection{Parametrizations transferred for family $f_a$}\label{subsec.summary.fa}

The first thing to do here is to find a dynamically defined curve $\mathcal{I}\subset Per_1(e^{2\pi i \frac{p}{q}})$ symmetric with respect to $a\mapsto -a$ linking $-\sqrt{3\lambda},\sqrt{3\lambda}$, so that on $\mathbb{C}\setminus \mathcal{I}$ we can define the two critical points of $f_a$ such that they vary analytically for $a\in \mathbb{C}\setminus\mathcal{I}$.

$\sigma\iota^{-1}(\overline{\bigcup_{m=0}^{q-1}\check{\mathcal{I}}_m)}$ is a "good" candidate, but it is not necessarily symmetric with respect to $a\mapsto -a$. In order to solve this problem, let us first notice that if $q$ is odd, then $c=-1$ is double parabolic; if $q$ is even, then $c=-1\in \check{\mathcal{D}}_{\frac{q}{2}}$. Let $\rho = \mathfrak{Re}\{\phi_{-1}(g_{-1}^{\frac{q}{2}})(-1)\}$, then $0\textless \rho\textless 1$. Let $\Lambda\subsetneq\Omega^0_0\subset B_0^*(0)$ be the petal of $P = P_\lambda$ of level $\rho$. Since
\[\Phi^{bit}_{\frac{q}{2}}:\check{\mathcal{D}}_{\frac{q}{2}}\longrightarrow B^*_{\overline{\frac{q}{2}+p}}(0)\setminus \overline{{\Omega}^{-\frac{q}{2}-1}_{\overline{\frac{q}{2}+p}}},\,\,c\mapsto h_c(g_c(c))\]
is an isomorphism (Lemma \ref{lem.proper.bit}),  $(\Phi^{bit}_{\frac{q}{2}})^{-1}(P^{-\frac{q}{2}+1}(\partial\Lambda))$ is a curve linking the two double parabolic parameters on $\partial\check{\mathcal{B}}_{\frac{q}{2}}$. Let $\check{\gamma}$ be the subcurve of it linking the double parabolic parameter on $\partial\check{\mathcal{B}}_{\frac{q}{2}}\cap\partial\check{\mathcal{B}}_{\frac{q}{2}-1}$ and $c=-1$.

Let $\mathcal{Z}= \overline{\bigcup_{m =0}^{\lfloor\frac{q}{2}\rfloor} \check{\mathcal{I}}_m}$ if $q$ is odd or $\mathcal{Z}= \overline{\bigcup_{m =0}^{\lfloor\frac{q-1}{2}\rfloor} \check{\mathcal{I}}_m}\cup\check{\gamma}$ if $q$ even. Then $\mathcal{Z}$ is a curve linking $c=1$ and $c=-1$. Now $\iota^{-1}(\mathcal{Z})$ (in the $s$-plane, recall in (\ref{eq.para.s})) has two connected component. Take the one containing $s=-1$ and denote its image under $\sigma$ by $\mathcal{G}$. Set $\mathcal{I} = \mathcal{G}\cup-\mathcal{G}$. Thus $\mathcal{I}$ is a curve passing $a = 0$, symmetric under $a\mapsto -a$. Set $\mathcal{I}_m := \mathcal{I}\cap \mathcal{B}_m$. 

\begin{remark}\label{rem.checkI.and.I}
When $q$ is even and $m=\frac{q}{2}$, we have $\mathcal{I}_m\cap\sigma\iota^{-1}(\check{\mathcal{I}}_m) = \emptyset$. To see this, it suffices to prove that $\iota\sigma^{-1}(\mathcal{I}_m)\cap{\check{\mathcal{B}}_m}\cap\check{\mathcal{I}}_m = \emptyset$. Indeed, by construction $\iota\sigma^{-1}(\mathcal{I}_m)\cap{\check{\mathcal{B}}_m} = \check{\gamma}\cup\tau\check{\gamma}$ ($\tau:c\mapsto1/c$) and $\check{\gamma}\cap\check{I}_m=\emptyset$. To see that $\tau\check{\gamma}\cap\check{I}_m=\emptyset$, it suffices to justify that for $c\in\check{\gamma}$, $\mathfrak{Re}\{\phi_c(g_c^{\frac{q}{2}}(1))-\phi_c(g_c^{q}(c))\}\textless 0$. This is clear since $g_c^{\frac{q}{2}}(c)\in \Omega^0_{c,0}$, while $1\in\partial\Omega^0_{c,0}$.
\end{remark}

Let $a\mapsto\sqrt{a^2-3\lambda}$ be the inverse branch defined on $\mathbb{C}\setminus\mathcal{I}$ such that $(a-\sqrt{a^2-3\lambda})\to 0$ as $|a|\to+\infty$. Define $c_\pm(a) = \frac{-a\pm\sqrt{a^2-3\lambda}}{3}$ and let $v_{\pm}(a) = f_a(c_{\pm}(a))$. \\

\begin{proof}[Proof of Proposition \ref{Prop.grand2}]

It is just a summary of what we have obtained for the family $f_a$:
\begin{itemize}
    \item  For every $0\leq m\leq q$, there is a unique $m$-component $\mathcal{B}_m$ (Definition \ref{def.m-compo.f_a}). This is direct from Corollary \ref{cor.unique.B_m} and  the relation between families $g_c$ and $f_a$ (Lemma \ref{lem.relation}). Moreover, since $\mathcal{C}_{\lambda}$ is symmetric under $a\mapsto -a$, we have $\mathcal{B}_m = -\mathcal{B}_{q-m}$.
    
    \item For every $0\leq m\leq q-1$, there is a unique double parabolic parameter $\boldsymbol{\mathrm{a}}_m$ of $m$-type (Definition \ref{def.m-doubleparabolic}). Moreover $\boldsymbol{\mathrm{a}}_m\in  \partial\mathcal{B}_{m}\cap\partial\mathcal{B}_{m+1}$ (Lemma \ref{lem.double.para.boudary}).
    
    \item From Lemma \ref{lem.double.para.boudary}, \ref{lem.limit.Im} and Proposition \ref{prop.endpoints_specialcurve}, the special curve $\mathcal{I}$ defined above passes $\bigcup_{m=0}^q \mathcal{B}_m$ in the following order: \[\mathcal{B}_0, \boldsymbol{\mathrm{a}}_0, \mathcal{B}_1, \boldsymbol{\mathrm{a}}_1,\mathcal{B}_2,...........,\mathcal{B}_{q-2}, \boldsymbol{\mathrm{a}}_{q-2}, \mathcal{B}_{q-1}, \boldsymbol{\mathrm{a}}_{q-1}, \mathcal{B}_q.\]
     Moreover for $a\in\mathbb{C}\setminus \mathcal{I}$, $c_+(a)$ is always on the boundary of the maximal petal for $f^q_{a}|_{B^*_{a,0}(0)}$.
\end{itemize}
\end{proof}

\begin{figure}[H] 
\centering 
\includegraphics[width=0.55\textwidth]{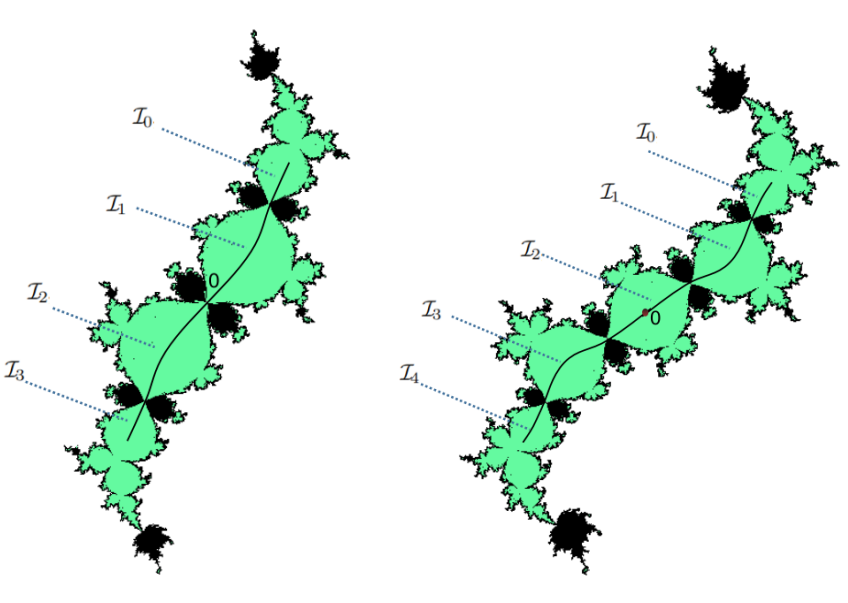} 
\caption{The curves $\mathcal{I}$ for $\mathcal{C}_\lambda$ with $\lambda = e^{\frac{2\pi i}{3}}$ and $e^{\frac{2\pi i}{4}}$.} 
\label{fig.main.parabolic} 
\end{figure}

The parametrization of $m$-components for the family $g_c$ (recall (\ref{eq.parametri})) can be transferred by $\sigma\iota^{-1}$ to the family $f_a$. More precisely:

\begin{itemize}
    \item For $m=0$, define $\Psi^{adj}_0:\mathcal{B}_m\setminus\mathcal{I}_0\longrightarrow B^*_p(0)$ by $\Psi^{adj}_0 = \Phi^{adj}_0\circ\iota\circ\sigma^{-1}$, where $\sigma^{-1}$ the inverse branch such that $\iota\circ\sigma^{-1}(\mathcal{B}_m\setminus\mathcal{I}_0) = \check{\mathcal{D}}_0$. 
    \item For $1\leq m\textless\lfloor \frac{q}{2}\rfloor$, define
$\mathcal{D}_m = \sigma\iota^{-1}(\check{\mathcal{D}}_m)$ and $\tilde{\mathcal{D}}_m = \sigma\iota^{-1}(\tilde{\check{\mathcal{D}}}_m)$, where $\iota^{-1}$ is the inverse branch such that $\sigma\iota^{-1}(\check{\mathcal{B}}_m) = \mathcal{B}_m$. Define $\Psi^{bit}_m:\mathcal{D}_m\longrightarrow B^*_{\overline{m+p}}(0)\setminus \overline{{\Omega}^{-s}_{\overline{m+p}}}$ by $\Psi^{bit}_m = \Phi^{bit}_m\circ\iota\circ\sigma^{-1}$ and $\tilde{\Psi}^{bit}_m:\tilde{\mathcal{D}}_m\longrightarrow B^*_{\overline{p-m}}(0)\setminus \overline{{\Omega}^{l}_{\overline{p-m}}}$ by $\tilde{\Psi}^{bit}_m = \tilde{\Phi}^{bit}_m\circ\iota\circ\sigma^{-1}$, where $\sigma^{-1}$ is the inverse branch such that $\iota\sigma^{-1}({\mathcal{B}}_m) = \check{\mathcal{B}}_m$.

    \item For $m = \frac{q}{2}$, $\mathcal{B}_m$ is divided by $\sigma\iota^{-1}(\check{I}_m)$ into two components. Let $\mathcal{D}_m$ be the one that is contained in a connected component of $\mathcal{B}_m\setminus\mathcal{I}_m$ (Remark \ref{rem.checkI.and.I}). Define similarly $\Psi^{bit}_m$ on $\mathcal{D}_m$.
\end{itemize}

From Proposition \ref{prop.injec.parame.bit}, we have

\begin{proposition}\label{prop.summary.parametri}
For $1\leq m\leq \lfloor \frac{q}{2}\rfloor$, $\Psi^{bit}_m,\tilde{\Psi}^{bit}_m$ are isomorphisms. For $m = \frac{q+1}{2}$, $\Psi^{bit}_m$ is an isomorphism.
\end{proposition}

Capture components can also be parametrized by locating $v_-(a)$:
\begin{proposition}\label{prop.paramet.cap}
Let $\mathcal{U}_k$ be a capture component of depth $k\geq 1$. For $a\in \mathcal{U}_k$, suppose $f_a^k(U_a) = B_{a,l}^*(0)$, where $U_a$ is the Fatou component containing $c_-(a)$. Then $\Psi_{\mathcal{U}_k}:\mathcal{U}_k\longrightarrow B_l^*(0)$ defined by $a\mapsto h_a(f^k(c_-(a)))$ is an isomorphism, where $h_a$ is the conjugating map between $f^q_{a}|_{B_{a,l}^*(0)}$ et $P^q|_{B_{l}^*(0)}$.
\end{proposition}
\begin{proof}
The proof goes exactly the same as Proposition \ref{prop.injec.parame.bit}.
\end{proof}

\begin{proposition}\label{prop.summary.para.adj}
For $m=0$, $(\Psi^{adj}_0)^{-1}(\Omega_p^{-1})$ is a topological disk whose boundary is a piecewise smooth closed curve passing $\sqrt{3\lambda},2\sqrt{\lambda}$. Let $\mathcal{D}_0,\tilde{\mathcal{D}}_0$ be the two connected component of 
$(\mathcal{B}_0\setminus\mathcal{I}_0)\setminus\Psi^{adj}_0(\overline{\Omega^{-1}_p})$. Then $\Psi^{adj}_0|_{\mathcal{D}_0},\Psi^{adj}_0|_{\tilde{\mathcal{D}}_0}$
are isomorphisms with image $B^*_p(0)\setminus\overline{\Omega^{-1}_p}$. See Figure \ref{fig.adj.component}.
\end{proposition}
\begin{proof}
First of all $\Psi^{adj}_0(\Omega_p^{-1})$ is not empty. Indeed, there is a holomorphic motion induced by $\phi^{-1}_a \circ\phi_{a_0}$ of $\overline{\Omega_{a,p}^{-1}}$ in a small neighborhood of $a_0=2\sqrt{\lambda}$. Notice that if $a_0=2\sqrt{\lambda}$, then $v_-(a_0) = 0$. Let $z_h\in\partial{\Omega_{a,0}^{-q}}$ be such that $\phi_{a_0}(z_h) = 1/h\in\mathbb{R}$. Apply Rouché's Theorem to $F(a,h) = v_-(a) - f^{1-q}_a\phi^{-1}_a\phi_{a_0}(z_h)$ for $h$ near 0 (take $f^{1-q}_a$ to be the inverse branch of $f^{q-1}_a|_{\overline{\partial\Omega^{-1}_{a,p}}}$), there is a sequence of $a_n\in\mathcal{B}_m\setminus\mathcal{I}_0$ converging to $2\sqrt{\lambda}$ such that $v_-(a_n)\in\overline{\Omega_{a,p}^{-1}}$, i.e. $a_n\in(\Psi^{adj}_0)^{-1}(\partial\Omega^{-1}_p)$ and $a_n\to2\sqrt{\lambda}$. So $\Psi^{adj}_0(\Omega_p^{-1})\not=\emptyset$. 

Next we investigate the end points of $\partial[(\Psi^{adj}_0)^{-1}(\Omega_p^{-1})]$. Notice that by properness of $\Psi^{adj}_0$ (Lemma \ref{lem.proper.bit}) and stability of Fatou coordinate, \[\partial[(\Psi^{adj}_0)^{-1}(\Omega_p^{-1})]\subset\{\boldsymbol{\mathrm{a}}_0,2\sqrt{\lambda}\}\cup(\Psi_0^{adj})^{-1}(\partial\Omega_p^{-1})\cup\mathcal{I}_0.\] 
By Proposition \ref{prop.injec.parame.adj} and a quasiconformal deformation argument, $\partial[ (\Psi_0^{adj})^{-1}(\partial\Omega_p^{-1})]$ is the closure of the union of two curves $\gamma_1,\gamma_2$ having $2\sqrt{\lambda}$ as a common end point. We parametrize $\gamma_1,\gamma_2$ by $I(a) = \mathfrak{Im}[\phi_(c_-(a))]$. By the above analysis, we see that as $|I_a|\to\infty$, $a\to2\sqrt{\lambda}$. We want to prove that as $|I_a|\to 0$, $a\to \sqrt{3\lambda}$. By Lemma \ref{lem.cr1=cr2}, it suffices to show that $a$ do not accumulate at $\partial\mathcal{B}_0$. Suppose the contrary. If $a\in\gamma_1$ accumualtes at $a_0\in\partial\mathcal{B}_0\setminus\mathcal{A}_{p/q}$, then by stability of Fatou coordinate, $a_0 = 2\sqrt{\lambda}$, then $\gamma_1$ surround a topological disk $U$ such that $\overline{\Psi^{adj}_{0}(U)}\cap\partial\mathcal{B}_0 = \{2\sqrt{\lambda}\}$. But $\Psi^{adj}_0(\partial U\setminus\{2\sqrt{\lambda}\}) = \Psi^{adj}_0(\gamma_1)$ is a semi-arc of $\partial\Omega^{-1}_{p}$, which does not separate $B_p^*(0)$, a contradiction. So it remains to exclude the case where both $\gamma_1,\gamma_2$ land at $\boldsymbol{\mathrm{a}}_0$ as $|I(a)|\to\infty$. Suppose we have this, Then $\gamma_1\cup\gamma_2$ bounds a topological disk $V$ such that $V\setminus\mathcal{I}_0$ is sent conformally onto $\Omega_p^{-1}$ by $\Psi^{adj}_0$ (Proposition \ref{prop.injec.parame.adj}). Since $\partial V\cup\mathcal{I}_0$ is locally connected, $(\Psi^{adj}_0)|_{\Omega_p^{-1}})^{-1}$ can be continuously extended to $\partial\Omega^{-1}_p$ and in particular $(\Psi^{adj}_0)^{-1}(0) = 2\sqrt{\lambda}$. On the other hand, by Lemma \ref{lem.proper.bit}, when $a\in \mathcal{B}_0\setminus\mathcal{I}_0$ tends to $\boldsymbol{\mathrm{a}}_0$, $\Psi_0^{adj}(a)\to\partial{B}^*_p(0)$. Hence $(\Psi^{adj}_0)^{-1}(0) = \boldsymbol{\mathrm{a}}_0$, a contradiction.

The rest of the proposition is immediate by Lemma \ref{lem.proper.bit} and Proposition \ref{prop.injec.parame.adj}. 
\end{proof}

\begin{figure}[H] 
\centering 
\includegraphics[width=0.5\textwidth]{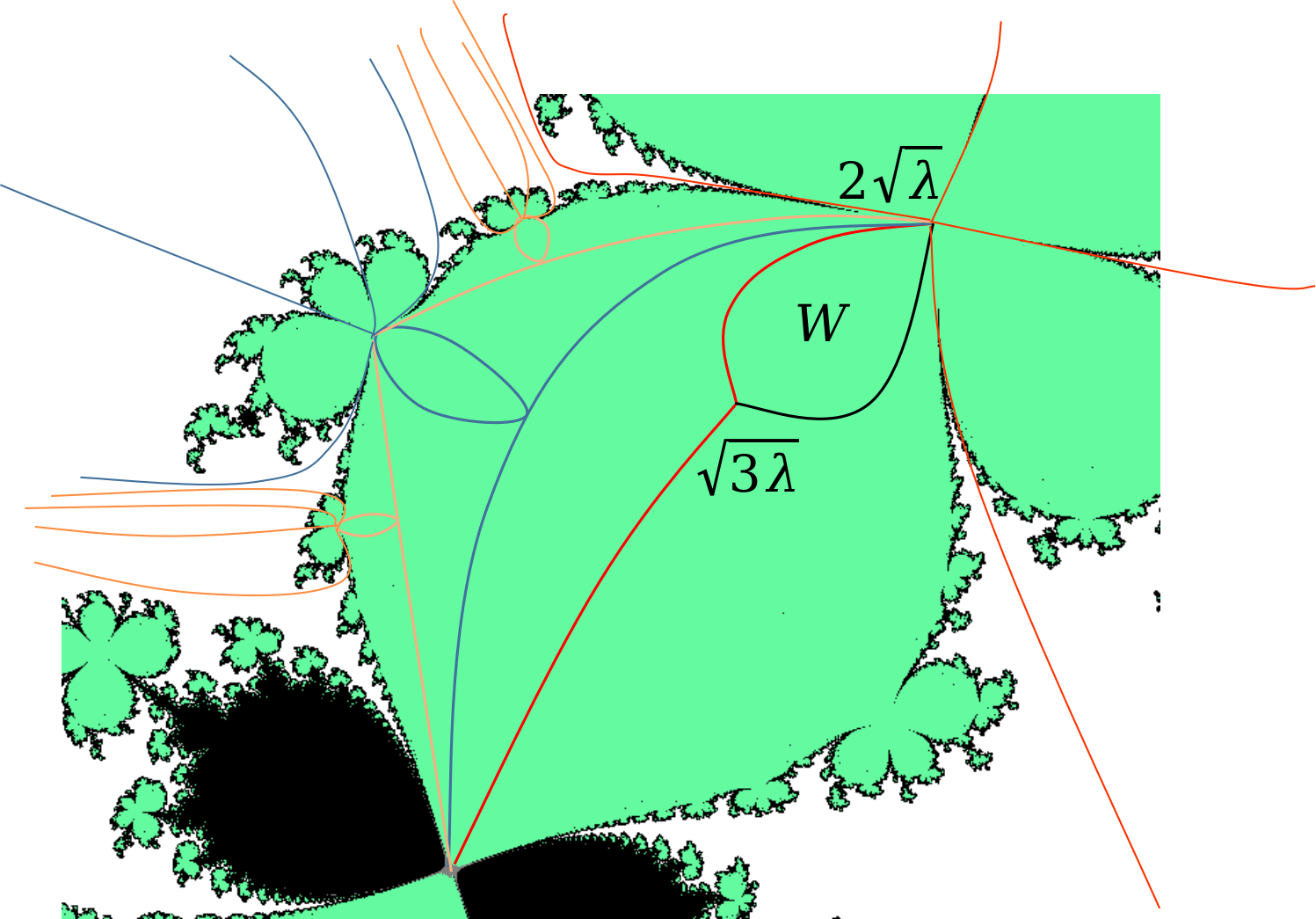} 
\caption{A zoom of $\mathcal{B}_0$ in $\mathcal{C}_{\lambda}$ with $\lambda = e^{2\pi i\frac{1}{4}}$. The union of curves in red, blue and orange is equipotentials and external rays of level 0,1,2 in $\mathcal{D}_0$. $W=(\Psi^{adj}_0)^{-1}(\Omega^{-1}_p)$. $a=\sqrt{3\lambda}$ is the parameter such that $c_+(a) = c_-(a)$; $a=\sqrt{3\lambda}$ is the parameter such that $v_-(a) = 0$.}  
\label{fig.adj.component} 
\end{figure}

\begin{figure}[H] 
\centering 
\includegraphics[width=0.8\textwidth]{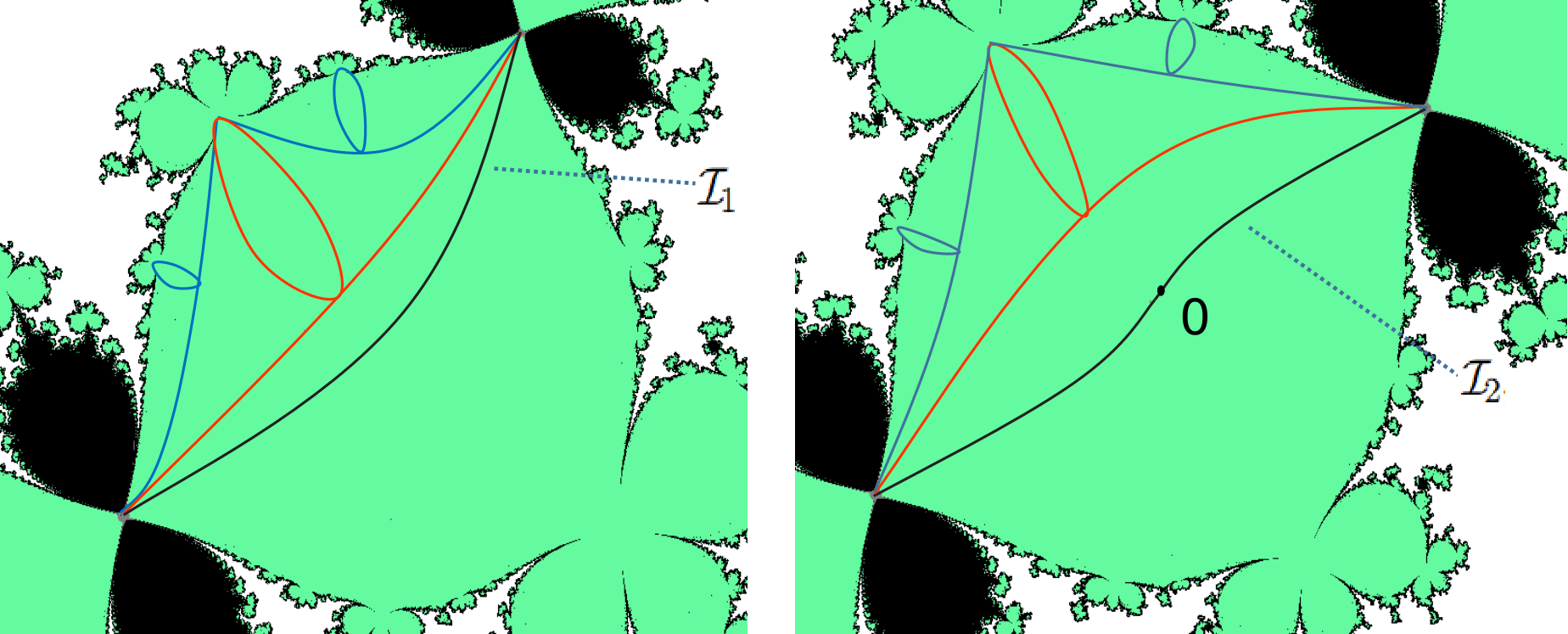} 
\caption{Equipotentials in $\mathcal{D}_1\subset\mathcal{C}_{\lambda}$ and $\mathcal{D}_2\subset\mathcal{C}_{\lambda}$, $\lambda = e^{2\pi i\frac{1}{4}}$} 
\label{fig.main.parabolic} 
\end{figure}

\subsection{Landing properties at double parabolic parameters}\label{subsec.landing_double_parabolic}

\begin{proposition}\label{prop.landing.four}
Let $\boldsymbol{\mathrm{a}}_m\in\mathcal{A}_{p/q}$ be the double parabolic parameter of $m$-type. Then among all the external rays with angles in $\bigcup_k\Theta_k$, there are exactly four landing at $\boldsymbol{\mathrm{a}}_m$.
\end{proposition}

\begin{proof}
First we prove the existence of four such rays. For $0\leq m\leq q$, let $\mathcal{O}_m\subset(\mathbb{C}\setminus\mathcal{A}_{p/q})$ be the set of parameters such that the dynamical external rays with angles $\Theta_m$ (recall in Definition \ref{def.rotation.angle}) land at 0. Then $\mathcal{O}_m$ is open. Its boundary is contained in $\bigcup_{\theta_i\in\Theta_m}\overline{\mathcal{R}_{\infty}(\theta_i)}$. From the discussion in \ref{subsec.existence.compo}, we have ${\mathcal{B}}_m\subset \mathcal{O}_m$ and $\mathcal{B}_m\cap\mathcal{O}_k = \emptyset$ for $k\not=m$. Therefore for $0\leq m\leq q-1$, there are at least four parameter external rays $\mathcal{R}_{\infty}(t_{1,2})$ with $t_{1,2}\in\Theta_m$ and two rays $\mathcal{R}_{\infty}(t'_{1,2})$ with $t'_{1,2}\in\Theta_{m+1}$ landing at the double parabolic parameter of $m$-type. Moreover both $\mathcal{R}_{\infty}(t_{1,2})$ and $\mathcal{R}_{\infty}(t'_{1,2})$ separate $\mathcal{B}_m,\mathcal{B}_{m+1}$.

The harder part is the uniqueness. Since the parametrisation $\Phi_{\infty}$ is of degree 3, there are in total $Q := 3q(q+1)$ parameter external rays whose angles belong to $\bigcup_k\Theta_k$. So in order to prove uniqueness, it suffices to find $Q-4q$ rays not landing at $\mathcal{A}_{p/q}$ among these $Q$ rays. By Lemma \ref{lem.landing.boundary.H00}, in $Per_1(0)$ there are $Q-4q$ rays landing at $\mathcal{C}_0\setminus\overline{\mathcal{H}^0_0}$. Let $A$ be the set of landing points of these $Q-4q$ rays. Write $A = \text{Mis}\cup \text{Par}$ where "Mis" (resp. "Par") means that $a_0$ is Misiurewicz (resp. parabolic).

\begin{claim*}
Let $a_1,a_2\not\in\mathcal{H}$ be two different geometrically finite parameters. let $a'_1,a'_2\in Per_1(e^{2\pi i\frac{p}{q}})$ be their pinching limit. If $a'_1,a'_2\not\in\mathcal{A}_{p/q}$, then $a'_1\not=a'_2$.
\end{claim*}

Admitting the claim, we finish the proof of the proposition: 
\begin{itemize}
    \item for $a_0\in \text{Mis}$, its pinching limit $a'_0\in Per_1(e^{2\pi i\frac{p}{q}})$ is also Misiurewicz (since $a_0\not\in \partial\mathcal{H}^0_0$ and $v_{0,a}$ is periodic, $v_{0,a_0}$ does not belong to any skeleton). Recall that the pinching deformation preserves external rays, thus the portrait at $v_{0,a_0}$ is the same as that of $v_{a'_0}$. Suppose there are $r$ external rays among the $Q-4q$ rays landing at $a_0$. Then by Lemma \ref{lem.landing.Misiurewicz}, there are also $r$ external rays with angles in $\bigcup_k\Theta_k$ landing at $a'_0$. 
    
     \item If $a_0\in \text{Par}$, then for the same reason its pinching limit $a'_0\in Per_1(e^{2\pi i\frac{p}{q}})$ is also parabolic. We want to prove that there exist two external rays with angles in $\bigcup_k\Theta_k$ landing at $a'_0$. Let $\theta_+,\theta_-$ be the two angles given in Lemma \ref{lem.landing.mandelbrot.Per0}. Since pinching preserves external rays, $R^{\infty}_{a'_0}(\theta_+),R^{\infty}_{a'_0}(\theta_-)$ land at the parabolic periodic point, bounding the critical value $v_{a'_0}$. By a simple plumbing surgery (see \cite{CuiTan}), for any neighborhood $U$ of $a_0$, there exists $a'_1\in U$ such that $R^{\infty}_{a'_1}(t_+),R^{\infty}_{a'_1}(t_-)$ land at the same repelling periodic point and bound $v_-({a'_1})$. By Lemma \ref{lem.landing.stable-region}, there is at least one ray $\mathcal{R}_{\infty}(\theta)$ with $\theta\in\bigcup_k\Theta_k$ landing at $a'_0$. Suppose the contrary that this is the only ray landing at $a'_0$. Then the holomorphic motion given by Lemma \ref{lem.landing.stable-region} implies that $R^{\infty}_{a}(\theta_+),R^{\infty}_{a}(\theta_-)$ land at the same repelling periodic point and bound $v_-({a})$ when $a$ is close to $\mathcal{R}_\infty({\theta})$. Since  $a\in\mathcal{R}_\infty({t})$ is equivalent to $v_-(a)\in R^\infty_a({t})$ (definition of parameter external rays), therefore $\theta\in(t_-,t_+)$. This contradicts Corollary \ref{cor.two_angles}.
\end{itemize}

To conclude, notice that by Lemma \ref{lem.landing.mandelbrot.Per0}, 
the $Q-4q$ rays in $Per_1(0)$ are decomposed into 
\[\left(\bigcup_{a_0\in\text{Mis}}\bigcup_{i=0}^r\mathcal{R}^0_\infty(s_i)\right)\cup\left(\bigcup_{a_0\in\text{Par}}(\mathcal{R}^0_\infty(t)\cup\mathcal{R}^0_\infty(t'))\right).\] 
While by the above discussion and the claim, we obtain $Q-4q$ external rays with angles in $\bigcup_m\Theta_m$. This finishs the proof.
\end{proof}

\begin{proof}[Proof of the claim]
Suppose the contrary that we have $a'_1 = a'_2$. Notice that in the dynamical plans of $a_1,a_2$, their critical values $v_{0,a_1},v_{0,a_2}$ are bounded respectively by wakes $W_1,W_2$ attached at the boundary of the immediate basin of 0. Since the pinching deformation preserves external rays, we conclude that the angles of the two rays defining $W_1$ are the same to those defining $W_2$. By Lemma \ref{lem.thurston.rigid}, this implies that $a_1,a_2$ belongs to the same wake (in the parameter plan) attached at $\partial\mathcal{H}^0_0$, which in particular is contained in a quadrant. Thus the angles of external rays landing at them are distinct, since the parametrisation $\Phi^0_\infty:\mathbb{C}\setminus\mathcal{C}_0\longrightarrow\mathbb{C}\setminus\overline{\mathbb{D}}$ is injective on each quadrant. Thus the pinching limits $a'_1,a'_2$ are distinct, a contradiction.
\end{proof}

\begin{definition}\label{def.angle.four.rays}
Let $\boldsymbol{\mathrm{a}}_m\in\mathcal{A}_{p/q}$ be of type $m$. The four rays landing at $\boldsymbol{\mathrm{a}}_m$ is separated by $\mathcal{I}$ into two groups $\alpha^+_m,{\beta}^+_{m}$ and  $\alpha^-_m,{\beta}^-_{m}$ with $\alpha^{\pm}_m\in\Theta_m,\beta^{\pm}_{m}\in\Theta_{m+1}$. They bound 2 open regions separated by $\mathcal{I}$. These two regions are called \textbf{double parabolic wakes} attached to $\boldsymbol{\mathrm{a}}_m$, denoted by $\mathcal{W}^{\pm}(\boldsymbol{\mathrm{a}}_m)$ respectively.
\end{definition}

\begin{figure}[H] 
\centering 
\includegraphics[width=0.6\textwidth]{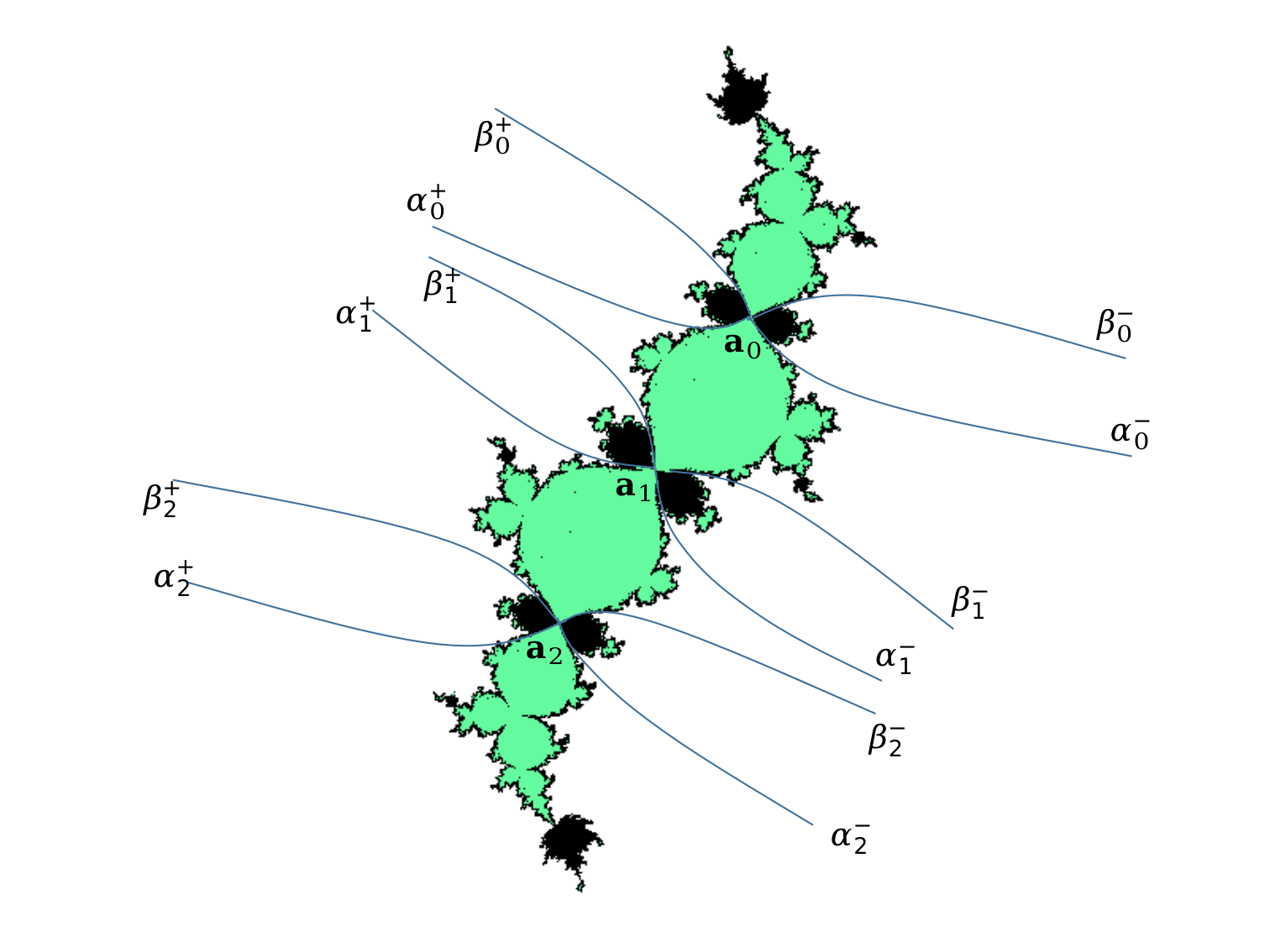} 
\caption{The external rays landing at double parabolic parameters for $\mathcal{C}_{\lambda}$, $\lambda = e^{2\pi i\frac{1}{3}}$.} 
\label{fig.preimage} 
\end{figure}

Denote by $\mathcal{S}_m$ the connected component of $\mathcal{V}$ containing $\mathcal{B}_m$, where
\begin{equation}
    \mathcal{V} := \mathbb{C}\setminus\bigcup_{m=0}^{q-1}(\overline{\mathcal{R}_\infty(\alpha^+_m)\cup\mathcal{R}_\infty({\alpha}^-_m)\cup\mathcal{R}_\infty(\beta^+_{m})\cup\mathcal{R}_\infty({\beta}^-_{m})})
\end{equation}

\begin{definition}\label{def.fundamental.Sm}
For $1\leq m\leq\lfloor\frac{q}{2}\rfloor$, denote by ${\mathcal{S}}^+_m$ the connected component of $\mathcal{S}_m\setminus\mathcal{I}_m$ intersecting $\mathcal{D}_m$, $\mathcal{S}^-_m$ the other component. For $m=0$, set ${\mathcal{S}}^\pm_0 = \mathcal{S}_0\setminus\mathcal{I}_0$.
\end{definition}

Now as a corollary of Proposition \ref{prop.landing.four}, we can give a description of the portrait at the parabolic fixed point $0$ for the family $f_a$:

\begin{corollary}\label{cor.double.para.wake}
Let $a\in\mathbb{C}\setminus\mathcal{A}_{p/q}$. If $a\in\mathcal{W}^\pm(\boldsymbol{\mathrm{a}}_m)$, 
the portrait of $f_a$ at the parabolic fixed point $0$ is $\Theta_m,\Theta_{m+1}$; if $a\in\mathcal{S}_m$, the portrait is $\Theta_m$. 
\end{corollary}
\begin{proof}
First we prove that when we go through $\mathcal{I}$ in the direction $\mathcal{B}_0,...,\mathcal{B}_q$, $\alpha^{\pm}_m$ is on the left-hand side of $\beta^\pm_m$. Suppose the contrary. Then By Proposition \ref{prop.landing.four}, there exists $a\in\mathcal{W}^\pm(\boldsymbol{\mathrm{a}}_m)\cap\mathcal{C}_\lambda$ such that $a$ and $\mathcal{B}_m$ (resp. $\mathcal{B}_{m+1}$) are contained in the same connected component of 
\[\mathbb{C}\setminus\bigcup_{\theta\in\hat{\Theta}_{m+1}}\overline{\mathcal{R}_\infty(\theta)},\,\, resp.\,\, \mathbb{C}\setminus\bigcup_{\theta\in\hat{\Theta}_{m}}\overline{\mathcal{R}_\infty(\theta)},\]
where $\hat{\Theta}_k= (\bigcup_i\Theta_i)\setminus\Theta_{k}$. Since the portrait at $z=0$ for $a$ in $\mathcal{B}_k$ is $\Theta_k$, Lemma \ref{lem.landing.stable-region} implies that the portrait at $z=0$ of $f_a$ can not contain $\Theta_m,\Theta_{m+1}$. But it can neither contain other $\Theta_k$ with rotation number $p/q$ for the same reason. But $z=0$ should at least admits a cycle of landing rays with rotation number $p/q$, a contradiction.

Let $\mathcal{O}^\pm_m\subset\mathcal{W}^\pm(\boldsymbol{\mathrm{a}}_m)$ be the collection of parameters whose portrait at $z=0$ is \textbf{exactly} $\Theta_m,\Theta_{m+1}$. Then clearly $\mathcal{O}^\pm_m$ is open.
By the above analysis, $\mathcal{O}^\pm_m\not=\emptyset$ and \[\overline{\mathcal{R}_\infty(\alpha^\pm_m)\cup\mathcal{R}_\infty({\beta}^\pm_m)}\subset\partial\mathcal{O}^\pm_m.\]
Suppose the inclusion above is strict. Then for any $a_0\in\partial\mathcal{O}^\pm_m\setminus\overline{\mathcal{R}_\infty(\alpha^\pm_m)\cup\mathcal{R}_\infty({\beta}^\pm_m)}$, there exists $t_{a_0}\in\Theta_m\cup\Theta_{m+1}$ such that $R^\infty_{a_0}(t_{a_0})$ crashes on $c_-(a_0)$. So $a_0\in\mathcal{R}_\infty(3t_{a_0})$. By stretching external rays we can show that for all $a\in\mathcal{R}_\infty(3t_{a_0})$, $R^\infty_{a}(t_{a_0})$ crashes on $c_-(a)$ and also $a\in\partial\mathcal{O}^\pm_m$. Therefore $\mathcal{R}_\infty(3t_{a_0})$ must land at $\boldsymbol{\mathrm{a}}_m$. By Proposition \ref{prop.landing.four}, $3t_{a_0}$ is one of $\alpha^\pm_m,\beta^\pm_m$, a contradiction. Hence $\mathcal{O}^\pm_m = \mathcal{W}^\pm(\boldsymbol{\mathrm{a}}_m)$.

Let $\tilde{\mathcal{O}}_m\subset\mathbb{C}$ be the collection of parameters whose portrait at $z=0$ is \textbf{exactly} $\Theta_m$. By a similar argument as above we can show that $\tilde{\mathcal{O}}_m = \mathcal{S}_m$. 

\end{proof}

The following finishes the proof of Proposition \ref{Prop.grand1}.

\begin{corollary}\label{cor.landing.rational.double.parabolic}
For any $m$, there are no other external rays with rational angles landing at $\boldsymbol{\mathrm{a}}_m\in\mathcal{A}_{p/q}$.
\end{corollary}
\begin{proof}
By Lemma \ref{lem.landing.rational.external}, it suffices to treat the case $3^nt = \alpha^+_m$ for some $n\geq 1$. Since $t$ is rational, the three external rays with angle $t$ in $Per_1(0)$ land at 3 different parameters $a_1,a_2,a_3$. Now we prove that $a_1,a_2,a_3$ are all Misiurewicz parameters. Suppose the contrary that $a_1$ is a parabolic parameter, then by hypothesis on $t$, $R_\infty^0(\alpha_m^+)$ lands at the parabolic point of $f_{0,a_1}$. Then by stability of rays landing at repelling point, there exists $\theta\in\Theta_m$ such that $\mathcal{R}_\infty^0(\theta)$ lands at $a_1$. This contradicts Lemma \ref{lem.landing.mandelbrot.Per0}. Notice that by the claim in Proposition \ref{prop.landing.four}, the corresponding pinching limits (clearly not double parabolic) $a'_1,a'_2,a'_3\in Per_1(e^{2\pi i\frac{p}{q}})$ $a'_1,a'_2,a'_3$ are distinct. By Lemma \ref{lem.landing.Misiur.parabo}, they admit in total three external rays with angle $t$ landing.
\end{proof}

\section{Dynamical graphes and puzzles}\label{sec.dym.graph}
This section mainly aims to construct admissible graph for family (\ref{eq.family_a}) that infinitely rings the free critical point $c_-(a)$ in order to apply Yoccoz' Theorem (Theorem \ref{thm.yoccoz}) (For basic knowledge of admissible graphs, we refer the readers to \cite[Appendix]{Roesch1}). More precisely, under the following assumption on $f_a$
\begin{assum*}[$\Diamond$]
$a\in\mathcal{C}_\lambda\setminus\mathcal{H}^\lambda$ is not Misiurewicz parabolic and is not contained in any double parabolic wake.
\end{assum*}
we prove the following technical lemma (Proposition \ref{prop.summary.infinite.ringed}):
\begin{lem*}
If $a$ satisfies Assumption ($\Diamond$), then there is an admissible graph infinitely ringing $c_-(a)$.
\end{lem*}

Let us fix some notations. Let
$f_a\in \mathcal{C}_{\lambda}\setminus\mathcal{H}^\lambda$ and $f_a$ not Misiurewicz parabolic. In the sequel always fix such an $f_a$ and omit the index $a$. The critical point in the parabolic basin of 0 will be denoted by $c_+$, the free critical point by $c_-$. Let $B^*_0(0)$ be the immediate basin containing $c_+$ and $B^*_1(0),...,B^*_{q-1}(0)$ the other immediate basins in cycle order. Let $\{R^\infty(\theta_i);\,\theta_i\in\Theta\}$ be a cycle of external rays landing at 0 and  $R^\infty(\theta_0^{\pm})$ be the two external rays landing at 0 and bounding $B^*_0(0)$ with $\theta_0^{\pm}\in\Theta$. These $q$ rays together with $\{0\}$ divide $\mathbb{C}$ into $q$ sectors and
let $S_i$ be the sector containing $B^*_i(0)$. Let $\Omega\subset B^*_0(0)$ be the maximal petal, $\tilde{\Omega} = f^q(\Omega)$. For any $r\textgreater 1$, denote by $E^\infty(r)$ the external equipotential of potential $r$.

\subsection{Local connectivity at $f^{-n}(0)$}
Take some $r_0\textgreater 1$. Define the graph of depth $n$ by
\begin{equation}\label{eq.graph.Y}
    Y_0 = \left(\bigcup_{i=1}^{q}(f^{i}(\partial\tilde{\Omega})\cup R^\infty(\theta_i))\right)\cup E^\infty(r_0),\,\, Y_n = f^{-n}(Y_0).
\end{equation}

A puzzle piece $Q_n$ of depth $n$ is a connected component of $\mathbb{C}\setminus Y_n$ intersecting the Julia set. Let $Q^{\pm}_{n}$ be the puzzle piece whose boundary intersects $R^\infty(\theta_{\pm})$ respectively. 

\begin{proposition}\label{prop.localconnec.0}
Let $X_{\pm} = \bigcap_{k\geq 0}\overline{Q^{\pm}_{kq}}$. Then either $X_{\pm} = \{0\}$, or
it is a continuum and $f^q:X_{\pm}\longrightarrow X_{\pm}$ is a degree two ramified covering. Moreover in the second case, there are two cycles of external rays landing at 0 with angle cycles $\Theta,\Theta'$ and there exists $\zeta^{\pm}\in\Theta'$ such that $\gamma_\pm = R^\infty(\theta_0^{\pm})\cup R^\infty(\zeta^{\pm})\cup \{0\}$ separates $X_{\pm}\setminus\{0\}$ from $\overline{B^*_0(0)}\setminus\{0\}$.
\end{proposition}

\begin{proof}
We only prove for $X = X_{+}$ and the other is similar. Set $\theta = \theta^+_0$. First notice that since $f^q(\overline{Q_{(k+1)q}}) = \overline{Q_{kq}})$, hence $f^q(X) = X$. Suppose that $X \not= \{0\}$, then there exists an isomorphism $\phi:\mathbb{C}\setminus X\longrightarrow \mathbb{C}\setminus\overline{\mathbb{D}}$. Take a small enlargement $U$ of X with $U$ simply connected and open, such that $f^q: U'\longrightarrow U$ has degree $d\geq2$ and $U\setminus X$ does not contain $f^q(c_-)$, where $U'$ is the connected component of $f^{-q}(U)$ containing $X$. Hence $f^{-q}(X)\cap U' = X$. Let $W = \phi(U)$ and $W' = \phi(U')$: Then the map $F = \phi\circ f\circ\phi^{-1}$ is a ramified covering of degree $d$. By Schwarz reflection principle, $F$ extends holomorphically to a neighborhood of $\mathbb{S}^1$. Let $g = F|_{\mathbb{S}^1}$ Adapting the same strategy in the proof of \cite[Prop. 2.4]{Roesch}, we can show that $g:{\mathbb{S}^1}\longrightarrow{\mathbb{S}^1}$ is a degree $d$ covering and has $d-1$ fixed points. Since in the dynamical plane of $f$, there is an access $\delta\in\mathbb{C}\setminus X$ to $0$ fixed by $f^q$ (for example one may take $\delta$ to be one of the external rays landing at 0), then $\phi(\delta)$ lands at $\mathbb{S}^1$ and gives a fixed point of $g$ (see \cite{GoldMil}). Thus $d = 2$.

Now we prove that if $X\not=\{0\}$, there is another cycle of external rays landing at 0 and satisfying the conclusion stated in the proposition. Let $R^\infty(\theta_k)$ be the external ray landing at $\partial B^*_0(0)\setminus\{0\}$ involved in $\partial Q_{kq}$. Then $3^{q}\theta'_{k+1} = \theta'_k$ since $f^q(R^\infty(\theta_{k+1}))=R^\infty(\theta_k)$. Clearly $\theta_k$ has a limit $\theta'$ when $k\to\infty$ since $\theta_k$ is monotone. Thus $3^q\theta = \theta$ and $R^\infty(\theta')$ enters every puzzle piece $Q_{kq}$. This implies that $R^\infty(\theta')$ lands at a fixed point $x\in X$ of $f^q$. Notice that $R^\infty(\theta')$ is also fixed by $f^q$, thus $x$ corresponds to the unique fixed point of $g$, i.e. $x = 0$. Since $X\not = \{0\}$ and is contained in every $Q_{kq}$, there exists $\eta$ between $\theta,\theta_k$ for all $k$ such that $R^(\eta)$ landing at $x'\in X$ with $x\not = 0$. Hence $\theta'\not = \theta$. Clearly $\gamma$ separates $X_{\pm}\setminus\{0\}$ from $\overline{B^*_0(0)}\setminus\{0\}$.
\end{proof}

We get immediately 
 \begin{corollary}\label{cor.loc.connec.inverse-orbit0}
 $\partial B^*_0(0)$ is locally connected at 0 and at its inverse orbit.
 \end{corollary}

\subsection{Wakes attached to $\partial B^*_i(0)$}\label{subsec.wakes-to-B}
Let $T_0 = \left(\bigcup_{i=1}^{q}(f^{i}(\tilde{\Omega})\cup R^\infty(\theta_i))\right)\cup\{0\}$ and $T_n$ the connected component of $f^{-n}(T_0)$ containing 0. Let $\Delta_{\pm}$ be the two unbounded connected components of $S_0\setminus T_q$, such that $R^\infty(\theta_\pm)\subset\partial\Delta_{\pm}$. Hence any $z\in \partial B^*_0(0)$ which is not in the inverse orbit of 0 has a unique dyadic representation $(\epsilon_n)_{n\geq 0}$ encoding its orbit position under $f^q$. More precisely, $\epsilon_n = 1$ if $f^{nq}(z)\in \Delta_+$, $\epsilon_n = 0$ if $f^{nq}(z)\in \Delta_-$. For $z$ in the inverse orbit of 0, we take the convention that $\epsilon_n$ is zero for all but finitely many $n$.
Then the following mapping is well defined:
\[\kappa:\partial  B^*_0(0)\longrightarrow \mathbb{S}^1,\,\, z\mapsto (\epsilon_n)_{n\geq 0}\] 

For $i\geq 1$, $\kappa$ is naturally extended to $\partial  B^*_i(0)$ by $\kappa(z) = \kappa(w)$ where $w\in\partial B^*_0(0)$ is the first iterated image of $z$. Notice that there are exactly $q$ external rays involved in $T_n$ at any $u\in \bigcup_i\partial B^*_i(0)$ which is in the inverse orbit of 0. Suppose $u\not=0$, then among these $q$ rays there are two of them bounding a open region separating all the other $q-2$ rays with $\bigcup_i\partial B^*_i(0)$. Let $U(u)$ be the closure of this region. Now $\kappa$ is extended to $\mathbb{C}\setminus\bigcup_i B^*_0(0)$ in the following way: if $z$ belongs to some $U(u)$, then $\kappa(z) := \kappa(u)$ (which is dyadic); if $z\in T_0$, then $\kappa(z) := 0$; else, for every $n\geq 0$, let $U_n$ be the component of $\mathbb{C}\setminus T_n$ containing $z$, then set $z\in U(z) := \bigcap_n\overline{U_n}$ and define $\kappa(z) = \kappa(u)$ for any $u\in (\bigcup_i \partial B^*_0(0))\cap U(z)$. Notice that in the second case $\kappa(z)$ does not depend on the choice of $u$. $U(z)$ is called a \textbf{wake} attached to $\partial B^*_i(0)$. Define the corresponding \textbf{limb} by $L(z) = U(z)\cap K_f$. Denote by $U(t)$ resp. $L(t)$ the union of $U(z)$ resp. $L(z)$ with $\kappa(z) = t$.

\begin{remark}\label{rem.union.wakes-to-immediat-bassin}
By construction, $\mathbb{C} = \bigcup_i B_{i}^*(0)\cup\bigcup_t U(t)$, $K_f = \bigcup_i B_{i}^*(0)\cup\bigcup_t L(t)$.
\end{remark}

\begin{lemma}
$L(0) = 0$ if and only if $f$ does not belong to wakes attached at double parabolic parameters.
\end{lemma}
\begin{proof}
If $L(0) = 0$ but $f$ belongs to some wake attached at a double parabolic parameter. Then by Corollary \ref{cor.double.para.wake}, there are two cycles of external rays landing at 0. Suppose the corresponding cycle of angles are $\Theta,\Theta'$ respectively, then there exists $\theta\in\Theta,\theta'\in\Theta'$ such that $R^\infty(\theta),R^\infty(\theta')$ bound a open sector $S$ not intersecting any immediate basin at 0, hence $S\cap K_f\subset L(0)$, contradicting $L(0) = 0$.

If $f$ does not belong to wakes at double parabolic parameters, then one may deduce by the same argument in the second part of the proof of Proposition \ref{prop.localconnec.0} that $L(0) = 0$.
\end{proof}

\subsection{Finding infinitely ringed puzzle pieces around $c_-$}\label{subsec.infini.ringed}
\begin{theorem}[Yoccoz]\label{thm.yoccoz}
Let $f:U'\longrightarrow U$ be a quadratic rational-like map and $x_0$ be its unique critical point. Let $x\in K_f$. For any admissible graph $\Gamma$ that rings $x_0$ and rings infinitely $x$, we have the following alternative: 
\begin{itemize}
    \item if the tableau of $x_0$ is $k$-periodic, then $f^k:P_{l+k}(x_0)\longrightarrow P_l(x_0)$ is quadratic-like for $l$ large enough. $Imp(x)$ is either $x$ or a conformal copy of $Imp(x_0)$, depending on whether the forward orbit of $x$ intersects $Imp(x_0)$.
    \item if the tableau of $x_0$ is not periodic, then $Imp(x) = x$.
\end{itemize}
\end{theorem}

\begin{observation*}
If $\Gamma$ rings infinitely $x_0$, then $\Gamma\cap K_f$ contains no parabolic cycle. Indeed, since a parabolic cycle (say period $k$) must attract a critical point, we may suppose that $f^{kn}(x_0)$ are close to each other when $n$ large enough. Fix such $n$, pick $n_i$ large enough such that $f^{nk}(x_0)$ is ringed at depth $n_i$, then $P_{n_i}(f^{nk}(x_0)$ is compactly contained in $f^{kN}(P_{n_i}(f^{nk}(x_0)) = P_{n_i-kN}(f^{(n+N)k}(x_0))$ for $N$ large enough. In particular $\Gamma\cap K_f$ does not contain the parabolic cycle.
\end{observation*}

Now we start to construct admissible graphs for $f = f_a$ satisfying Assumption ($\Diamond$) until the end of this subsection. Firstly we work in the quadratic model $P_\lambda(z) = e^{2\pi i\frac{p}{q}}z+z^2$ with $\lambda = e^{2\pi i\frac{p}{q}}$. Recall the notations at the beginning of Subsection \ref{subsec.parame.component}.

Fix $0\leq m\leq q-1$. Suppose $m = lp \pmod{q}$ with $1 \leq l\leq q$. Let $H$ be the connected component of $P_\lambda^{-2q+l}(\overline{{\Omega^{-q}_0}})$ contained in $\overline{B^*_{m}(0)}$. Notice that $\mathbb{C}$ is separated by the $q$ external rays landing at 0 into $q$ sectors $S_0,...,S_{q-1}$ (written in cyclic order) with $B^*_i(0)\subset S_i$. Since $P^q:B^*_m(0)\longrightarrow B^*_m(0)$ is of degree 2, the dynamics of $P^q|_{B^*_m(0)}$ can be identified with $P_1|_{B(0)}$ by their Fatou coordinates. Thus the construction of "jigsawed" internal rays for $P_1$ (\cite{Roesch}, \cite{runze}) can be transferred to $P_\lambda$. For $k\geq 1$, let $R(\theta_\pm)\subset B^*_m(0)$ be the internal ray with angle $\theta_\pm = \frac{\pm 1}{2^k-1}$. For $0\leq i\leq k-1$, define $R(2^i\theta_\pm) = f^{qi}(R(\theta_\pm))\setminus P^q(\mathring{H})$. The rays $R(\theta)$ with $\theta$ strictly pre-periodic to some $2^i\theta_\pm$ are then defined to be the iterated preimages of $R(2^i\theta_\pm)$ in $B_m^*(0)$. Set $\Theta_\pm = \{\theta;\,\exists i \text{ s.t. } 2\theta = 2^i\theta_\pm\}$. Define the equipotential of depth $n\geq0$ by 
\begin{equation}\label{eq.equi.dym}
    E^m_0 = \bigcup_{i=1}^{q}P_\lambda^{i}(\partial{H}),\,\, E^m_n =P_\lambda^{-n}(E^m_0).
\end{equation}
Define the union of internal rays of depth $n\geq 0$ by
\begin{equation}\label{eq.internal-ray.dym}
    R^m_0 = \bigcup_{i = 1}^{q} \left(\bigcup_{\theta\in\Theta_\pm}P_\lambda^{i}(R(\theta)\setminus{H})\right),\,\, R^m_n =P_\lambda^{-n}(R^m_0).
\end{equation}

Now return to $f$. We abuse the notations of $S_i,B_i^*(0)$, etc.. Suppose $c_-\in S_m$.  Since the dynamics on the immediate basins of $f$ is equivalent to that of $P_\lambda$, $E_0$ in (\ref{eq.equi.dym}) and $R_0$ in (\ref{eq.internal-ray.dym}) can be transferred to the immediate basins of $f$. Define similarly $E_n$, $R_n$ for $f$ to be the $n$-th preimage. Let $\theta\in \Theta_\pm$, then the landing point $x_\theta$ of $R(\theta)$ is pre-periodic, hence there is at least an external ray $R^\infty(t)$ landing at it. For each $x_\theta$ we choose a $R^\infty(t)$ and denote by $T$ the collection of these $t$ such that $3^qT\subset T$.

Pick $r\leq 1$. Define the graph of depth $n\geq 0$ by

\begin{equation}\label{eq.graph.Si}
    X^m_n = E^m_n\cup R^m_n\cup\left(\bigcup_{t\in T}R^\infty(3^it)\right)\cup E^\infty(r).
\end{equation}

Notice that if $m=0$, the graph above can also be written as

\begin{equation}\label{eq.graph.S0}
    X^0_0 = \left(\bigcup_{i=1}^{q}f^{i}(\partial{\Omega})\right)\cup\bigcup_{i = 1}^{q} \left(\bigcup_{\theta\in\Theta_\pm}f^{i}(R(\theta)\setminus{\Omega})\cup\bigcup_{t\in T}R^\infty(3^it)\right)\cup E^\infty(r)
\end{equation}
\[X^0_n = f^{-n}(X^\pm_0).\]
Recall that $\Omega\subset B^*_0(0)$ is the maximal petal for the return map $f^q|_{B^*_0(0)}$. 

\begin{figure}[H] 
\centering 
\includegraphics[width=0.5\textwidth]{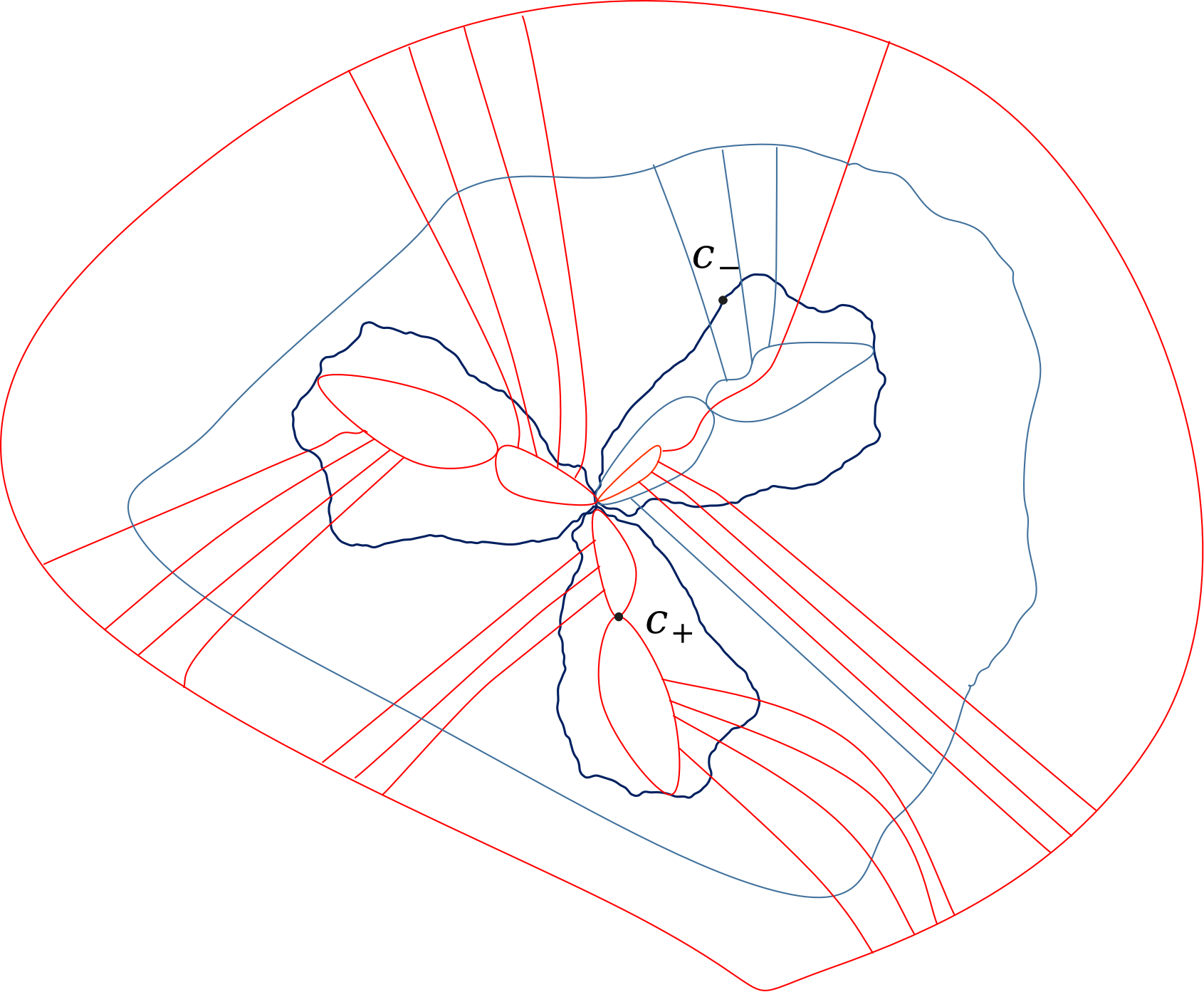} 
\caption{An admissible graph for $f\in Per_1(e^{\frac{2\pi i}{3}})$ with $m=1$ and $\theta$. The graph in red is $X^1_0$, the graph in blue is the first pre-image of $X^1_0$ that is not contained in $X^1_0$. For this graph, $c_-$ is ringed at depth 0.} 
\label{fig.preimage} 
\end{figure}

\begin{lemma}\label{lem.ringed}
If $z\in (K_f\setminus\bigcup_i B_i^*(0))\cap S_m$ satisfies $\kappa(z)\in[\frac{1}{4},\frac{1}{2})\cup(\frac{1}{2},\frac{3}{4}]$, then $z$ is ringed at depth 0 by one of the graphs (\ref{eq.graph.Si}) for all $k$ large enough. 
\end{lemma}
\begin{proof}
By construction, the annulus $P_0\setminus\overline{P_1}$ of depth 0 is non-degenerated if $P_1$ is a piece whose boundary contains two of the internal rays $R^0(2^j\theta_\pm+\frac{1}{2}),\,j=0,...,k-2$. The lemma follows by noticing that 
\[\theta_++\frac{1}{2}\textless2\theta_++\frac{1}{2}\textless...\textless2^{k-2}\theta_++\frac{1}{2}\]
\[\theta_-+\frac{1}{2}\textgreater2\theta_-+\frac{1}{2}\textgreater...\textgreater2^{k-2}\theta_-+\frac{1}{2}\]
and $2^{k-2}\theta_++\frac{1}{2}\textgreater\frac{3}{4}$ tends to $\frac{3}{4}$ as $k\to\infty$; $2^{k-2}\theta_-+\frac{1}{2}\textless\frac{1}{4}$ tends to $\frac{1}{4}$ as $k\to\infty$. 
\end{proof}

\begin{lemma}\label{lem.ringed.critical}
Suppose that $t_0\not = \frac{1}{2}$. Then the critical value $v_- = f(c_-)\in S_{\overline{m+p}}$ and $f^q(c)\in L(2t_0)\cap S_m$. Moreover if $z\in L(t_0)\cap S_m$ such that $f(z)$ is not in the puzzle piece (defined by (\ref{eq.graph.Si})) of depth 0 containing $v_-$, then $z$ is ringed at depth 0.
\end{lemma}
\begin{proof}
Suppose $v_-\in S_{m'}$ and let $L'(t_0)\cap S_{m'}$ be the limb containing $v_-$. Notice that $f^{-1}(L'(t_0))$ has only two connected components, one containing $c_-$. If $m=0$, then $m' = p$, otherwise $f^{-1}(L'(t_0))$ will have three components: two in $S_0$ and one in $S_{\overline{m'-p}}$. If $m\not = 0$, then $m' = m+p\pmod{q}$, otherwise $f^{-1}(L'(t_0))$ will have three components: one in $S_m$, one in $S_0$ and one in $S_{\overline{m'-p}}$. It is then clear that $f^q(c)\in L(2t_0)\cap S_m$.

Now suppose $z\in L(t_0)\cap S_m$ such that $f(z)$ is not in the puzzle piece of depth 0 containing $v_-$. Notice that in $L(t_0)\cap S_m$ there is a preimage of $\overline{\bigcup_i B^*_i(0)}$. Hence $z$ is bounded by a puzzle piece $P_1$ of depth 1 whose boundary does not intersect puzzle pieces of depth 0, since the internal/external rays in $\partial P_1$ are all stemming from $f^{-1}(\overline{\bigcup_i B^*_i(0)})\cap L(t_0)$, which are different from the rays involved in depth 0.

\end{proof}

In the next lemma we do \textbf{not} suppose $L(0) = 0$.
\begin{lemma}\label{lem.orbit.hit.boundary}
Suppose $z\in K_f$ is not in the inverse orbit of 0. If $f^n(z)$ eventually hits $\partial B^*_0(0)$. Then $z$ is infinitely ringed by one of the graphs (\ref{eq.graph.S0}) for all $k$ large enough. 
\end{lemma}

\begin{proof}
The proof goes the same as in \cite[Lem. 6.1]{Roesch}. The orbit of $z$ must satisfy one of the following conditions, and the result follows by Lemma \ref{lem.ringed}:
\begin{itemize}
    \item If there exists a subsequence $n_j$ such that $2^{n_j}\kappa(z)\to \frac{1}{4}$ resp. $\frac{3}{4}$, then $z$ is infinitely ringed by (\ref{eq.graph.Si}) defined by $\theta_-$ resp. $\theta_+$.
    \item Otherwise $z$ is ringed by (\ref{eq.graph.Si}) defined by $\theta_+$ and $\theta_-$.
\end{itemize}
\end{proof}

\begin{lemma}\label{lem.orbit.L1/2}
Suppose $z\in K_f$ is not in the inverse orbit of 0 and not in any basin of 0. If $z$ satisfies $f^n(z)\not\in L(t_0)$ for $n$ large enough, then $z$ is infinitely ringed by one of the graphs (\ref{eq.graph.S0}) for all $k$ large enough.
\end{lemma}
\begin{proof}
The proof goes the same as in \cite[Lem. 6.1]{Roesch} if $f^n(z)$ avoid $L(\frac{1}{2})$ for $n$ large enough. If $f^n(z)$ meets infinitely many times $L(\frac{1}{2})$, then we treat by two subcases:
\begin{itemize}
    \item If $f^n(z)\in L(\frac{1}{2})$ for $n$ large enough. Let $U$ be a connected component of $f^{-j}(\mathring{L(\frac{1}{2})})$ ($1\leq j\leq q$) contained in $\mathring{L(\frac{1}{2})}\cap S_0$ that intersects infinitely the orbit of $z$. Notice that $\mathring{L(\frac{1}{2})}\cap S_0$ is divided into $q-1$ connected components by $q-2$ external rays landing at $x_0$, the first preimage of 0 (other than 0) on $\partial B_0^*(0)$. Suppose $U$ belongs to one of these $q-1$ components $V$, which is bounded by $R^\infty(\eta_1),R^\infty(\eta_2)$. Then $U\subset\subset V$ and the angles of external rays involved in $\partial U$ are between $\eta_1,\eta_2$. Recall graph (\ref{eq.graph.Y}) and its corresponding puzzle pieces $Q_n$. Let the $Q^{\pm}_n$ be the two puzzle pieces of depth $n$ at $x_0$. Since $L(0) = 0$, all the angles of external rays involved in $Q^{\pm}_n$ converge to $\eta_1$ or $\eta_2$. Thus if we take $k$ large enough in $\theta_\pm = \frac{\pm 1}{2^k-1}$, $U$ will be included by a puzzle piece $P^\pm_1$ of depth 1 defined by graph (\ref{eq.graph.S0}) such that $P^\pm_1$ is compactly contained in some puzzle piece of depth 0.
    
    \item There is a subsequence such that $f^{n_j}(z)\not \in L(\frac{1}{2})$. By hypothesis, $z$ returns infinitely many times to $L(\frac{1}{2})\cap S_0$. Thus $\{f^n(z);\, n\geq 0\}\cap (L(\frac{1}{2})\cap S_0)^c$ will meet infinitely $(L(\frac{1}{4})\cup L(\frac{3}{4}))\cap S_0$. By Lemma \ref{lem.ringed}, if $f^n(z)$ is contained in $(L(\frac{1}{4})\cup L(\frac{3}{4}))\cap S_0$, then it is ringed at depth 0.
\end{itemize}
Thus in both cases, we have found a graph that rings infinitely $z$.
\end{proof}

\begin{proposition}\label{prop.ringed.infini_return}
Suppose $t_0\not =\frac{1}{2}$ and that $z\in K_f$ returns infinitely many times to $L(c_-) = L(t_0)\cap S_m$ but never hits $\partial B^*_m(0)$. Then for all $k$ large enough, there exists one of the two graphs (\ref{eq.graph.Si}) that rings infinitely $z_-$ 
\end{proposition}

\begin{proof}
For simplicity we only do the proof for the case $m=0$. The case $m\not = 0$ will be similar. We adapt a similar strategy in the proof of \cite[Lem. 8.2]{Roesch}.
By hypothesis there is a subsequence $\{f^{n_i}(z)\}_i\subset L(t_0)$. Take any $w\in\{f^{n_i}(z)\}_i$.
\begin{itemize}
    \item $t_0\in[\frac{1}{4},\frac{1}{2})$. Then by Lemma \ref{lem.ringed}, $w$ is ringed at depth 0 by graph (\ref{eq.graph.S0}) defined by $\theta_-$ for $k$ large enough.
    \item $t_0\in(\frac{1}{8},\frac{1}{4})$. By Lemma \ref{lem.ringed.critical}, the limb containing the critical value $v_-$ is $L(2t_0)\cap S_p$ with $2t_0\in(\frac{1}{4},\frac{1}{2})$. Take $k$ large enough and the corresponding graph $X^+_0$ such that $L(2t_0)\cap S_p$ is bounded by the puzzle piece of depth 0 whose boundary intersects $f(R^0(2^{k-2}\theta_+)),f(R^0(2^{k-1}\theta_+))$. Thus by Lemma \ref{lem.ringed.critical}, any $w$ satisfying \[\kappa(f(w))\in(0,2^{k-2}\theta_+)\cup(2^{k-1}\theta_+,1)\] 
    is ringed by graph (\ref{eq.graph.S0}) defined by $\theta_+$ at depth 0. Hold the same $k$ and take $\theta_- = \frac{-1}{2^k-1}$.  If $w$ satisfies
    \[\kappa(f(w))\in[2^{k-2}\theta_+,\frac{1}{2}+\theta_-),\]
    then by Lemma \ref{lem.ringed}, $f^j(f(w))$ is ringed at depth 0 by graph (\ref{eq.graph.S0}) defined by $\theta_-$, where $j\leq q-1$ is the smallest integer such that $f^j(f(w))\in S_0$. Now suppose that $w$ satisfies 
    \[\kappa(f(w))\in[\frac{1}{2}+\theta_-,2^{k-1}\theta_+].\]
    Take $k'$ slightly smaller than $k$ such that $\theta'_- = \frac{-1}{2^{k'}-1}$ satisfies $2t_0\textless1+2^{k'-1}\theta'_-\textless\frac{1}{2}+\theta_-$. (adjust $k,k'$ to be larger if necessary). Thus we can apply Lemma \ref{lem.ringed.critical} to graph (\ref{eq.graph.S0}) defined by $\theta'_-$ and hence $w$ is ringed at depth 0. To conclude, we have proved that for any $w$, there exists a graph (\ref{eq.graph.S0}) defined by $\theta_+,\theta_-$ or $\theta_-'$ (not depending on $w$) that rings $w$ at some depth less than $q$.
    
    \item $t_0\in(\frac{1}{2^{n+1}},\frac{1}{2^n}]$, $n\geq 3$. The strategy is quite similar as above. In this case we have $2t_0\in(\frac{1}{2^{n-2}},\frac{1}{2^{n-1}}]$. For the same reason, $w$ is ringed by graph defined by $\theta_+$ for $k$ large if $\kappa(f(w))\in(0,2^{k-n}\theta_+)\cup(2^{k-n+1}\theta_+,1)$. Take $\epsilon\textgreater 0$ small enough such that $2^{n-2}(2^{k-n+1}\theta_+-\epsilon)\textgreater \frac{1}{2}+\theta_-$. Thus for $w$ satisfying \[\kappa(f(w))\in[2^{k-n}\theta_+,2^{k-n+1}\theta_+-\epsilon)\]
    and $j$ the smallest integer such that $f^j(f(w))\in S_0$, we have $\kappa(f^{N}(w))\in[\frac{1}{4},\frac{1}{2}+\theta_-)$, where $N = (n-2)q+j+1$. Thus $f^N(w)$ is ringed at depth 0 by graph defined by $\theta_-$. If $w$ satisfies
    \[\kappa(f(w))\in[2^{k-n+1}\theta_+-\epsilon,2^{k-n+1}\theta_+),\]
    and $j\leq q-1$ is the smallest integer such that $f^j(f(w))\in S_0$, then \[\kappa(f^{N'}(w))\in[2^{k-2}\theta_+-2^{n-3}\epsilon,2^{k-2}\theta_+)\subset[\frac{1}{4},\frac{1}{2}+\theta_-)\]
    (take $\epsilon$ smaller if necessary), where $N' = (n-3)q+j+1$. Hence $f^{N'}(w)$ is ringed at depth 0 by graph defined by $\theta_-$.
    
    \item $t_0 \in (\frac{1}{2},1)$. The argument is symmetric to the case $t_0\in(0,\frac{1}{2})$, which is already handled above. 
\end{itemize}
\end{proof}

\begin{lemma}\label{lem.critlimb=1/2}
Let $t_0 = \frac{1}{2}$ and $z$ satisfy the hypothesis in Proposition \ref{prop.ringed.infini_return}. Then $z$ is infinitely ringed by some graph (\ref{eq.graph.S0}).
\end{lemma}

\begin{proof}
The prove goes exactly the same as Lemma \ref{lem.orbit.L1/2}.
\end{proof}

To summarize all the cases above, Lemma \ref{lem.orbit.hit.boundary}, \ref{lem.orbit.L1/2}, \ref{lem.critlimb=1/2} and Proposition \ref{prop.ringed.infini_return} gives

\begin{proposition}\label{prop.summary.infinite.ringed}
Suppose $L(0) = 0$. Let $z\in K_f$ not in the inverse orbit of 0 and not in any basin of 0, then for all $k$ large enough, $z$ is infinitely ringed by one of the graphs (\ref{eq.graph.Si}).
\end{proposition}

Applying Yoccoz's Theorem to $v_- = f(c_-)$, $f = f_a$, we get

\begin{theorem}\label{thm.dym}
Let $\mathcal{U}\subset\mathcal{C}_{\lambda}$ be a parabolic component of adjacent, bitransitif or capture type that is not in any $\mathcal{W}^\pm(\boldsymbol{\mathrm{a}}_m)$. Suppose $a\in \partial\mathcal{U}$, $f^n_a(c_-(a))\not=0,\forall n\geq 1$ and $a\not\in\mathcal{A}_{p/q}$. Then there exists a graph (\ref{eq.graph.Si}) and a sequence of non-degenerated annuli $A^a_{n_i},{i\geq 0}$, such that 

\begin{enumerate}
    \item $A^a_{n_i} = P^{a,v}_{n_i}\setminus\overline{P^{a,v}_{n_i+1}}$, $i\geq 1$, where $P^{a,v}_{n}$ is the puzzle piece of depth $n$ containing $v_-(a)$.
    \item $f^{n_i-n_0}_a:A^a_{n_i}\longrightarrow A^a_{n_0}$ is a non-ramified covering.
    \item\label{dichotomie} either $\sum_i mod(A^a_{n_i}) = \infty$ or there exists $k\geq 1$ such that $f_a^k: P^{a,v}_{m+k}\longrightarrow P^{a,v}_m$ is quadratic-like for all $m$ large enough and $\bigcap P^{a,v}_m$ is the filled Julia set of the renormalized map $f_a^k$.
\end{enumerate}
\end{theorem}

The following lemma allows us to apply Yoccoz's Theorem to other points of the Julia set.

\begin{lemma}\label{lem.ringed-critical}
Let $L(t_n)$ be the limb containing $f^n(c_-)$, $n\geq 0$. If
\begin{itemize}
    \item $t_n$ is not dyadic for $n$ large enough,
    \item or there exists $n_1,n_2$ such that $t_{n_1}\in [\frac{1}{4},\frac{1}{2})$, $t_{n_2}\in (\frac{1}{2},\frac{3}{4}]$,
\end{itemize}
then there exists $N$ such that $c_-$ is ringed at depth $N$ for \textbf{both} graphs (\ref{eq.graph.Si}) for all $k$ large enough.
\end{lemma}
\begin{proof}
The proof for the first case is just a repeat of \cite[Lem. 6.2]{Roesch}. The second case is deduced directly from Lemma \ref{lem.ringed}
\end{proof}

\begin{theorem}\label{thm.local-connectivity-julia}
Let $f$ satisfy Assumption ($\Diamond$). For all $n\geq 0$, $f^{-n}(B^*_m(0))$ is locally connected. 
\end{theorem}
\begin{proof}
Let $z\in B^*_m(0)$. Corollary \ref{cor.loc.connec.inverse-orbit0} treat the case when $z$ is in the inverse orbit of $0$. For other $z$, suppose first that $c_-\in L(t_n)$ satisfies one of the hypothesis in Lemma \ref{lem.ringed}, then Lemma \ref{lem.orbit.hit.boundary} allows us to use Yoccoz's Theorem to get the dichotomy: if the intersection of the puzzle pieces containing $z$ shrink to $z$, then we are done; if the intersection is a quadratic copy, then this copy is separated from $B^*_m(0)$ by two external rays landing at $z$. Let $W$ be the open region bounded by the two rays containing the quadratic copy, then $(P_n(z)\setminus W)_n$ form a connected basis of $z$. For details of this part, see the proof of \cite[Thm. 1]{Roesch}.

The only case left not covered by Lemma \ref{lem.ringed-critical} is that $t_n$ is dyadic for all $n$, but either $t_{n}\not\in [\frac{1}{4},\frac{1}{2})$ for all $n$ or $t_{n}\not\in (\frac{1}{2},\frac{3}{4}]$ for all $n$. For example if it is the first alternative, we treat it furthermore by two subcases:
\begin{itemize}
    \item $t_n = \frac{1}{2}$ for all $n$. By Lemma \ref{lem.critlimb=1/2}, we may suppose (\ref{eq.graph.Si}) defined by $\theta_+$ rings $c_-$ for all $k$ large enough. By the proof of Lemma \ref{lem.orbit.hit.boundary}, $z$ might not be infinitly ringed by the graph of $\theta_+$ only when $\kappa(z)$ is recurrent to $\frac{1}{4}$. However, in this case there exists $k$ large enough such that the graph with $\theta_+$ does not contain $z$ for all $n$. There is $N$ such that the puzzle piece $P_0$ containing $w\in\partial B^*m(0)$ with $\kappa(w) = \frac{1}{4}$ contains $f^N(z)$. Then $\overline{P_0}$ does not contain forward orbit of $c_-$. Hence by the shrinking lemma, the pieces around $f^N(z)$ shrink to $f^N(z)$.
    \item $t_n$ is not always $\frac{1}{2}$. Without loss of generality, suppose $t_{n_0} = \frac{3}{4}$. Then by Lemma \ref{lem.ringed}, $c_-$ is ringed at depth 0 by graph of $\theta_+$. Similarly as the above case we can apply the shrinking lemma for $\kappa(z)$ recurrent to $\frac{1}{4}$.
\end{itemize}

\end{proof}

\subsection{Wakes attached to end points}\label{subsec.wake-end-points}
Suppose $f$ satisfy Assumption ($\Diamond$). Set $B_0 := \bigcup_m B^*_m(0)$. Define $B_n$ to be the union of Fatou components not in $B_{n-1}$ but attached to $\partial B_{n-1}$ at a preimage of $0$. Set $B = \overline{\bigcup_n B_n}$. For each component $U\subset B_n$, $\overline{U}$ is eventually, firstly and bijectively sent to some $\overline{B}_m^*(0)$ by a certain iteration of $f$. So any $z\in\partial U$ is associated to an angle $\omega_n$. Recall that $B_m^*(0)\subset S_m$, so we associate $U$ with $\epsilon_n = m$. Therefore we can using the following address to encode the position of $z$:
\begin{equation}\label{eq.parametri.julia}
    [(\epsilon_0,\omega_0);(\epsilon_1,\omega_1);...;(\epsilon_n,\omega_n))]
\end{equation}
where $\omega_0,...\omega_{n-1}$ are dyadic, $\omega_n\in\mathbb{R}/\mathbb{Z}$; $0\leq \epsilon_k\leq q-1$ is chosen so that $f^k(z)\in S_{\epsilon_k}$. In the same way we can encode the position of components of $B_n$.

Define wakes $W_U(\omega)$ and limbs $L_U(\omega)$ attached to $\partial U$ similarly as \ref{subsec.wakes-to-B}
Let $W$ be the open region bounded by the two external rays separating $U$ with $B_{n-1}$. Likewise we have $W  = U\cup\bigcup_{\omega\not=0} W_U(\omega)$, $K_f\cap W = U\cup\bigcup_{\omega\not=0} L_U(\omega)$. Notice that the wakes and limbs defined here include those in \ref{subsec.wakes-to-B}: $L_{B_m^*(0)}(\omega) = L(\omega)\cap S_m$.

\begin{proposition}\label{prop.non-trivial-limb}
Suppose $\omega$ is not dyadic. Let $\{z\} = L_U(\omega)\cap \partial U$. If $L_U(\omega)$ is non trivial, then either $z$ is in the inverse orbit of $c_-$ or $z$ is (pre)-periodic. Moreover, if $L_U(\omega)$ is non trivial, then there are exactly two external rays landing at $z$ separating $L_U(\omega)$ with $U$, otherwise there is only one.
\end{proposition}
\begin{proof}
Suppose $L_U(\omega)$ is non trivial and $z$ not in the inverse orbit of $c_-$. Then $L_U(\omega)$ is sent injectively to the critical limb by some $f^N$ (since the width of $W_U(\omega)\textgreater 0$ and it is multiple by $3$ by every iteration of $f$). So we may suppose $L_U(\omega)$ is the critical limb. Then $U = B^*_m(0)$, since for otherwise $f^n(L_{f(U)}(f(z))$ will never contain $L_U(\omega)$. Hence $z$ is periodic. 

Since in the proof of Theorem \ref{thm.local-connectivity-julia}, we see that the pieces shrink to $z$ if the limb is trivial, so it is direct that there is only one external ray landing at $z$ if furthermore $z$ does not hit $c_-$. Existence of two external rays landing at $z$ if the tableau of $z$ is (pre-)periodic comes from the shrinking property of $P_n\setminus W$, where $W$ is as in the begining of the proof of Theorem \ref{thm.local-connectivity-julia}. For the proof of the uniqueness of these two rays, see \cite[Lem. 7.2]{Roesch}.
\end{proof}

More generally, consider $z\in B\setminus\bigcup_n\overline{B_n}$,  
and suppose $z_k\in\bigcup_n\overline{B}_n$ converging to $z$. Then for all $n$, there exists $z_k$ such that $z\in\overline{B}_n$. Otherwise, there exists $N$ such that for $k\geq N$, $z_k$ belongs to the closure of some Fatou component in $B_{N+1}$ attached at some $\partial U\subset B_N$. Clearly $z_k$ converges to some $L_{U}(t)$ with $z_k\not\in L_U(t)$ since $z\in B\setminus\bigcup_n\overline{B_n}$. Then $L_U(t)$ must be non trivial since $z$ is not the root of $L_U(t)$. By Proposition \ref{prop.non-trivial-limb}, $L_U(t)$ is separated from all other limbs by two external rays, so $z_k$ can not converge to $z$. Therefore there exists a unique sequence of Fatou components $U_n\subset B_n$ such that $z\in \bigcap_n L_{U_n}(\omega_n)$ with $\omega_n$ dyadic. Thus any $z\in \partial B\setminus\{0\}$ can be represented by a unique infinite sequence (if $z\in\partial\overline{B_n}$, add zeros after $(\epsilon_n,\omega_n)$):
\[[z] := [(\epsilon_0,\omega_0);(\epsilon_0,\omega_1);...;(\epsilon_n,\omega_n));...].\]
where $\omega_k\in(0,1)$, $0\leq \epsilon_k\leq q-1$. The coordinate for $f(z)$ is 

\begin{equation}
   [f(z)] = \left\{
\begin{aligned}
&[(\epsilon_0+p,2\omega_0);(\epsilon_1,\omega_1);...;(\epsilon_n,\omega_n))...],\,\,\text{if } \epsilon_0 = 0, \omega_0\not=\frac{1}{2}\\
&[(\epsilon_1,\omega_1);(\epsilon_2,\omega_2);...;(\epsilon_n,\omega_n))...],\,\,\text{if } \epsilon_0 = 0, \omega_0=\frac{1}{2}\\
&[(\epsilon_0+p,\omega_0);(\epsilon_1,\omega_1);...;(\epsilon_n,\omega_n))...],\,\,\text{if } \epsilon_0 \not= 0
\end{aligned}
\right.
\end{equation}

\begin{lemma}\label{lem.twopreimages.B}
$z\in B$ has at least 2 preimages in $B$.
\end{lemma}
\begin{proof}
This is clear for $z\in \overline{B_n}$. Take $z\in B\setminus\bigcup_n\overline{B_n}$, then $z$ is the limit of some $(z_n)$ with $z_n\in \overline{B_{N_n}}$. Each $z_n$ has two preimages $x_n,y_n$, which, up to taking a subsequence, have different limits $x,y$. Clearly $f(x) = f(y) = z$.
\end{proof}

\begin{remark}\label{rem.coordi.quadratic.model}
The construction of the sequence $[(\epsilon_0,\omega_0);(\epsilon_0,\omega_1);...;(\epsilon_n,\omega_n);...]$ is also valid for the quadratic model $P_{\lambda}(z)=e^{2\pi i\frac{p}{q}}z+z^2$. Moreover the following map is bijective:
\begin{equation}\label{eq.encoding.julia}
    \Xi: J_{\lambda}\setminus\{0\}\longrightarrow [{\mathbb{Z}}/{q\mathbb{Z}}\times(0,1)]^\mathbb{N},\,\,z\mapsto[z].
\end{equation}
Indeed, the surjectivity is clear; the injectivity is also clear for $[z]$ finite. Notice that $J_{\lambda}$ can be regarded as a pinching limit of the Julia set of $\lambda z+z^2$ with $|\lambda|\textless 1$ and the pinching map $\phi$ is injective beyond the skeletons, while the skeletons are $\phi^{-1}(y)$ with $y$ preimage of $0$ (cf. Theorem \ref{thm.pinching}). Thus if $[z]$ is infinite, then clearly $\phi^{-1}(z)$ does not intersect any skeleton, hence $\#\phi^{-1}(z) = 1$.
\end{remark}

\begin{lemma}\label{lem.unique.external.end-point}
Suppose $f$ is non renormalisable. If $v_- = f(c_-)\in B\setminus\bigcup_n\overline{B_n}$, then there is only one external ray landing at $v_-$.
\end{lemma}
\begin{proof}
The existence of external ray landing at $v_-$ comes from the shrinking property of puzzles pieces around it given by Theorem \ref{thm.dym}. Suppose the contrary that there are two rays landing. Let $W$ be the open region separated from $B$ by these two rays. Then there exists $N$ such that $W$ is sent injectively onto $f^N(W)$ with $c_-\in f^N(W)$. But this is impossible, since by Lemma \ref{lem.twopreimages.B} $f^{-1}(v_-)$ has two preimages in $B$, so $c_-$ must belong to $B$ because $f$ has degree 3.
\end{proof}

\section{Passing to parameter plan}\label{sec.paragraph}
In this section we pass the dynamical combinatorics to parameter ones. To simplify the redaction, from now on we will mainly work in the  "fundamental region" ${\mathcal{S}}^+_m$ for $0\leq m\leq\lfloor\frac{q}{2}\rfloor$ (recall Definition \ref{def.fundamental.Sm}).

\subsection{Parameter equipotentials and rays}

For $m=0$, define parameter equipotentials and union of internal rays in ${\mathcal{D}}_0$ and $\tilde{\mathcal{D}}_0$ by

\begin{equation}
        {\mathcal{E}}^0_n = ({\Psi}^{adj}_0)^{-1}(E^0_n\cap B^*_{{p}}(0)),\,\,\mathcal{R}^0_n = ({\Psi}^{adj}_0)^{-1}(R^0_n\cap B^*_{{p}}(0)),\,\,  n\geq 0.
\end{equation}

For $1\leq m\leq \lfloor\frac{q}{2}\rfloor$, define parameter equipotentials and union of internal rays in ${\mathcal{D}}_m$ resp. $\tilde{\mathcal{D}}_m$ by 

\begin{equation}
        \mathcal{E}^m_n = (\Psi^{bit}_m)^{-1}(E^m_n\cap B^*_{\overline{m+p}}(0)),\,\, \mathcal{R}^m_n = (\Psi^{bit}_m)^{-1}(R^m_n\cap B^*_{\overline{m+p}}(0)),\,\, n\geq 0.
\end{equation}

Let $k\geq 1$, $\mathcal{U}_k\subset\tilde{\mathcal{S}}_m$ be a capture component. Suppose for $a\in\mathcal{U}_k$, $c_-(a)$ firstly hit $B_{a,l}^*(0)$ after $k$ iterations. Define the parameter equipotentials and internal rays by
\begin{equation}
    \mathcal{E}^{\mathcal{U}_k}_n = (\Psi_{\mathcal{U}_k})^{-1}(E^m_n\cap B^*_{{l}}(0)),\,\, \mathcal{R}^{\mathcal{U}_k}_n = (\Psi_{\mathcal{U}_k})^{-1}(R^m_n\cap B^*_{l}(0)),\,\, n\geq 0.
\end{equation}

Next we investigate landing properties for equipotentials in $\mathcal{B}_m$. We only state the result for $\mathcal{E}^m_n$ in $\mathcal{D}_m$ since it will be the same for $\tilde{\mathcal{E}}^m_n$.

\begin{proposition}\label{prop.landing.equi.bitran}
Let $n,k\geq 0$. Let $1 \leq l\leq q$ be such that $m = lp \pmod{q}$.   Then $\mathcal{E}_n^m\cap\partial\mathcal{D}_m$ consists of finitely many points. Denote by $\sfrac{\mathcal{E}_n^m}{\sim}$ the quotient space of $\mathcal{E}_n^m\cap\partial\mathcal{D}_m$ by 
\begin{itemize}
    \item gluing the double parabolic parameter on $\partial\mathcal{B}_0$ and $2\sqrt{\lambda}$ if $m=0$;
    \item gluing the two parabolic parameters on $\partial\mathcal{B}_m$ if $1\leq m\leq \frac{q}{2}$.
\end{itemize}
 Then $\sfrac{\mathcal{E}_n^m}{\sim}$ is homeomorphic to $E_n^m\cap B_{\overline{m+p}}^*$. In particular $\mathcal{E}^m_{kq} = \mathcal{E}^m_{kq+1} = ...= \mathcal{E}^m_{(k+1)q-1}$ and $(\mathcal{E}^m_{kq}\cap \partial\mathcal{D}_m)\subsetneq(\mathcal{E}_{(k+1)q}^m\cap \partial\mathcal{D}_m)$. Moreover for $a\in\mathcal{E}^m_{kq}\cap \partial\mathcal{D}_m$ not double parabolic, $f^{(k+1)q-l}_a(v_-(a))=0$.
\end{proposition}
\begin{proof}
The proof has no difference to the case $p/q = 1$. See \cite[Prop. 3.2.10]{runze}. 
\end{proof}

By the above proposition, we see that $\mathcal{D}_m$ is cut "binarily" by $\mathcal{E}^m_{kq}$ into $2^k$ pieces if $m=0$ and $2^{k+1}$ pieces otherwise. For a double parabolic parameter $b\in\partial\mathcal{B}_m$, let $\{a_k\}_{k\geq0}$ be the sequence such that $a_k\in\partial\mathcal{B}_m$ be the closest summit to $b$ among $\mathcal{E}^m_{kq}\cap\partial\mathcal{B}_m$. By Lemma \ref{lem.landing.Misiurewicz}, there are $q$ parameter external rays landing at $a_k$ and $q$ dynamical rays with the same angles landing at $v_-(a_k)$. The angles are preimages of 0 under multiplication by $3$. Denote by $T_k$ the set of these angles.

\begin{lemma}\label{lem.external.angle.limit}
$3^qT_{k+1} = T_k$. Moreover all angles in $T_k$ converges to $\alpha$, the angle of the external ray $\mathcal{R}_\infty(\alpha)$ landing at $b$ with $\mathcal{R}_\infty(\alpha)\subset\partial\mathcal{S}_m$.
\end{lemma}
\begin{proof}
We only do the proof for $m=0$ to illustrate the idea. In this case, $b = \boldsymbol{\mathrm{a}}_0$ and hence $\theta = \beta^+_{0}$ (Definition \ref{def.angle.four.rays}). For $n\geq0$, let $\mathcal{U}_k$ be the component bounded by $\mathcal{R}_\infty(\alpha^+_{0})\cup\mathcal{E}^0_{kq}\cup\mathcal{R}_\infty(t_k)$, where $\mathcal{R}_\infty(t_k)$ is the external ray closest to $\mathcal{R}_\infty(\beta^+_{0})$ landing at $a_k\in\mathcal{E}^0_{kq}\cap\partial\mathcal{B}_0$, the point closest to $\boldsymbol{\mathrm{a}}_0$ (in the binary sense). By Corollary \ref{cor.landing.rational.double.parabolic}, external rays $\mathcal{R}_\infty(t)$ with $3^{kq}t\in\Theta_0$ do not land at $ \boldsymbol{\mathrm{a}}_0$. Hence 
\[ \mathcal{U}_k\setminus\bigcup_{\{t;\,3^{kq}t\in\Theta_0\}}\mathcal{R}_\infty(t)\]
contains $\mathcal{W}^+(\boldsymbol{\mathrm{a}}_0)$, and denote by $\tilde{\mathcal{U}}_k$ the component containing it. Notice that for $a\in\tilde{\mathcal{U}}_{k-1}$, there is a dynamical holomorphic motion of $\bigcup_{t\in T_k}R^\infty_{a}(t)\cup\overline{\Omega_{a,p}^{kq-1}}$. Therefore $T_k$ is exactly the set of angle of external rays landing at $f_a^{-kq}(0)\cap\partial B_{a,p}^*(0)$ which is closest to $\beta^+_{0}$. By the proof of Proposition \ref{prop.localconnec.0} we see that $T_k$ satisfies $3^qT_{k+1} = T_k$ and tends to $\beta^+_{0}$ as $k\to+\infty$.

\end{proof}

\begin{proposition}\label{prop.nodouble.boundary.capture}
Let $\mathcal{U}_k$ be a capture component. Then $\partial\mathcal{U}_k$ contains no double parabolic parameter.
\end{proposition}
\begin{proof}
Suppose the contrary that there exists $b\in\mathcal{A}_{p/q}\cap\partial\mathcal{U}_k$. Consider respectively two cases:

\begin{itemize}
    \item $\mathcal{U}_k$ is contained in some $\mathcal{S}_m$. Since $\partial\mathcal{U}_k\subset\partial\mathcal{C}_{\lambda}$ (Lemma \ref{lem.boundary.components}), there is an external ray contained in $\mathcal{S}_m$ accumulating to $\partial\mathcal{U}_k$. This contradicts Lemma \ref{lem.external.angle.limit}.
    
    \item $\mathcal{U}_k$ is contained in some double parabolic wake. Consider equation of $z$, $f_a^q(z) = z$. When $a = b$, 0 is a multiple solution of order $2q+1$, while for $a\not=b$ near $b$, 0 is has order $q+1$. All the other solution besides 0 of $f_b^q(z) = z$ gives a repelling periodic cycle for $b$, which admits a holomorphic motion for $a$ near $b$. Thus there are only $q$ non-zero solutions left of $f_a^q(z) = z$ for $a\not=b$ ($a$ near $b$) which do not come from the holomorphic motion. Among these $q$ solutions, there is at least a repelling cycle $C$ of period dividing $q$ for $f_a$. Now suppose furthermore $a\in\mathcal{U}_k$. Let $\Theta$ be the set of angle of external rays for $f_a$ landing at $C$. Notice that $\Theta$ does not depend on $a$ since in $\mathcal{U}_k$ there is a holomorphic motion of $B_a(\infty)$ induced by Böttcher coordinate. $\Theta$ can not be $\Theta_m,\Theta_{m+1}$ since $C$ is not $0$. Hence for $f_b$, the external rays with angles in $\Theta$ land at a repelling cycle, which in turn gives a repelling cycle for $f_a$ by holomorphic motion. This contradicts how we choose $C$. 
\end{itemize}
\end{proof}

We get immediately
\begin{corollary}\label{cor.landing.Equi.capture}
$\overline{\mathcal{E}^{\mathcal{U}_k}_n}$ is homeomorphic to $\overline{\Psi_{\mathcal{U}_k}(\mathcal{E}^{\mathcal{U}_k}_n)}$. Moreover $a\in\overline{\mathcal{E}^{\mathcal{U}_k}_n}\cap\partial\mathcal{U}_k$ is Misiurewicz parabolic.
\end{corollary}

\begin{corollary}\label{cor.loc.connected.cap-in-double-wake}
Let $\mathcal{U}_k$ be a capture component contained in some double parabolic wake $\mathcal{W}^{\pm}(\boldsymbol{\mathrm{a}}_m)$, then for $a_0\in\partial\mathcal{U}_k$, $f^k_{a_0}(c_-(a_0))$ is on the boundary of some $B^*_{a_0,m}$. Moreover $\partial\mathcal{U}_k$ is a Jordan curve.
\end{corollary}

\begin{proof}
By Proposition \ref{prop.nodouble.boundary.capture}, there is an open neighborhood of $\mathcal{U}_k$, on which the union of immediate basins $\bigcup_i B_{a,i}^*(0)$ admits an holomorphic motion $h_a$ induced by Fatou coordinate. One can choose the base point $a'_0$ of $h_a$ in $\mathcal{U}_k$. By $\lambda$-lemma, the motion is extended to $\overline{\bigcup_i B_{a,i}^*(0)}$. For $a\in\mathcal{U}_k$, by definition of capture component there exists $m$ such that $f^k_{a}(c_-(a))\in B^*_{a,m}$. Let $a$ tend to $a_0$, we get $f^k_{a_0}(c_-(a_0))\in \partial B^*_{a,m}$ since $\overline{B^*_{a,m}}$ moves holomorphically.

Define $H:\partial\mathcal{U}_k\longrightarrow\partial B^*_{a'_0,m}$ by $a\mapsto h^{-1}_a(f_a^k(c_-(a)))$. $H$ is locally regular, hence $\partial\mathcal{U}_k$ is locally connected since $\partial B^*_{a'_0,m}$ is (Julia set of $f_{a'_O}$ is locally connected since it is geometrically finite \cite{TaYi}). $\partial\mathcal{U}_k$ is then a Jordan curve by Lemma \ref{lem.boundary.components}.
\end{proof}

From Proposition \ref{prop.landing.equi.bitran} and \ref{prop.nodouble.boundary.capture}, we see that $\mathcal{R}^m_n,\mathcal{R}^{\mathcal{U}_k}_n$ can not land at double parabolic parameters. We have the following landing property of internal rays (the proof is similar to that for external rays) :

\begin{lemma}\label{lem.landing.rational.internal-ray}
Let $0\leq m\leq\lfloor\frac{q}{2}\rfloor$ and $\mathcal{U}_k$ be a capture component not in double parabolic wake. Then any component of $\mathcal{R}^{m}_n$ resp. $\mathcal{R}^{\mathcal{U}_k}_n$ land at some $a_0\in\partial\mathcal{B}_m$ resp. $\partial\mathcal{U}_k$ which is neither Misiurewicz parabolic nor double parabolic. The landing points of $f^{-k}_{a_0}(R^m_{a_0})$ are (pre-)periodic to the same cycle, where $R^m_{a_0}$ is just $R^m_0$ in (\ref{eq.internal-ray.dym}). Moreover if this cycle is repelling, then some component of $f^{-k}_{a_0}(R^m_{a_0})$ will land at $v_-(a_0)$.
\end{lemma}

We also have the analogue to Lemma \ref{lem.landing.Misiur.parabo}:
\begin{lemma}\label{lem.landing.Misiure.internal-rays}
Suppose $a_0\in\mathcal{C}_{\lambda}\setminus\bigcup_m\mathcal{W}^\pm(\boldsymbol{\mathrm{a}}_m)$ is a Misiurewicz parameter. Then $a_0$ is the landing point of some $\mathcal{R}^m_n$ or $\mathcal{R}^{\mathcal{U}_k}_n$, if and only if $v_-(a_0)$ is the landing point of some $f^{-k}_{a_0}(R^m_{a_0})$.
\end{lemma}

\subsection{Dynamical objects move holomorphically}\label{subsec.dynamics-objects}

\begin{definition}
Let $A_0\subset\mathbb{S}^1$ be finite and satisfy $3A_0 = A_0$, $A_n$ be the $n$-th preimage of $A_0$ under multiplication by 3. A parameter object of depth $n\geq0$ is the intersection of one of the following four sets and $\tilde{\mathcal{S}}_m$.
\begin{equation}
        \bigcup_{\eta\in A_n}\overline{\mathcal{R}_{\infty}(\eta)},\,\,
        \bigcup_{k=1}^{n}\left(\bigcup_{\mathcal{U}_k\subset\mathcal{H}_k}\overline{\mathcal{E}^{\mathcal{U}_k}_{n-k}}\right)\cup\overline{\mathcal{E}^m_n},\,\,
        \bigcup_{k=1}^{n}\left(\bigcup_{\mathcal{U}_k\subset\mathcal{H}_k}\overline{\mathcal{R}^{\mathcal{U}_k}_{n-k}}\right )\cup\overline{\mathcal{R}^m_n},\,\,\mathcal{E}_{\infty}(r^{1/3^n}).
\end{equation}
\end{definition}

Let $\mathcal{O}_n$ be a parameter object of depth $n$. Let $\tilde{\mathcal{O}}_n$ be a connected component of ${\mathcal{S}}^+_m\setminus\mathcal{O}_n$ intersecting $\partial\mathcal{C}_{\lambda}$. Let $a\in\tilde{\mathcal{O}}_n$. 

When $\mathcal{O}_n$ is $\bigcup_{\eta\in A_n}\overline{\mathcal{R}_{\infty}(\eta)}$ or $\mathcal{E}_{\infty}(r^{1/3^n})$, $\bigcup_{\eta\in A_{n+1}}\overline{{R}^{\infty}_a(\eta)}$ resp. ${E}^{\infty}_a(r^{1/3^{n+1}})$ does not contain $c_-(a)$. Hence each external ray $R_a^\infty(\eta)$ is well defined and lands at $J_a$, resp. $E_a^\infty(r^{1/3^{n+1}})$ is homeomorphic to $\mathbb{S}^1$. Moreover there is a dynamical holomorphic motion of $\bigcup_{\eta\in A_{n+1}}\overline{{R}^{\infty}_a(\eta)}$ resp. ${E}^{\infty}_a(r^{1/3^{n+1}})$ induced by Böttcher coordinate $\phi_a^\infty$.

Now let $\mathcal{O}_n=\bigcup_{k=1}^{n}\left(\bigcup_{\mathcal{U}_k\subset\mathcal{H}_k}\overline{\mathcal{E}^{\mathcal{U}_k}_{n-k}}\right)\cup\overline{\mathcal{E}^m_n}$. We want to find a dynamical holomorphic motion for equipotentials in $B_a(0)$. Recall that $\Omega_a$ is the maximal petal contained in $B^*_{a,0}(0)$ and $\tilde{\Omega}_a$ Let $H_a$ be the connected component of $P^{-2q+l}(\overline{\tilde{\Omega}_a})$ contained in $\overline{B^*_{a,m}(0)}$. Then by definition of $\mathcal{E}_n^m,\mathcal{E}^{\mathcal{U}_k}_{n-k}$, $f^{-n-1}_a(E^a_0)$ does not contain $c_-(a)$, where $E_0^a := \bigcup_{i=1}^{q}f_a^{i}(\partial{H}_a)$. Hence there is a dynamical holomorphic motion of $E_{n+1}^a:=f^{-n-1}_a(E^a_0)$ induced by Fatou coordinate and the pull back of $f_a$.

Now let $\mathcal{O}_n=\bigcup_{k=1}^{n}\left(\bigcup_{\mathcal{U}_k\subset\mathcal{H}_k}\overline{\mathcal{R}^{\mathcal{U}_k}_{n-k}}\right)\cup\overline{\mathcal{R}^m_n}$. Notice that here the construction of internal rays is more subtle, since $c_-(a)$ might be in $B^*_{a,m}(0)$, as a consequence we can not identify completely the dynamics in $B^*_{a,m}(0)$ with the quadratic model (but still partially, up to certain depth, see below the rewritten diagram of (\ref{diag.commu1})). The idea is to pull back the beginning of the internal rays so long as they do not contain $v_-(a)$. Similar to the quadratic model, one can define internal rays $R_a(\theta)\subset B^*_{a,m}(0)$ with $\theta\in\Theta_\pm = \{\theta;\,\exists i \text{ s.t. } 2\theta = 2^i\theta_\pm\}$ where $\theta_\pm = \frac{\pm1}{2^k-1}$. Define similarly $R^a_0$ to $R^m_0$ in (\ref{eq.internal-ray.dym}). There is also a dynamical holomorphic motion of $R_{n+1}^a:=f^{-n-1}_a(R_0^a)$. See \cite[3.1]{runze} for more details. 

\begin{equation}\label{diag.commu.f_a}
\begin{tikzcd}
\Omega_{a,\overline{m+p}}^{n_{a}} \ar[r]{}{f_{a}}\ar[d]{}{h_{a}} & \Omega_{a,\overline{m+2p}}^{n_{a}-1} \ar[d]{}{h_{a}} \ar[r]{}{f_{a}} & ... \ar[d]{}{h_{a}}  \ar[r]{}{f_{a}} & \Omega_{a,q-p}^1 \ar[d]{}{h_{a}} \ar[r]{}{f_{a}} & \Omega_{a,0}^0 \ar[d]{}{h_{a}}\\
\Omega^{n_a}_{\overline{m+p}} \ar[r]{}{P_{\lambda}} & \Omega^{n_a-1}_{\overline{m+2p}}  \ar[r]{}{P_{\lambda}} & ...  \ar[r]{}{P_{\lambda}} & \Omega^1_{q-p}  \ar[r]{}{P_{\lambda}} & \Omega^0_0
\end{tikzcd}
\end{equation}
In the above diagram,  $\phi_a:\overline{\Omega^0_{a,0}}\longrightarrow\overline{\mathbb{H}}$ is the Fatou coordiante of $f^q_{a}|_{B^*_{a,0}(0)}$ normalised by $\phi_a(c_+(a)) = 0$, $n_a$ is the smallest integer such that $\Omega_{a,\overline{m+p}}^{n_{a}}$ contains $v_-(a)$.

\begin{definition}
Let $a\in\tilde{\mathcal{O}}_n$. The dynamical objects $O^a_n$ of depth $n$ corresponding to the four parameter ones are
\begin{equation}\label{eq.dynamical.object}
        \bigcup_{\eta\in A_{n+1}}\overline{{R}^{\infty}_a(\eta)},\,\, f^{-n-1}_a(E^a_0),\,\,f^{-n-1}_a(\overline{R^a_0}),\,\,{E}^{\infty}_a(r^{1/3^{n+1}})
\end{equation}
\end{definition}
 
To conclude from the discussion in this subsection:
\begin{proposition}\label{prop.holomorphic.motion.dym.obj}
Let $\mathcal{O}_n$ be a parameter object of depth $n$, $a,a_0\in\tilde{\mathcal{O}}_n$. Let $O^a_{n+1},O^{a_0}_{n+1}$ be the corresponding dynamical object of depth $n+1$. There exists a dynamical holomorphic motion $L_n:\mathcal{O}_n\times O^{a_0}_{n+1}\longrightarrow\mathbb{C}$ with $L_n(a,O^{a_0}_{n+1}) = O^{a}_{n+1}$. Moreover, if $O^a_{n+1}$ is the first one or the third one in (\ref{eq.dynamical.object}), $O^a_{n+1}\cap J_a$ is either repelling or pre-periodic to $z=0$.

\end{proposition}

\subsection{Parameter graphs and puzzles}\label{sub.sec.paragraph}

\paragraph{Parameter graphs $\mathcal{Y}_n$ and puzzles $\mathcal{Q}_n$ in $\overline{\mathcal{S}^+_m}$}\mbox{}\\
Fix some $r\textgreater 1$, for each $n\geq0$ define the graph adapted for parameters of Misiurewicz parabolic type.
\[\mathcal{Y}_n = \bigcup_{k=1}^{n}\left(\bigcup_{\mathcal{U}_k\subset\mathcal{H}_k\cap\mathcal{S}^+_m}\overline{\mathcal{E}^{\mathcal{U}_k}_{n-k}}\right)\cup\left(\bigcup_{t\in T_n}\overline{\mathcal{R}_{\infty}(t)}\right)\cup\mathcal{E}^m_n\cup(\mathcal{E}_{\infty}(r^{1/3^n})\cap\mathcal{S}^+_m)\]
where 
\[T_n = \{t;\,3^{n+l}t \in\Theta_m\}\cap[\alpha_{m-1}^+,\beta^+_{m}],\,\, 0\leq l\leq q-1 \text{ s.t. } lp+m = 0 \pmod q\]

By Proposition \ref{prop.landing.equi.bitran}, Corollary \ref{cor.landing.Equi.capture}, the landing points of equipotentials in parabolic components are Misiurewicz parabolic. By Lemma \ref{lem.landing.Misiur.parabo} these points are also landing points of external rays with angles in $T_n$. Hence $\mathcal{Y}_n$ is connected, and every connected component of ${\mathcal{S}}^+_m\setminus\mathcal{Y}_n$ is simply connected. We call such a connected component $\mathcal{Q}$ a puzzle piece associated to $\mathcal{Y}_n$ if it is bounded and $\partial \mathcal{Q}\cap\mathcal{E}_{\infty}(r^{1/3^n})\not =\emptyset$. We denote it by $\mathcal{Q}_n$ in the sequel. By construction, every $\mathcal{Q}_{n+1}$ is contained in a unique puzzle piece $\mathcal{Q}_{n}$.

\paragraph{Parameter graphs $\mathcal{X}_n$ and puzzles $\mathcal{P}_n(a_0)$ in $\overline{{\mathcal{S}^+_m}}$}\mbox{}\\
Next we define the graph adapted for parameters which are \textbf{not} of Misiurewicz parabolic type. Let $a_0\in \mathcal{S}^+_{m}$ such that $f_{a_0}$ satisfies Assumption ($\Diamond$). By Theorem \ref{thm.dym}, there is a graph (\ref{eq.graph.Si}) infinitely ringing $v_-(a_0)$. Recall that this graph is associated to the angle $\theta_{l} = \frac{\pm1}{2^l-1}$ of an internal ray. Let $H$ be the collection of angles of all external rays in (\ref{eq.graph.Si}) of depth 0. For each $n\geq0$ consider
\[\mathcal{X}_n = \bigcup_{k=1}^{n}\left(\bigcup_{\mathcal{U}_k\subset\mathcal{H}_k\cap\mathcal{S}^+_m}\overline{\mathcal{E}^{\mathcal{U}_k}_{n-k}}\cup\overline{\mathcal{R}^{\mathcal{U}_k}_{n-k}}\right)\cup\left(\bigcup_{\eta\in A_n}\overline{\mathcal{R}_{\infty}(\eta)}\right)\cup\mathcal{E}^m_n\cup\mathcal{R}^m_n\cup(\mathcal{E}_{\infty}(r^{1/3^n})\cap\mathcal{S}^+_m)\]
where
\[A_n = \{\eta;\,\,3^n\eta\in H\}\cap [\alpha^+_{m-1},\beta^+_m].\]
 We call a connected component $\mathcal{P}$ of ${\mathcal{S}}^+_m\setminus\mathcal{X}_n$ a puzzle piece associated to $\mathcal{X}_n$ if it is bounded and $\partial \mathcal{P}\cap\mathcal{E}_{\infty}(r^{1/3^n})\not =\emptyset$. We denote it by $\mathcal{P}_n$ in the sequel. Clearly every $\mathcal{P}_{n+1}$ is contained in a unique $\mathcal{P}_n$. Define $\mathcal{P}_n(a_0)$ to be the puzzle piece containing $a_0$. This is well-defined since $a_0\not\in \mathcal{X}_n$, for otherwise some external ray $R^{\infty}_{a_0}(\eta)$ or internal ray will land at a parabolic pre-periodic point or $v_-(a_0)$, contradicting with the construction of the dynamical graph for $f_{a_0}$ (see the observation at the beginning of \ref{subsec.infini.ringed}).

\begin{lemma}\label{lem.para.Misiur}
Let $n\geq 1$. Any $a\in \mathcal{X}_n\cap\partial\mathcal{H}\cap \mathcal{P}_0(a_0)$ is a Misiurewicz parameter. It is the landing point of an internal ray $\mathcal{R}\subset\mathcal{X}_n$ and an external ray $\mathcal{R}_\infty(t)\subset\mathcal{X}_n$. In particular $\mathcal{X}_n$ is connected.
\end{lemma}
\begin{proof}
By Proposition \ref{prop.holomorphic.motion.dym.obj}, landing points of $R^0_a$ are repelling. Therefore by Lemma \ref{lem.landing.rational.internal-ray}, $f_a$ is Misiurewicz. Also each landing point is landed by some external ray $R^\infty_a(\eta)$ with $R^\infty_{a_0}(\eta)$ involved in the admissible graph of $f_{a_0}$. By Lemma \ref{lem.landing.Misiur.parabo}, $a$ is the landing point of $\mathcal{R}_\infty(\eta)$. Hence $\mathcal{X}_n$ is connected.
\end{proof}

\paragraph{Dynamics graphs defined up to certain depth}\mbox{}\\
From the discussion in Subsection \ref{subsec.dynamics-objects}, we see that for $a$ satisfying the condition in Proposition \ref{prop.holomorphic.motion.dym.obj}, $E^a_0,R^a_0$ are well defined and homeomorphic to the corresponding objects in the quadratic model (\ref{eq.equi.dym}), (\ref{eq.internal-ray.dym}). Set $E^a_n = f_a^{-n}(E^a_0),R^a_n = f_a^{-n}(R^a_0)$. Define
\[ Y^a_n = E^a_n\cup\left(\bigcup_{t\in T_n}\overline{R_a^\infty(3^it)}\right)\cup E_a^\infty(r).\]
 \[ X^a_n = E^a_n\cup R^a_n\cup\left(\bigcup_{\eta\in A_n}\overline{R_a^\infty(3^i\eta)}\right)\cup E_a^\infty(r).\]

Proposition \ref{prop.holomorphic.motion.dym.obj} gives immediately

\begin{lemma}\label{lem.holomotion.parabolic}
Let $a_0\in \mathcal{Q}_n$. Then for $0\leq k\leq n+1$ there exists a dynamical holomorphic motion $L_k:\mathcal{Q}_n\times Y^{a_0}_k\longrightarrow \mathbb{C}$ with base point $a_0$ such that $L_k(a,Y^{a_0}_k) = Y^a_k$.
\end{lemma}

\begin{lemma}\label{lem.holomotion.misur}
Let $a_0\in \mathcal{S}^+_{m}$ such that $f_{a_0}$ satisfies the Assumption ($\Diamond$). For $0\leq k\leq n+1$, there is a dynamical holomorphic motion $L_k:\mathcal{P}_n(a_0)\times X^{a_0}_k\longrightarrow \mathbb{C}$ with base point $a_0$ such that $L_k(a,X^{a_0}_k) = X^a_k$.
\end{lemma}

\section{Local connectivity of the central part $\partial\mathcal{K}_{\lambda}$, $\lambda = e^{2\pi i\frac{p}{q}}$}\label{sec.loc}

By Yoccoz's Theorem, parameters in 
\[\mathcal{C}_\lambda\setminus\left[\mathring{\mathcal{H}^\lambda}\cup\left(\bigcup_{m=0}^{q-1}\mathcal{W}^+(\boldsymbol{\mathrm{a}}_m)\cup\mathcal{W}^-(\boldsymbol{\mathrm{a}}_m)\right)\right].\]
are divided into four cases:
\begin{itemize}
    \item[1.] Double parabolic: $a\in\mathcal{A}_{p/q}$.
    \item[2.] Misiurewicz parabolic: $\exists n\geq 1$ s.t. $f_{a}^n(c_-(a)) = 0$.
    \item[3.] Renormalisable (in the sense of Douady-Hubbard \cite{ASENS_1985_4_18_2_287_0}): there exists $U'\subset\subset U$ surrounding $c_-(a)$ and $k\geq 1$ such that $f^k|_{U'}: U'\longrightarrow U$ is quadratic like, and the copy of the quadratic Julia set is connected.
    \item[4.] Non renormalisable.
\end{itemize}

In Section \ref{sec.dym.graph}, we have constructed dynamical admissible graph infinitely ring the free critical point (critical value) for $f_a$ and obtain the dichotomy given by Yoccoz's Theorem. Now we pass this dichotomy to the parameter plane:

\begin{thmx}\label{thm.0}
If $a$ satisfies Assumption ($\Diamond$), i.e. is of case 3 and 4, then there is a dynamically defined nest of puzzles $\{\mathcal{P}_n(a)\}_n$ surrounding $a$, such that $\bigcap_n\mathcal{P}_n(a) = \{a\}$ if $a$ is non renormalisable; $\bigcap_n\mathcal{P}_n(a) = \textbf{M}_a$ is a copy of the Mandelbrot set containing $a$ if $a$ is renormalisable. In particular, $\mathcal{C}_\lambda$ is locally connected at $a$ if $a$ is non renormalisable.
\end{thmx}

In general, we do not get local connectivity at all renormalisable parameters, since essentially this needs the MLC conjecture to be true. However for certain $a$ renormalisable (for instance the cusp of the main cardioid or the tips of $\textbf{M}_a$), we can separate $\textbf{M}_a$ from $\mathcal{C}_\lambda$ by two external rays landing at $a$. Hence $\mathcal{C}_\lambda$ is also locally connected at such $a$. If we are in case 1 or 2, then the parameter puzzles we construct no longer surround $a$, but have $a$ on their boundaries. But still there is a decreasing sequence of parameter puzzles shrinking to $a$, since the corresponding decreasing dynamical puzzles shrink to a certain preimage of $z=0$, and this shrinking property can be passed to parameter puzzles by holomorphic motion.\\

To summarize, we obtain

\begin{thmx}\label{thm.A}
The boundary of every component of $\mathring{\mathcal{H}^\lambda}$ is a Jordan curve.
\end{thmx}

Theorem \ref{thm.0} actually gives more than local connectivity: the shrinking property of decreasing parameter puzzle pieces gives landing parameter external rays, which provide clean cuts of $\mathcal{C}_\lambda$. This enables us to give a global description of $\mathcal{C}_\lambda$ (see Theorem \ref{thm.decompo.para} for a precise statement):

\begin{thmx}\label{thm.B}
$\mathcal{C}_{\lambda}$ is decomposed into the central part $\mathcal{K}_{\lambda}$ and the limbs attached to them at end points.
\end{thmx}

Section \ref{sec.loc} is organised as follows: Theorem \ref{thm.0} together with Theorem \ref{thm.A} is a synthesis of  \ref{subsec.loc.connec.Misiurewicz-parabolic}, \ref{subsec.loc.connec.Non-renormalizable} and \ref{subsec.renor}; \ref{sub.sec.global} is devoted to the proof of Theorem \ref{thm.B}, the Main Theorem and Corollary \ref{cor.blokh}.

\subsection{Misiurewicz parabolic and double parabolic case}\label{subsec.loc.connec.Misiurewicz-parabolic}

\begin{proposition}\label{prop.loc.paramisiur}
Let $\mathcal{U}\subset\mathcal{S}_m^+$ be a parabolic component. Then $\partial\mathcal{U}$ is locally connected at all Misiurewicz parabolic parameters.
\end{proposition}
\begin{proof}
Let $a_0\in \partial\mathcal{U}_k$ be Misiurewicz parabolic. By Lemma \ref{lem.landing.Misiur.parabo}, let $t^1_0,t^2_0,...,t^{q-1}_0$ be the $q$ angles such that $\mathcal{R}_{\infty}(t^i_0)$ lands at $a_0$. Then $\mathcal{R}_{\infty}(t_0)$ is part of $\mathcal{Y}_n$ for all $n\geq N$, where $N$ is some fixed integer. Therefore there are exactly $q+1$ adjacent puzzle pieces of depth $n$ $\mathcal{Q}^{0}_n,\mathcal{Q}^{1}_n,...,\mathcal{Q}^{q}_n$ attached at $a_0$. Clearly $\mathcal{Q}^{i}_{n+1}\subset\mathcal{Q}^{i}_{n}$. Now we claim that for $a\in \mathcal{Q}^{i}_{n}$, $v_-(a)\in (Q^a_n)^{i}$. Indeed, for any $\epsilon\textgreater 0$, there exists $n'\textgreater n$ large enough and $t'_0$ such that $3^{n'}t'_0 = 1$, $\mathcal{R}_{\infty}(t'_0)\cap \mathcal{Q}^{i}_n \not=\emptyset$ and $|t'_0-t_0|\textless \epsilon$. Hence for $a'\in \mathcal{R}_{\infty}(t'_0)\cap \mathcal{Q}^{i}_n$, $v_-(a')\in (Q^{a'}_n)^{\pm}$. By Lemma \ref{lem.holomotion.parabolic}, $Y^a_n$ moves holomorphically for $a\in \mathcal{Q}^{\pm}_n$. Notice that $v_-(a)$ also moves holomorphically and it does not belong to $Y^a_n$ when $a\in \mathcal{Q}^{i}_n$, we deduce that $v_-(a)\in (Q^a_n)^{i}$ since $v_-(a')\in (Q^{a'}_n)^{i}$. Now if $a\in \bigcap_n\mathcal{Q}_n^{i}$, then $a$ is in $\mathcal{C}_{\lambda}$ but does not belong to any parabolic component since $a$ is excluded by every equipotentials in $\mathcal{H}_{\infty}$ and $\mathcal{H}$. Moreover $f^n_a(c_-(a)) \not= 0,\forall n\geq 1$ by definition of $\mathcal{Q}_n$. Therefore by Proposition \ref{prop.localconnec.0}, $\bigcap_{m\geq1} (Q^{a}_m)^{i} = \emptyset$. But on the other hand $v_-(a)\in \bigcap_n(Q^a_n)^{\pm}$, a contradiction. Hence $\bigcap_n\mathcal{Q}_n^{\pm} = \emptyset$. This means that if we set 
  \[\mathcal{O}_n := \bigcup_{i=0}^{q}\mathcal{Q}_n^{i}\cup\bigcup_{i=0}^{q-1}{\mathcal{R}_{\infty}(t_0)*(0,r^{1/3^n})}\cup\{a_0\}.\]
  then $(\mathcal{O}_n\cap\partial\mathcal{U})_n$
form a basis of connected neighborhood of $a_0$.
\end{proof}

Notice that $a_0 = \boldsymbol{\mathrm{a}}_m\in\mathcal{A}_{p/q}$ is not covered by Proposition \ref{prop.loc.paramisiur}. But still we can verify the local connectivity with a similar argument. 
\begin{proposition}\label{prop.loc.connec.double-parabolic}
Let $\mathcal{Q}^{\pm}_n\subset\mathcal{S}_m$ be the unique two puzzle pieces of depth $n$ containing $\boldsymbol{\mathrm{a}}_m$ on their boundaries. Without loss of generality suppose that their boundaries all contain $\mathcal{R}_\infty(\alpha^+_m)*(0,r^{1/3^n})$. Then $\bigcap_n\mathcal{Q}_n^\pm = \emptyset$.
\end{proposition}

\begin{proof}
     We only prove for the sequence $\{\mathcal{Q}_n^+\}_n$, the case left is similar. Suppose the contrary $\bigcap{\mathcal{Q}^+_n}\not = \emptyset$. Recall that in Lemma \ref{lem.external.angle.limit} there is a sequence of external rays $\mathcal{R}_\infty(t_n)$ landing at $\partial\mathcal{B}_m$ with $t_k$ in the preimage of $\Theta_m$ and $t_n$ converging monotonously to $\beta_m^+$ as $n\to \infty$. In particular $\mathcal{R}_\infty(t_n)*(0,r^{1/3^n})\subset\partial\mathcal{Q}^+_n$. Therefore for $a\in\mathcal{Q}^+_n$, the critical value $v_-(a)$ is contained in $Q^a_n$, the puzzle piece whose boundary contains $R^{\infty}_a(\alpha_m^+)*(0,r^{1/3^n}),R^{\infty}_a(t_n)*(0,r^{1/3^n})$. Now if $a\in\bigcap_n\mathcal{Q}^+_n$, then $a\in\mathcal{C}_{\lambda}$ is not Misiurewicz parabolic and not in any parabolic component. This contradicts Proposition \ref{prop.localconnec.0} which tells us that $\bigcap_{n}Q^a_n=\emptyset$.
\end{proof}

\subsection{Non renormalizable case}\label{subsec.loc.connec.Non-renormalizable}
In this subsection, we always fix some $a_0\in \mathcal{S}^+_{m}$ such that $f_{a_0}$ satisfies the Assumption ($\Diamond$). Consider the corresponding graphs $\mathcal{X}_n$ and puzzle pieces $\mathcal{P}_n(a_0)$.

We give the relation between para-puzzles and dynamical puzzles:
\begin{lemma}\label{lem.homeo.para.dym}
Let $n\geq 0$. The mapping $H_n:\mathcal{P}_n(a_0)\cap\mathcal{X}_{n+1}\longrightarrow P^{a_0,v}_{n}\cap X_{n+1}^{a_0}$ defined by $H_n(a) = (L^a_{n+1})^{-1}(v_-(a))$ is injective. Moreover there exists $N\geq 0$ such that $\forall n\geq N$, $H_n$ is surjective.
\end{lemma}

\begin{proof}
First we prove that $H_n$ is well-defined. By definition of $\mathcal{X}_n$, we have $v_-(a)\in P_n^{a,v}\cap X_{n+1}^{a}$. Consider the holomorphic motion starting at $a$: $\Tilde{L}_{n+1}(a',z) := L_{n+1}(a',(L^a_{n+1})^{-1}(z))$. By continuity there exists a disk $B(a,r)$ on which $v_-(a')$ is surrounded by $\Tilde{L}_{n+1}(a',\partial P_n^{a,v})$. Pick any $a'$ on $\partial B(a,r)$, if $a'\not\in \partial\mathcal{P}_n(a_0)$, then for any $a''$ in a small disk $B(a',r')$, $v_-(a'')$ is surrounded by $\Tilde{L}_{n+1}(a'',\partial P_n^{a,v})$. Hence we can extend this property step by step until we reach $\partial \mathcal{P}_n(a_0)$. In particular $H_n(a)\in P^{a_0,v}_{n}$. By Lemma \ref{lem.holomotion.misur}, $\partial P_{n+1}^{a_0,v}$ moves holomorphically when $a\in \mathcal{P}_n(a_0)$, hence $H_n(a)\in X_{n+1}^{a_0}$.

Next we verify injectivity. The injectivity is clear when $a$ belongs to parameter external rays and equipotentials since in the dynamical plan, angles of external rays and equipotentials are preserved by $L^a_{n+1}$. We will only prove injectivity for $a$ belongs to parameter internal rays. The proof for internal equipotentials is similar. Suppose there are two distinct parameters $a,a'$ such that $H_n(a) = H_n(a')$. Then clearly $a,a'$ belong to different components $\mathcal{U},\mathcal{U}'$. Set $a\in \overline{\mathcal{R}_\mathcal{U}},a'\in\overline{\mathcal{R}_\mathcal{U'}}$. 
\begin{itemize}
    \item First suppose that the landing point $b,b'$ of $\mathcal{R}_\mathcal{U},\mathcal{R}_\mathcal{U'}$ do not coincide. Consider the external rays $\mathcal{R}_{\infty}(s),\mathcal{R}_{\infty}(s')$ involved in $\mathcal{X}_{n+1}$ landing at $b,b'$ respectively. Then the two rays $H_n(\mathcal{R}_{\infty}(s)),H_n(\mathcal{R}_{\infty}(s'))$ land at a common point $x(a_0)$ since $H_n(a) = H_n(a')$. But this is impossible since $f_{a_0}$ is injective near the forward orbit of $x(a_0)$.
    \item Next if $b = b'$. There are two possibilities: $H_n(\mathcal{R}_{\infty}(s))=H_n(\mathcal{R}_{\infty}(s'))$ or they are two different rays landing at a common point $x(a_0)$. While in the graph $X^{a_0}_{n+1}$, only one internal ray lands at $x_(a_0)$, so the second is impossible. Then there exists two other internal rays $\tilde{\mathcal{R}}_{\mathcal{U}}\subset\mathcal{U},\tilde{\mathcal{R}}_{\mathcal{U}'}\subset\mathcal{U'}$ landing at $\Tilde{b},\Tilde{b}'$ respectively which have the same image under $H_n$. Then $\Tilde{b}\not = \Tilde{b}'$, for otherwise one finds a loop in $\overline{\mathcal{U}\cup\mathcal{U}'}$ surrounding points in $\mathcal{H}_{\infty}$. We repeat the argument above to $\Tilde{b},\Tilde{b}'$ and get a contradiction.
\end{itemize}

Finally we verify surjectivity. First we prove 
\begin{claim*}
$\partial\mathcal{P}_n(a_0)\cap(\mathcal{R}_\infty(\alpha^+_{m-1})\cup\mathcal{R}_\infty(\beta^+_{m}))=\emptyset$ for $n$ large enough.
\end{claim*}
\begin{proof}[Proof of the claim]
    Suppose for example $\partial\mathcal{P}_0(a_0)$ contains parts of $\beta^+_m$ (the proof proceeds the same way if it is $\alpha^+_{m-1}$). Recall in Lemma \ref{lem.external.angle.limit}, there is a sequences of external rays $\mathcal{R}_{\infty}(s_n)$ landing at $a_n\in\partial\mathcal{B}_m$ with $s_n\to\beta^+_{m})$ and $a_n$ Misiurewicz parabolic. For $n$ large enough, let $\mathcal{R}_n\subset\mathcal{B}_m\cap\mathcal{X}_n\cap\mathcal{P}_0(a_0)$ be a sequence of internal rays landing at $b_n\in\partial\mathcal{B}_m$. By Lemma \ref{lem.para.Misiur}, there exists $\mathcal{R}_\infty(t_n)\subset\mathcal{X}_n$ landing at $b_n$ with $|t_n-\beta_m^+|\textless |s_n-\beta_m^+|$. Thus if $\partial\mathcal{P}_n(a_0)\cap\mathcal{R}_\infty(\beta^+_{m}) \not= \emptyset$, then $a_0\in\mathcal{Q}^+_n$, where $\mathcal{Q}^+_n$ is as defined in Proposition \ref{prop.loc.connec.double-parabolic}. Recall there we have shown that $\bigcap_n\mathcal{Q}^+_n = \emptyset$. Hence $\partial\mathcal{P}_n(a_0)\cap\mathcal{R}_\infty(\beta^+_{m}) = \emptyset$ for $n$ large enough.
\end{proof}

$H_n$ is therefore surjective on the part of external equipotential and external rays by its injectivity and the claim. So we verify surjectivity for internal rays and for internal equipotentials will be similiar. Let $z_0\in X_{n+1}^{a_0}\cap B_{a_0}(0)\cap P^{a_0,v}_n$. Suppose that $z_0$ is on some connected component $\zeta$ of $f^{-k}_{a_0}(R_0^{a_0})$. Let $x_0$ be the landing point of $\zeta$, then there is an external ray $R^{\infty}_{a_0}(t)\subset X_{n+1}^{a_0}$ landing at $x_0$. By Lemma \ref{lem.para.Misiur}, $\mathcal{R}_{\infty}(t)$ lands at a Misiurewicz parameter $a'$ which is accessible by some internal ray $\mathcal{R}_{\mathcal{U}}$ (Lemma \ref{lem.landing.Misiure.internal-rays}). Clearly $H_n(\mathcal{R}_{\mathcal{U}}) = \zeta$.
\end{proof}

\begin{corollary}\label{cor.criticalpuzzle}
Let $a\in \mathcal{P}_{n-1}(a_0)$. Let $C^a_n$ be the puzzle piece bounded by $L_n(\partial P^{a_0,v}_n)$. Then $a\not\in \overline{\mathcal{P}_{n}(a_0)}$ if and only if $v_-(a)\not\in \overline{C^a_n}$.
\end{corollary}
\begin{proof}
If $a\not\in \overline{\mathcal{P}_{n}(a_0)}$, then take a simple path $a_t\subset \mathcal{P}_{n-1}(a_0)$ connecting $a_0,a$ such that $a_t\cap\mathcal{P}_{n}(a_0)$ only contains one point. Thus Lemma \ref{lem.homeo.para.dym} ensures that once $a_t$ goes out of $\overline{\mathcal{P}_{n}(a_0)}$, $f_{a_t}(c_-(a_t))$ will never enter again $\overline{C^a_n}$. Conversely if $a\in\mathcal{P}_n(a_0)$, then clearly $v_-(a)\in C^a_n$ since $\partial P^{a_0,v}_n$ moves holomorphically and $v_-(a)$ does not intersect $X^a_n$.
\end{proof}

\begin{corollary}\label{cor.homeo.dym.parapuzzle}
For $n$ large enough, if $\overline{P^{a_0,v}_{n+1}}\subset P^{a_0,v}_n$, then $\overline{\mathcal{P}_{n+1}(a_0)}\subset\mathcal{P}_n(a_0)$. Moreover \[H_n:\partial\mathcal{P}_{n+1}(a_0)\longrightarrow\partial P^{a_0,v}_{n+1}\] is a homeomorphism.
\end{corollary}
\begin{proof}
By Lemma \ref{lem.homeo.para.dym} $H^{-1}_n(\overline{\partial P^{a_0}_{n+1}})\subset \mathcal{P}_n(a_0)$ bounds a puzzle piece $\mathcal{P}$. So it suffices to prove that $a_0\in\mathcal{P}$. Suppose not, then there is $a'\in\partial\mathcal{P}$ but $a'\not \in \overline{\mathcal{P}_{n+1}(a_0)}$. By Corollary \ref{cor.criticalpuzzle}, $v_-(a')\not\in\overline{C^{a'}_{n+1}}$, a contradiction.
\end{proof}

\begin{corollary}\label{cor.critical-value-on-boundary}
Let $\mathcal{U}\subset\mathcal{S}_m^+$ be a capture component. If $a_0\in\partial\mathcal{U}$ or $\partial\mathcal{B}_m$ is non renormarlisable, then $v_-(a_0)$ eventually hit the boundary of immediate basins.
\end{corollary}
\begin{proof}
By Corollary \ref{cor.homeo.dym.parapuzzle}, $\partial P^{a_0,v}_n$ intersects a Fatou component $U$ (preimage to some immediate basin) for all $n$. Hence by Theorem \ref{thm.dym} $\bigcap_n\overline{P^{a_0,v}_n}$ intersects $\partial U$ at $v_-(a_0)$.
\end{proof}

By Theorem \ref{thm.dym}, there exists graph (\ref{eq.graph.Si}), such that in the dynamical plane of $f_{a_0}$ one has a sequence of non-degenerated annuli $A^{a_0}_{n_i} := P^{a_0,v}_{n_i}\setminus\overline{P^{a_0,v}_{n_i+1}}$. The above corollary implies that in the parameter plane, $\mathcal{A}^{a_0}_{n_i} := \mathcal{P}_{n_i}(a_0)\setminus\overline{\mathcal{P}_{n_i+1}}(a_0)$ is also non-degenerated for $i$ large enough. Applying Shishikura's trick \cite{roesch_2000}, we can get the distortion control between the moduli of para-annuli and dynamical annuli:

\begin{lemma}\label{lem.shishikura}
There exists $K\textgreater 1$ such that for $i$ large enough
\[\frac{1}{K}mod(A^{a_0}_{n_i})\leq mod\mathcal(\mathcal{A}_{n_i})\leq Kmod(A^{a_0}_{n_i}).\]
\end{lemma}

\begin{corollary}\label{cor.loc.connec.non-renormalisable}
$\mathcal{C}_{\lambda}$ is locally connected at non renormalisable parameters.
\end{corollary}
\begin{proof}
This follows immediately from  the above lemma and Grötzsch's inequality.
\end{proof}

\subsection{Renormalizable case}\label{subsec.renor}

In this subsection, we always fix some $a_0\in \mathcal{S}^+_{m}$ such that $f_{a_0}$ satisfies the Assumption ($\Diamond$) and is renormalisable. Recall that a set $\textbf{M}'\subset \mathbb{C}$ is called a copy of the Mandelbrot set (cf. \cite{ASENS_1985_4_18_2_287_0}) $\textbf{M}$ if there exists $k\geq 1$ and a homeomorphism $\chi:\textbf{M}'\longrightarrow\textbf{M}$ such that $f_a$ is $k-$renormalisable and $f^k_a$ is quasiconformally conjugated to $z^2+\chi(a)$.

\begin{proposition}\label{prop.mandel}
$\textbf{M}_{a_0} := \bigcap\mathcal{P}_n(a_0)$ is a copy of the Mandelbrot set. Moreover, there are exactly 2 external rays landing at the cusp $\chi^{-1}(\frac{1}{4})$, separating $\textbf{M}_{a_0}$ with $\mathcal{B}_m$. In particular, if $a_0\in\partial\mathcal{B}_m$, then $a_0 = \chi^{-1}(\frac{1}{4})$, $\partial\mathcal{B}_m$ is locally connected at $a_0$.
\end{proposition}

\begin{proof}
For the first statement, see \cite[Prop. 3.26]{Roesch1}. As for the second statement, the proof for the existence is similar to Step 1, Step 2 in Lemma \ref{lem.landing.mandelbrot.Per0}; for the uniqueness, see the proof of \cite[Thm. 3]{Roesch1}
\end{proof}

\begin{remark}
By the above proposition, the wake $\mathcal{W}(a)$ and limb $\mathcal{L}(a)$ can be defined in the classical sense.
\end{remark}

\begin{proposition}\label{prop.renor.capture}
Let $\mathcal{U}\subset\mathcal{S}_m^+$ be a capture component. Suppose $a_0\in\partial\mathcal{U}$ is renormalisable. Then $a_0$ is a tip of $\textbf{M}_{a_0}$, with $\textbf{M}_{a_0}\cap\mathcal{B}_m = a_1$. Moreover there are exactly two external rays landing at $a_0$, separating $\mathcal{U}$ with $\textbf{M}_{a_0}$. In particular, $\partial\mathcal{U}$ is locally connected at $a_0$.
\end{proposition}

\begin{proof}
See \cite[Prop. 4.15, Thm. 4]{Roesch1}.
\end{proof}

\begin{proof}[\textbf{Proof of Theorem~\ref{thm.A}}]
Let $\mathcal{U}$ be an open connected component of $\mathcal{H}^{\lambda}$. Corollary \ref{cor.loc.connected.cap-in-double-wake} solves already the case where $\mathcal{U}$ is contained in a double parablic wake. So suppose $\mathcal{U}$ is in the complement. From the discussion in \ref{subsec.loc.connec.Misiurewicz-parabolic}, \ref{subsec.loc.connec.Non-renormalizable} and \ref{subsec.renor} we see that $\partial\mathcal{U}$ is locally connected. It is then a Jordan curve by Lemma \ref{lem.boundary.components}.
\end{proof}

\subsection{Global descriptions}\label{sub.sec.global}
\begin{lemma}\label{lem.landing.Misiur.parabo}
 Let $a_0\in \mathcal{C}_{\lambda}$ be Misiurewicz parabolic. Suppose moreover that $a_0$ is not in any double parabolic wake. Then there are $q$ external rays landing at it, each two adjacent rays bound a parabolic component.
\end{lemma}
\begin{proof}
One can prove the existence of parabolic component attached to $a_0$ by showing that there are parameter equpotentials rays landing at $a_0$. To prove the existence of landing rays, one uses holomorphic motion and Rouché's Theorem, similar to the proof of Proposition \ref{prop.summary.para.adj}.
\end{proof}

Let $a_0\in\partial\mathcal{B}_m$ be Misiurewicz parabolic. By Lemma \ref{lem.landing.Misiur.parabo}, there are $q$ external rays (whose angles are preimages of angles in $\Theta_m$) landing at $a_0$ each two adjacent rays separate a capture component attached at $a_0$ with $\mathcal{B}_m$. Applying a similar argument as in the proof of Lemma \ref{lem.external.angle.limit}, one can show that when a sequence of Misiurewicz parabolic parameters $a_n\in\partial\mathcal{B}_m$ converges to $a_0$, the angles of external rays landing at $a_n$ converges to the biggest or smallest angle among the angles of the $q$ external rays landing at $a_0$. Therefore there are \textbf{exactly} $q$ external rays landing at $a_0$ and \textbf{exactly} $q-1$ capture components attached at $a_0$. A similar result holds for the capture components attached at $\partial\mathcal{B}_m$, also for the next "generation" of capture components attached, and so on.

\begin{definition}
Let $0\leq m\leq \lfloor\frac{q}{2}\rfloor$. A capture component $\mathcal{U}_1\subset\mathcal{S}_m^+$ is of level 1, if it is attached to $\partial\mathcal{B}_m$ at a Misiurewicz parabolic parameter. A capture component $\mathcal{U}_n$ is of level $n\geq 2$ if it is not of level $n-1$, and is attached to some $\partial\mathcal{U}_{n-1}$ at a Misiurewicz parabolic parameter with $\mathcal{U}_{n-1}$ of level $n-1$.  
The point $\boldsymbol{\mathrm{r}}_n = \partial\mathcal{U}_n\cap\partial\mathcal{U}_{n-1}$ is called the root of $\mathcal{U}_n$. The open region bounded by the two external rays containing $\mathcal{U}_n$ is called the wake of $\mathcal{U}_n$, denoted by $\mathcal{W}_{\mathcal{U}_n}$. The corresponding limb is defined to be $\mathcal{L}_{{\mathcal{U}_n}} = \overline{\mathcal{W}_{\mathcal{U}_n}}\cap\mathcal{C}_{\lambda}$.
\end{definition}

\begin{definition}\label{def.central-component}
Set $\mathscr{K}^m_0 = \mathcal{B}_m\cap\mathcal{S}_m^+$, $\mathscr{K}^m_n$ to be the union of all level $n\geq 1$ capture components in $S^+_m$, $\mathscr{K}^m :=\overline{\bigcup_{n\geq0}\mathscr{K}^m_n}$ is called a \textbf{central component}. Let $\mathscr{K} := \bigcup_{m=0}^{\lfloor q/2\rfloor}(\mathscr{K}^m$. Similarly, if we work in $\mathcal{S}^-_m$, we can construct $\tilde{\mathscr{K}}^m_n,\tilde{\mathcal{K}}^m,\tilde{\mathcal{K}}$.
\end{definition}

\paragraph{Construction of $\mathfrak{G}$ in the Main Theorem.} Recall $P_{\lambda}(z) = \lambda z+z^2$, $K_\lambda$ its filled-in Julia set, $\Omega^0_0$ its maximal petal contained in $B^*_0(0)$, $\Omega^k_l$ defined as in (\ref{eq.sequence}). Let $D_0,\tilde{D}_0$ be the two connected components of $B^*_0(0)\setminus\overline{P^{-1}_\lambda(\Omega^{-1}_p)}$, such that $D_0$ is on the left-hand side of $\tilde{D}_0$. Let $\tilde{\Omega} = P_\lambda^{-1}(\Omega^{-1}_p)\setminus\Omega_0^0$. Recall also the parametrisations in \ref{subsec.summary.fa}: $\Psi^{adj}_0:\mathcal{B}_m\setminus\mathcal{I}_0\longrightarrow B_p^*(0)$, $\Psi^{bit}_m:\mathcal{D}_m\setminus\mathcal{I}_m\longrightarrow B^*_{\overline{m+p}}(0)\setminus\overline{\Omega^{-s}_{\overline{m+p}}}$ and $\Psi_{\mathcal{U}_k}:\mathcal{U}_k\longrightarrow B^*_l(0)$. By Proposition \ref{prop.injec.parame.adj} and \ref{prop.summary.para.adj}, the following mapping is an isomorphism from $\mathcal{B}_0\setminus\mathcal{I}_0$ to $B^*_0(0)\setminus\overline{\Omega^0_0}$:
\begin{equation}\label{eq.critical.adj}
   a\mapsto \left\{
\begin{aligned}
&(P_\lambda|_{D_0})^{-1}\circ\Psi_0^{adj}(a),\,\,\text{if } a \in \mathcal{D}_0,\\
&(P_\lambda|_{\tilde{D}_0})^{-1}\circ\Psi_0^{adj}(a),\,\,\text{if } a \in \tilde{\mathcal{D}}_0,\\
&[(P_\lambda|_{\overline{\Omega}\setminus\{-\frac{\lambda}{2}\}})^{-1}\circ\Psi_0^{adj}(a),\,\,\text{if } a \in \overline{W}\setminus\{\sqrt{3\lambda}\}
\end{aligned}
\right.
\end{equation}
Next, by Proposition \ref{prop.summary.parametri}, the following mapping is an isomorphism from $\mathcal{D}_m\setminus\mathcal{I}_m$ to $B^*_m(0)\setminus\overline{\Omega_m^{-s+1}}$, where $2\leq s\leq q$ is the unique integer such that such that $-sp+m+p=0\pmod q$:
\begin{equation}\label{eq.critical.bit}
    a\mapsto(P_\lambda|_{B^*_m(0)})^{-1}\circ\Psi^{bit}_m(a)
\end{equation}
Finally, for a capture component $\mathcal{U}_n$ of $\mathscr{K}_n^m$, we define its address $ [\mathcal{U}_n]$ as the address of the Fatou component containing $c_-(a)$ for $a\in\mathcal{U}_n$ (recall \ref{eq.parametri.julia}). By Proposition \ref{prop.paramet.cap} the mapping defined below
\begin{equation}\label{eq.critical.cap}
    a\mapsto(P_\lambda|_{U_n})^{-1}\circ\Psi_{\mathcal{U}_n}(a)
\end{equation}
is an isomorphism from $\mathcal{U}_n$ to $U_n$, where $U_n$ is the Fatou component of $P_\lambda$ with address $[\mathcal{U}_n]$. Combining (\ref{eq.critical.adj}) (\ref{eq.critical.bit}) (\ref{eq.critical.cap}), we have constructed an bijection ($s=1$ if $m=0$)
\[\mathfrak{G}^m:\bigcup_{n\geq0}\mathscr{K}^m_n\longrightarrow (\mathring{K_\lambda}\cap S_m)\setminus\overline{\Omega^{-s+1}_m}.\]
For $1\leq m\leq\lfloor\frac{q}{2}\rfloor$, if we work in $\mathcal{S}_m^-$, then similarly we can define $\tilde{\mathscr{K}}_n^m$, $\tilde{\mathscr{K}}^m$ and construct an bijection
\[\tilde{\mathfrak{G}}^m:\bigcup_{n\geq0}\tilde{\mathscr{K}}^m_n\longrightarrow (\mathring{K_\lambda}\cap S_{q-m})\setminus\overline{\Omega^{l+1}_{q-m}}.\]
where $0\leq l\leq q-1$ is the unique integer such that $lp+p-m=0\pmod q$. So to prove the Main Theorem,
it remains to extend $\mathfrak{G}^m$ (resp. $\tilde{\mathfrak{G}}^m$) to $\mathscr{K}^m\cap\mathcal{S}^+_m$ (resp. $\tilde{\mathscr{K}}^m\cap\mathcal{S}^-_m$). We will only do the extension for $\mathfrak{G}^m$, since it is the same for $\tilde{\mathfrak{G}}^m$.

\paragraph{The set  $\mathscr{K}^m\setminus\bigcup_{n\geq0}\mathscr{K}^m_n$.} 
By Proposition \ref{prop.renor.capture}, there are no renormalisable parameter on $\partial\mathcal{U}_n$ for $n \geq 1$. By Corollary \ref{cor.critical-value-on-boundary}, if
$a \in\partial\mathcal{U}_n$, $c_-(a)$ eventually hits the boundary of immediate basins for $a \in\partial\mathcal{U}_n$. For $a \in\partial\mathcal{U}_n$, define its address by
 \begin{equation}\label{eq.parame.finite.H}
     [[a]] := [c_-(a_n)] = [(m,\omega_0);(\epsilon_1,\omega_1);...;(\epsilon_n,\omega_n)].
 \end{equation}
where $\omega_0,...\omega_{n-1}$ are dyadic, $\omega_n\in(0,1)$; $0\leq \epsilon_k\leq q-1$ is chosen so that $f^k_{a_n}(c_-(a_n))\in S_{a,\epsilon_k}$ ($S_{a,\epsilon_k}$ is the sector $S_{\epsilon_k}$ in \ref{subsec.wake-end-points}). If $a\in\partial\mathcal{B}_m\setminus\mathcal{A}_{p/q}$, define $[[a]]$ to be $[r_a]$ if $a$ is renormalisable, where $r_a\in \partial B^*_{a,m}(0)\setminus\{0\}$ is the parabolic periodic point (Proposition \ref{prop.mandel}); otherwise $[[a]] :=  [c_-(a_n)]$.    \\

\begin{lemma}
For $a_0\in\partial\mathcal{U}_n$ or $\partial\mathcal{B}_m$ non renormalisable, there is a unique external ray landing at it.
\end{lemma}
\begin{proof}
This comes from Proposition \ref{prop.non-trivial-limb} and the homeomorphism between para-puzzles and dynamical puzzles, cf. Corollary \ref{cor.homeo.dym.parapuzzle}.
\end{proof}

\begin{lemma}\label{lem.decompo.para}
We have decompositions
\[\mathcal{S}_m^+\cap\mathcal{C}_{\lambda} = (\overline{\mathcal{B}_m}\cap\mathcal{S}_m^+)\cup\left(\bigcup_{\mathcal{U}_1\subset\mathscr{K}^m_1}\mathcal{L}_{{\mathcal{U}_1}}\cup\bigcup_{\substack{a\in\partial\mathcal{B}_m\cap\mathcal{S}_m^+\\\text{parabolic}}}\mathcal{L}(a)\right)\]
\[\mathcal{L}_{{\mathcal{U}_n}}= \overline{\mathcal{U}_n}\cup\bigcup_{\substack{\mathcal{U}_{n+1}\subset\\\mathscr{K}^m_{n+1}\cap\mathcal{W}_{{\mathcal{U}_n}}}}\mathcal{L}_{{\mathcal{U}_{n+1}}}.
\]
As a direct consequence, $\overline{\mathscr{K}^m_n}\setminus\mathcal{A}_{p/q} = \bigcup_{\mathcal{U}_n\subset\mathscr{K}^m_n}\overline{\mathcal{U}_n}$.
\end{lemma}
\begin{proof}
This is a direct consequence of the above lemma. Let us do the proof for $\mathcal{S}_m^+\cap\mathcal{C}_{\lambda}$. For any $n\geq0$, consider all the Misiurewicz parabolic parameters of depth $n$ (that is, $f^n_a(v_-(a)) = 0$) on $\partial\mathcal{B}_m\cap{S}_m^+$ and the $q$ unique landing external rays. These rays together with $\partial\overline{\mathcal{B}_m}$ separate $\mathcal{S}_m^+\setminus\overline{\mathcal{B}_m}$ into several open sectors of depth $n$. Take $b\in{\mathcal{C}}_{p/q}\cap{(\overline{\mathcal{U}_0}})^c\cap\mathcal{S}_m^+$ not in any $\mathcal{L}_{{\mathcal{U}_1}}$. Let $\mathcal{S}_n(b)$ be the sector of depth $n$ containing $b$. Let $\mathcal{R}_{\infty}(t_n),\mathcal{R}_{\infty}(t'_n)$ be the two external rays bounding $\mathcal{S}_n(b)$ and $a_n,a'_n$ their landing point repectively. Clearly $a_n,a'_n$ converges to some $a\in\partial\mathcal{B}_m$ since $\partial\mathcal{B}_m$ is a Jordan curve. Then $a$ must be renormalisable. If not, then by the previous lemma, $t_n,t'_n$ converge to the same angle, which implies that $b = a$, a contradiction since we take $b\not\in\overline{\mathcal{B}_m}$.
\end{proof}

\begin{proposition}\label{prop.landing.mathscrH}
Let $a_0\in\mathscr{K}^m\setminus\bigcup_{n\geq0}\overline{\mathscr{K}^m_n}$ not be double parabolic. Then $a_0$ is contained in a infinite sequence of limbs, i.e. there exist $\mathcal{U}_n\subset\mathscr{K}_n$ such that $a_0\in\bigcap_n\mathcal{L}({\boldsymbol{\mathrm{r}}_{\mathcal{U}_n}})$. Moreover $a_0$ can not be Misiurewicz parabolic. If $a_0$ is non renormalisable, there is only one external ray landing at it; if it is renormalisable, then it is the cusp of $\textbf{M}_{a_0}$ and exactly two external rays land at it.
\end{proposition}
\begin{proof}
The existence of $\mathcal{U}_n$ is just a direct consequence of Lemma \ref{lem.decompo.para}. Next we prove that $a_0$ is not Misiurewicz parabolic. Suppose the contrary that $f_{a_0}^N(v_-(a_0)) = 0$. Let $a_n\in\overline{\mathcal{U}_{N_n}}\subset\mathscr{K}^m_{N_n}$ be a sequence converging to $a_0$ and suppose $f^{M_n}_{a_n}(v_-(a_n)) = 0$. Clearly $M_n\geq n$. Fix $N'>>N$, set $N'':=M_{N'}-N$. Let $\tilde{B}_{N''}$ be the connected component containing 0 of $f^{-N''}_{a_0}(\overline{\bigcup_i B_{a_0,i}^*(0)})$. Notice that there exists a neighborhood $\mathcal{V}$ of $a_0$ on which there is a dynamical holomorphic motion $h_a$ of $\tilde{B}_{N''}$. By construction, $f_{a_n}^N(v_-({a_n}))$ is bounded by a wake attached at a preimage of 0 (not equal to 0) on $\partial (h_a(\tilde{B}_{N''}))$. Moreover the angles of the two external rays $R^\infty_{a_n}(t_1),R^\infty_{a_n}(t_1)$ determining this wake do not depend on $n$. Shrink $\mathcal{V}$ if necessary so that there is a holomorphic motion of $R^\infty_{a_0}(t_1)\cup R^\infty_{a_0}(t_2)$. This implies that $f^N_{a_n}(v_-(a_n))$ does not converge to $0$ as $n\to\infty$, contradicting $a_n\to a_0$.

So $a_0$ is either non renormalisable or renormalisable. Suppose the first case. The existence of external ray landing comes from the shrinking property of parameter puzzles around $a_0$ (Theorem \ref{thm.dym} and Lemma \ref{lem.shishikura}). The uniqueness comes from Lemma \ref{lem.homeo.para.dym} and Lemma \ref{lem.unique.external.end-point}. Next suppose $a_0$ is renormalisable. By Proposition \ref{prop.mandel}, $\textbf{M}_{a_0}$ is separated from $\mathscr{K}^m$ by two external rays landing at the cusp. While $a_0\in\mathscr{K}^m$, it has to be the cusp.
\end{proof}

We conclude that $\partial\mathscr{K}^m$ is combinatorically rigid:
\begin{corollary}\label{cor.rigidity.mathscrH}
Suppose $a_0\in\mathscr{K}^m\setminus\bigcup_{n\geq0}\overline{\mathscr{K}^m_n}$. Let $\{\mathcal{U}_n\}_n$ be as in Proposition \ref{prop.landing.mathscrH}. Then $\bigcap_n\mathcal{L}({{\mathcal{U}_n}})\cap\mathscr{K}^m = \{a_0\}$.
\end{corollary}
\begin{proof}
This comes from the shrinking property of puzzle pieces $\mathcal{P}_n(a)$ when $a$ is non renormalisable; for $a$ renormalisable, it comes from the shrinking property of $\mathcal{P}_n(a)\setminus\overline{\mathcal{W}(a)}$.
\end{proof}

\begin{theorem}\label{thm.decompo.para}
We have the decomposition
\[\mathcal{S}_m^+\cap\mathcal{C}_{\lambda} = \mathscr{K}^m\cap\mathcal{S}_m^+\cup\bigcup_{\substack{a\in\partial\mathscr{K}^m\\\text{renormalisable}}}\mathcal{L}(a).\]
In particular, $\mathcal{K}_\lambda=\mathscr{K}\cup-\mathscr{K}\cup\tilde{\mathscr{K}}\cup-\tilde{\mathscr{K}}$.
\end{theorem}
\begin{proof}
Here $\mathscr{K}$ is parallel to $\overline{\mathcal{B}_m}$ in Lemma \ref{lem.decompo.para}, and instead of considering sectors defined by rays landing at $\partial\mathcal{B}_m$, one should consider sectors defined by rays landing at $\partial\mathscr{K}^m$. One then concludes by using the landing property given in Proposition \ref{prop.landing.mathscrH}.
\end{proof}

\paragraph{Extension of $\mathfrak{G}^m$.}
By Proposition \ref{prop.landing.mathscrH}, every $a\in\partial\mathscr{K}^m\setminus\bigcup_n\partial\mathscr{K}^m_n$ can be represented by a infinite sequence
\[[[a]] := [(m,\omega_0);(\epsilon_1,\omega_1);(\epsilon_2,\omega_2);...;(\epsilon_n,\omega_n),...]\]
with $\omega_k$ dyadic, $0\leq \epsilon_k\leq q-1$ are chosen so that $f^k_{a_n}(c_-(a_n))\in S_{a,\epsilon_k}$. Compare with Remark \ref{rem.coordi.quadratic.model} the way we encode $J_{\lambda}\setminus\{0\}$, where $J_{\lambda}$ is the Julia set of the quadratic model $P_{\lambda} = \lambda z+z^2$.

\begin{proof}[\textbf{End of the proof of the Main Theorem}]
First of all, $\mathfrak{G}^m$ extends continuously to $\bigcup_n\overline{\mathscr{K}_n^m}$, since by Lemma \ref{lem.decompo.para},  $\overline{\mathscr{K}_n^m}\cap\mathcal{S}_m^+ = \bigcup_{\mathcal{U}_n\subset\mathscr{K}^m_n}\overline{\mathcal{U}_n}$, and by Theorem \ref{thm.A}, every component of $\mathscr{K}_n^m$ is a Jordan curve. Let $\{a_k\}_k\subset\bigcup_n\overline{\mathscr{K}_n^m}$ be a sequence converging to some $b\in(\mathscr{K}^m\cap\mathcal{S}_m^+)\setminus\bigcup_n\overline{\mathscr{K}_n^m}$. Let $\mathcal{U}_k$ be the parabolic component whose closure containing $a_k$, and $\{\mathcal{U}_n(b)\}_n$ the sequence of parabolic components associated to $b$ given by Proposition \ref{prop.landing.mathscrH}. Then by Corollary \ref{cor.rigidity.mathscrH}, for any $n$, the first $n$ coordinates fo $[[\mathcal{U}_k]]$ will be the same for all $k$ large enough, i.e. equal to $[[\mathcal{U}_n(b)]]$. Define $\mathfrak{G}^m(b) := \lim\limits_{n\to\infty}\Xi^{-1}([[\mathcal{U}_n(b)]])$. Clearly $\mathfrak{G}^m|_{\partial\mathscr{K}^m\cap\mathcal{S}_m^+}$ coincides with the following mapping
\[\mathfrak{H}:\partial\mathscr{K}^m\cap\mathcal{S}_m^+\longrightarrow J_{\lambda}\cap S_{m},\,\,a\mapsto \Xi^{-1}([[a]]),\]
where $\Xi^{-1}$ is defined in (\ref{eq.encoding.julia}). Now we prove that $\mathfrak{H}$ is a homeomorphism. Clearly $\mathfrak{H}$ is continuous and surjective on $\partial\mathfrak{H}^m\cap\mathcal{S}^+_m$. Moreover it is injective on $\partial\mathscr{K}^m\cap\overline{\mathscr{K}^m_n}$ for all $n$. The injectivity on $\partial\mathscr{K}^m\setminus\bigcup_n\partial\mathscr{K}^m_n$ comes from the combinatorical rigidity (Corollary \ref{cor.rigidity.mathscrH}) i.e. the mapping $a\mapsto[[a]]$ is injective. It remains to verify continuity of . Take any sequence of $\{z\}_k\subset J_\lambda$ converging to some $w\in J_\lambda$. If $w$ is on the boundary of some Fatou Component, then the continuity of $\mathfrak{H}^{-1}$ at $w$ is guaranteed by Lemma \ref{lem.decompo.para}; if $w$ an end point, then the continuity is given by Corollary \ref{cor.rigidity.mathscrH}.
\end{proof}

\begin{proof}[\textbf{Proof of Corollary \ref{cor.blokh}}]
Clearly $\mathring{\mathcal{K}_\lambda}\subset\mathcal{CU}_\lambda$. Without loss of generality we may suppose $a\in\partial\mathscr{K}^m$. Then $a$ is either Misiurewicz parabolic, either non renormalisable, either double parabolic or renormalisable with a parabolic periodic point of multiplier 1 (Proposition \ref{prop.mandel} and \ref{prop.landing.mathscrH}). Moreover there is only one external ray landing at every repelling periodic point of $f_a$ (Proposition \ref{prop.non-trivial-limb} Lemma \ref{lem.unique.external.end-point}). Thus there are no repelling cut point in $J_a$, $a\in\mathcal{CU}_\lambda$.

Conversely, take $a\in\mathcal{CU}_\lambda$. Without loss of generality let $a\in\mathcal{S}^+_m$. By Theorem \ref{thm.decompo.para}, $a\in\mathscr{K}^m\subset\mathcal{K}_\lambda$.
\end{proof}

\begin{figure}[H] 
\centering 
\includegraphics[width=0.8\textwidth]{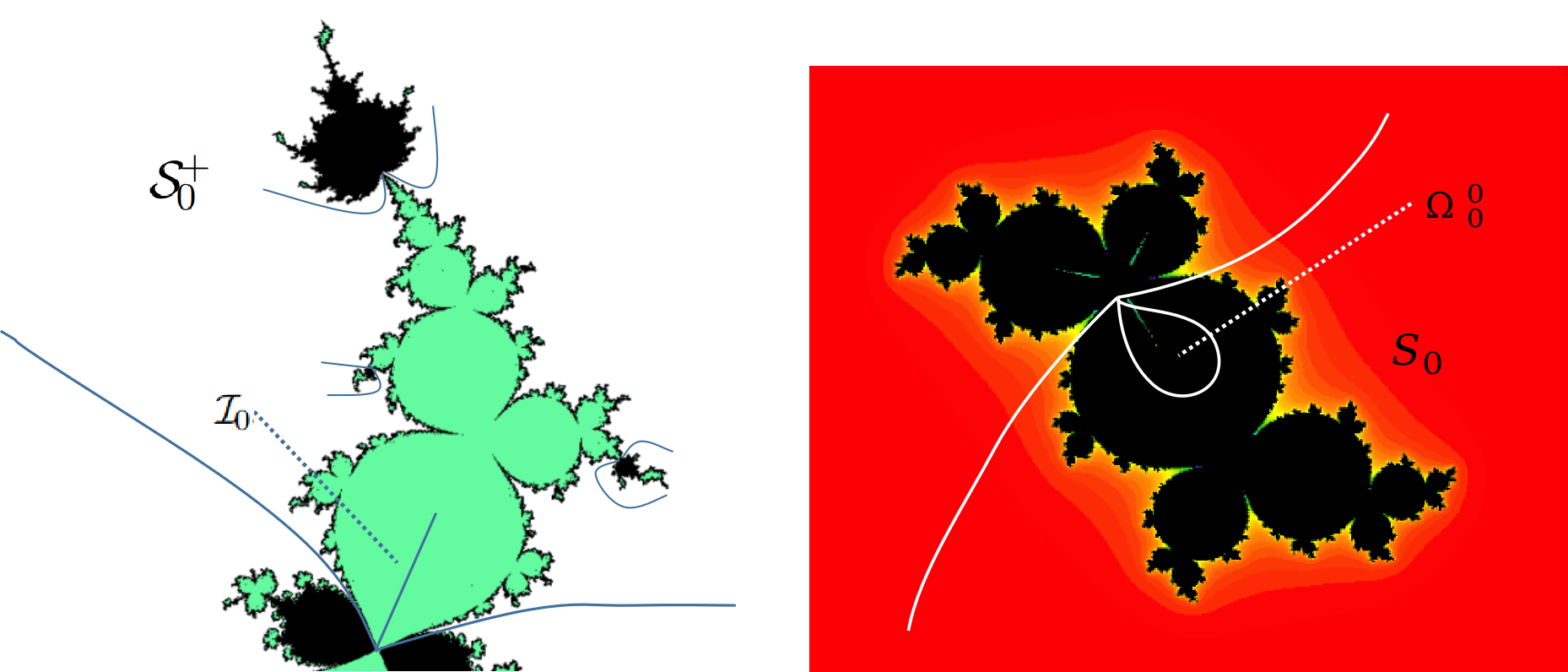} 
\caption{$\mathscr{K}^0\cap\mathcal{S}^+_0$ and $K_\lambda\cap S_0$ with $\lambda = e^{2\pi i\frac{1}{3}}$.} 
\label{fig.preimage} 
\end{figure}

\appendix
\section*{Appendices}
\addcontentsline{toc}{section}{Appendices}
\renewcommand{\thesubsection}{\Alph{subsection}}

\subsection{Rotation number}

\begin{definition}
Let $f:\mathbb{R}/\mathbb{Z}\longrightarrow\mathbb{R}/\mathbb{Z}$ be the $d$-fold covering map, i.e. $f(\theta) = d\theta$. A periodic cycle $\Theta = \{\theta_0,...,\theta_{q-1}\}$ has rotation number $\frac{p}{q}$ if $,\forall i, f(\theta_i) = \theta_{i+p \bmod q}$.
\end{definition}

\begin{theorem}[\cite{Goldberg_1992_4_25_6_679_0}]\label{thm.goldberg}
The $d-$fold mapping $\theta\mapsto d\theta$ has ${d+q-2}\choose{q}$ periodic cycles of rotation number $p/q$.
\end{theorem}

\begin{lemma}\label{lem.rotation.number}
If $\Theta$ has rotation number $p/q$, then $\Tilde{\Theta} = \{\Tilde{\theta}_0,...,\Tilde{\theta}_{q-1}\}$ has rotation number $1-\frac{p}{q}$, where $\Tilde{\theta}_i = 1-\theta_{q-i-1}$.
\end{lemma}
\begin{proof}
By definition, it suffices to prove that $f(\Tilde{\theta}_i) = \Tilde{\theta}_{i+q-p \bmod q}$:
\[f(\Tilde{\theta}_i) = d(1-\theta_{q-i-1}) = -d\theta_{q-i-1} = -\theta_{q-i-1+p \bmod q} =  -\theta_{-i-1+p \bmod q} = \Tilde{\theta}_{i+q-p \bmod q}.\]
\end{proof}

\subsection{Pinching deformation}\label{subsec.appendix_pinching}
This part is a resume of the theory of pinching deformation developed by Cui-Tan in \cite{CuiTan}.\\
Let $f:\hat{\mathbb{C}}\longrightarrow\hat{\mathbb{C}}$ be a rational map. Denote by $\mathcal{F}_f,\mathcal{J}_f$ the Fatou, Julia set respectively. Denote by $Crit_f$ the set of critical points of $f$. Define the post-critial set of $f$ by 
\[Post_f := \overline{\bigcup_{n\textgreater 0}f^n(Crit_f)}.\] 
$f$ is called \textbf{geometrically finite} if the accumulation points of $Post_f$ is finite. Define $\Tilde{\mathscr{R}}_f\subset\mathcal{F}_f$ as follows: $z\in\Tilde{\mathscr{R}}_f$ if and only if 
\begin{enumerate}
    \item $z$ does not belong to supper-attracting basins, 
    \item $\#(\{f^n(z);\,n\geq 0\})=+\infty$ and $\{f^n(z);\,n\geq 0\}\cap Post_f=\emptyset$.
\end{enumerate}
\begin{definition}
The \textbf{quotient space} $\mathscr{R}_f$ of $f$ is defined as the set $\Tilde{\mathscr{R}}_f$ quotient the grand orbit equivalence relation $\sim$, i.e. $z_1\sim z_2$ if and only if $\exists m,n\geq 0,\, f^m(z_1) = f^n(z_2)$. Denote by $\pi_f:\Tilde{\mathscr{R}}_f\longrightarrow\mathscr{R}_f$ the natural projection.
\end{definition}

\begin{remark}
$\mathscr{R}_f$ has only finitely many components, each of them is either an at least 1-punctured torus, corresponding to an attracting basin, of an at least 1-punctured infinite cylinder, corresponding to a parabolic basin. Refer to \cite{McMullen1998QuasiconformalHA} for details.
\end{remark}

In the sequel we always assume that $\Tilde{\mathscr{R}}_f\not=\emptyset$.

\begin{lemma}
Let $\gamma$ be a Jordan curve. Then either each component of $\pi_f^{-1}(\gamma)$ is a Jordan curve or each component of $\pi_f^{-1}(\gamma)$ is an eventually periodic arc. Moreover, every eventually periodic arc $\beta$ lands at both ends.
\end{lemma}
\begin{definition}
Let $\beta$ be a $k$-periodic arc and $b\in\beta$. The limit point of the forward (resp. backward) orbit of $b$ under $f^k$ is called the \textbf{attracting end} (resp. \textbf{repelling end}). Notice that the two ends might coincide if both of them are parabolic.
\end{definition}

\begin{definition}
A \textbf{multi-annulus} $\mathscr{A}\subset\mathscr{R}_f$ is a finite disjoint union of annuli whose boundaries are pairewise disjoint simple closed curves such that each component of $\pi_f^{-1}(e(\mathscr{A}))$ is an eventually periodic arc, where $e(\mathscr{A})$ is the union of the equators of the annuli in $\mathscr{A}$. A multi-annulus $\mathscr{A}$ is called \textbf{non-separating} if for any choice of finitely many components of $\pi_f^{-1}(e(\mathscr{A}))$, the closure of their union (denote by $K$) does not separate $\mathcal{J}_f$, i.e. there is only one component of $\hat{\mathbb{C}}\setminus K$ intersecting $\mathcal{J}_f$.
\end{definition}

\begin{definition}
A component $B$ of $\pi_f^{-1}(\mathscr{A})$ is called a \textbf{band}. Notice that $B$ is a topological open disk bounded by two eventually periodic arcs, and hence itself is also eventually periodic and lands at both ends. The \textbf{core arc} $\beta$ of $B$ is the lift of the equator of $\pi_f(B)$ to $B$, whose two (or maybe one) end points are exactly those of $B$. A band $B$ (resp. its core arc $\beta$) is of \textbf{level} $n$ ($n\geq 0$) if $n$ is the smallest integer such that $f^{n}(B)$ (resp. $f^{n}(\beta)$) is periodic. A \textbf{skeleton of level} $n$ is a connected component of $\overline{\bigcup\limits_{\text{level }n}\beta}$. The \textbf{fill-in} of a skeleton $S$, denote by $\hat{S}$, is the union of $S$ with all its complement components disjoint from $\mathcal{J}_f$. $\hat{S}$ is called a \textbf{fill-in skeleton}.
\end{definition}

For $r\textgreater1$, denote by $\mathbb{A}(r) := \{z;\,\frac{1}{r}\textless |z|\textless{r}\},\,\,\text{for }r \textgreater 1$. For $t\geq0$, let $w_{t,r}:\mathbb{A}(r)\longrightarrow\mathbb{A}(r^{1+t})$ be the pinching model defined as in \cite{CuiTan}, 5.1. Let $\mathscr{A}=\cup A_i$ be a non-separating multi-annulus. Let $\chi_i:A_i\longrightarrow \mathbb{A}(r_i)$ be a conformal representation and $\mu_{i,t}$ be the Beltrami differential of $w_{t,r_i}\circ{\chi_i}$. Set $\mu_t = \mu_{i,t}$ on $A_i$ and $\mu_t = 0$ elsewhere. Define $\Tilde{\mu_t}$ to be the pullback of $\mu_t$ under $\pi_f$, i.e. $\Tilde{\mu_t} = \pi_f^*\mu_t$ on $\Tilde{\mathscr{R}}_f$ and $\Tilde{\mu_t} = 0$ elsewhere. Integrate $\mu_t$ with a choice of normalization (not depending on $t$) by fixing 3 distinct points in $Post_f$, we get a quasiconformal mapping $\phi_t$. 

\begin{definition}
The path $f_t = \phi_t\circ{f}\circ{\phi^{-1}_t}$ ($t\geq 0$) is called a \textbf{pinching path} starting from $f$ supported on $\mathscr{A}$.
\end{definition}

The following result in \cite{CuiTan} affirms that the pinching path is converging:
\begin{theorem}\label{thm.pinching}
Let $f$ be a geometrically finite rational map. Let $f_t = \phi_t\circ{f}\circ{\phi_t^{-1}}$ be a pinching path starting from $f$ supported on a non-separating multi-annulus. Then the following properties hold:
\begin{enumerate}
    \item $f_t$ converges uniformly to a geometrically finite rational map $g$ as $t\to \infty$.
    \item $\phi_t$ converges uniformly to a continuous surjective map $\varphi$ with $\varphi(\mathcal{J}_f) = \mathcal{J}_g$.
    \item For each fill-in skeleton $\hat{S}$, $\varphi(\hat{S})$ is a parabolic (pre-)periodic point. Moreover $\varphi$ is injective in the complement of the union of all fill-in skeletons.
\end{enumerate}
\end{theorem}

\subsection{Admissible petals}
Let $R:\mathbb{C}\longrightarrow\mathbb{C}$ be a rational map with $R_a(0) = 0,R'_a(0) = e^{2\pi i\frac{p}{q}}$. Write the Taylor expansion near 0:
$R^q(z) = z + \omega(a)z^{vq+1} + o(z^{vq+1})$.
There exactly $v$ cycles of immediate basins with rotation number $p/q$ around 0. Let $B^*_0(0),...,B^*_{q-1}(0)$ be one of such cycle, written in cycle order. On every $B^*_m(0)$ there is a unique Fatou coordinate up to translation $\phi_k:B^*_m(0)\longrightarrow\mathbb{C}$ semi-conjugating $R$ to $z\mapsto z+1$. 

\begin{definition}
An \textbf{admissible petal} $P^{\gamma}\subset B^*_m(0)$ is a petal such that the boundary of $\phi_k(P)$ is a smooth curve $\gamma$ (well-defined up to translation) intersecting every horizontal line $y=b$ at exactly one point. We say that $P^{\gamma}$ is \textbf{standard} if $\gamma$ is a vertical line. $P^\gamma$ is called \textbf{maximal} if $\partial P^{\gamma}$ contains at least one critical point of $R^q|_{B^*_m(0)}$. Clearly once $\gamma$ is chosen, the corresponding maximal petal $P^\gamma$ is unique.
\end{definition}

\begin{lemma}\label{lem.admissible.petal}
If $P^{\gamma}\subset B^*_m(0)$ is an admissible petal, then so is $R^n(P^{\gamma})$ for any $n\geq 0$. Moreover $R^n(P^{\gamma})$ are also associated to $\gamma$.
\end{lemma}

Let $P^0_m:=P\subset B^*_m(0)$ be a maximal admissible petal (we omit the index of associated curve $\gamma$). For any $0\leq k\leq q-1$, $n\in \mathbb{Z}$ such that $np+k=m \pmod q$, $P^n_{k}$ is defined to be the (unique) connected component of $R^{-n}(P^0_m)$ satisfying $P^n_{k}\subset B^*_k(0)$ and $0\in\partial P^n_{k}$. See the sequence of mappings below, where the bar means $\pmod q$:
\begin{equation}\label{eq.sequence}
    P^n_{\overline{m-np}}\xrightarrow[]{R}...\xrightarrow[]{R}P^1_{\overline{m-p}}\xrightarrow[]{R} P^0_m\xrightarrow[]{R} P^{-1}_{\overline{m+p}}\xrightarrow[]{R}...\xrightarrow[]{R}P^{-n}_{\overline{m+np}}
\end{equation}

\subsection{Holomorphic dependence of Fatou coordinate}

Let $R_{a}:\mathbb{C}\longrightarrow\mathbb{C}$ be an analytic family (parametrized by $a\in \Lambda\subset \mathbb{C}$, $\Lambda$ open) with $R_a(0) = 0,R'_a(0) = e^{2\pi i\frac{p}{q}}$. Write the Taylor expansion near 0:
$R^q_{a}(z) = z + \omega(a)z^{vq+1} + o(z^{vq+1})$. In this section we always suppose that at some $a_0$, $\omega(a_0)\not = 0$. Then for every attracting (resp. repelling) axis of $R_{a_0}$, there exists
\begin{itemize}
    \item a small neighborhood $D_{a_0}$ of $a_0$; a topological disk $V_{att} = \{x+iy;\,x\textgreater c-b|y|\}$ (resp. $V_{rep} = \{x+iy;\,x\textless c-b|y|\}$), where the constants $b,c$ are positive and do not depend on $a$; a family of attracting (resp. repelling) petals $(P_a)_{a\in B(a_0)}$
    \item a family of Fatou coordinates $(\phi_a)_{a\in D_{a_0}}$ of $R_a$ such that $\phi_a: P_a\longrightarrow V_{att}\,\, (\text{resp. }V_{rep})$ is conformal and $\phi^{-1}_a$ is analytic in $a$.
\end{itemize}

An immediate consequence from this is that $h_a = \phi^{-1}_a\circ{\phi_{a_0}}$ (parametrized by $a\in U_{a_0}$) defines a dynamical holomorphic motion of ${P_{a_0}}$ such that $h_a({P_{a_0}})={P_{a}}$.

Next we prove that $h_a$ can be extended to "sub-petal" arbitrarily close to the maximal petal $\Omega_{a_0}$ of $R_{a_0}$. Suppose that $\phi_{a_0}:\Omega_{a_0}\longrightarrow \mathbb{H}_\rho$ is conformal, where $\mathbb{H}_\rho = \{z;\,\mathfrak{Re}(z)\textgreater\rho\}$. For any $\delta\textgreater 0$, define
$\Omega^{\delta}_{a_0} = \phi^{-1}_{a_0}(\mathbb{H}_{\rho+\delta})$. Notice that there exists $k\geq 1$ such that $h_a$ is well-defined on $\Tilde{P}_{a_0} := R^{kq}_{a_0}(\Omega^{\delta}_{a_0})$. Set $\Tilde{P}_{a} := h_a(\Tilde{P}_{a_0})$.

\begin{lemma}
Let $W_a$ be the connected component of $R_a^{-kq}(\Tilde{P}_{a})$ containing $\Tilde{P}_a$. Then for $a$ close enough to $a_0$, $R^{kq}_a:W_a\longrightarrow\Tilde{P}_{a}$ is conformal (injective, surjective).
\end{lemma}
\begin{proof}
Let $a_n\to a_0$. First we prove that 
\[\forall\epsilon\textgreater 0, \exists N \text{ such that }\forall n\geq N, \overline{W_{a_n}}\subset_\epsilon \overline{\Omega^{\delta}_{a_0}}\]
where for two sets $A,B\subset\mathbb{C}$, $A\subset_{\epsilon}B$ means that $A$ is contained in the $\epsilon$-neighborhood of $B$.
Suppose the contrary, then up to taking a subsequence we may suppose that 
\begin{equation}\label{eq.pfdetail}
    \exists\epsilon_0\textgreater 0, \text{ a sequence } x_n\in \overline{W_{a_n}} \text{ with } x_n\to x_0,\text{ such that } dist(x_n,\overline{\Omega^{\delta}_{a_0}})\textgreater \epsilon_0.
\end{equation}
Clearly $R^{kq}_{a_0}(x_0)\in \overline{\Tilde{P}_{a_0}}$ since $R^{kq}_{a_n}(x_n)\in \overline{\Tilde{P}_{a_n}}$. For every $n$ take a path $\gamma_n\subset \overline{W_{a_n}}$ connecting $0,x_n$ such that $\gamma\cap \overline{\Tilde{P}_{a_0}}$ does not depend on $n$. This is possible since $\overline{\Tilde{P}_{a_0}}$ moves holomorphically. Moreover since the Hausdorff topology on the space of compact sets is sequentially compact, we may suppose that $\gamma_n$ converges to a connected compact set $\gamma_0$ containing 0 and $x_0$. This implies that $dist(x_0,\overline{\Tilde{\Omega}_{a_0}}) = 0$, contradicting (\ref{eq.pfdetail}). Hence for $a$ close enough to $a_0$, $W_{a}$ contains no critical points of $R^{kq}_a$ and the lemma is proven.
\end{proof}

Pull back $h_a$ by $R^{kq}_{a_0},R^{kq}_{a}$ and apply $\lambda$-Lemma we get immediately:
\begin{corollary}\label{cor.fatou.extend}
For any $\delta\textgreater 0$, the holomorphic motion $h_a:\Tilde{P}_{a_0}\longrightarrow\Tilde{P}_{a}$ can be dynamically extended to $\overline{\Omega^{\delta}_{a_0}}$ when $a$ is close enough to $a_0$. 
\end{corollary} 
 
\begin{proposition}\label{prop.first_critical_point}
Suppose for $a$ in a neighborhood of $a_0$, $R^q_{a}$ has $l$ (does not depend on $a$) distinct critical points $c_1(a),...,c_l(a)$ which vary analytically on $a$. Suppose moreover that the immediate basin of $R_a$ is always simply connected. If $c_1(a_0)$ is the unique critical point on $\partial\Omega_{a_0}$, then for $a$ close enough to $a_0$, $c_1(a)$ is also the unique critical point on $\partial\Omega_{a_0}$.
\end{proposition}
\begin{proof}
Clearly for $a$ close enough to $a_0$, $c_1(a)$ is in the immediate basin. Indeed, first take a path $\gamma_{a_0}$ in the immediate basin, linking $c_1(a_0)$ and some $z_{a_0}$ in a small petal $P_{a_0}$. One can connected it with $h_a(z_{a_0})\in P_a = h_a(P_{a_0})$ by a compact path $\gamma_a$ which is very close to $\gamma_{a_0}$. Hence for $a$ very close to $a_0$, $\gamma_{a}$ is attracted by $z=0$.

Hence we can normalise $\phi_a$ by taking $\phi_a(c_1(a)) = 0$.  By Corollary \ref{cor.fatou.extend}, $\phi^{-1}_a:\mathbb{H}_{1}\longrightarrow \tilde{\Omega}_a$ is conformal such that $R^q_{a}(c_1(a))\in\partial\tilde{{\Omega}}_a$. Now we prove that for $a$ close enough to $a_0$, the component of $(R^q)^{-1}_a(\tilde{\Omega}_a)$ containing $\tilde{\Omega}_a$, denoted by ${\Omega}_a$, contains no critical point of $R^q_a(\tilde{\Omega}_a)$. Suppose the contrary, then there is a sequence $a_n\to a_0$ such that, without loss of generality, $c_2(a_n)\in{\Omega}_{a_n}$. Up to taking a subsequence, $\overline{{\Omega}_{a_n}}$ has a limit ${\Omega}^*$ in the Hausdorff topology. Moreover ${\Omega}^*$ is compact, connected and ${\Omega}^*$ is the closure of the union of several components of $(R^q)^{-1}_{a_0}(\tilde{\Omega}_{a_0})$. Clearly $\overline{{\Omega}_{a_0}}\subset {\Omega}^*$. Suppose $c_1(a)$ has multiplicity $m$, then besides ${\Omega}_{a}$, there are $m$ distinct components of $(R^q_{a})^{-1}(\tilde{\Omega}_{a})$ attached at $c_1(a)$ since the basin is simply connected. Denote by $U_{a,1},...,U_{a,m}$ these components. Then up to taking subsequences, $\overline{U_{a_n,i}}$ admits a limit $U^*_{i}$ containing $\overline{U_{a_0,i}}$. Since ${\Omega}_{a_n}\cap U_{a_n,i} = \emptyset$, so ${\Omega}^*\cap U_i^*\subset \partial\Omega^*$, ${\Omega}^* = \Omega_{a_0}$. Hence $c_2(a_0)\in\partial {\Omega}_{a_0}$, a contradiction.
\end{proof}

\bibliographystyle{plain}
\bibliography{references}

\href{mailto:runze.zhang@math.univ-toulouse.fr}{\nolinkurl{runze.zhang@math.univ-toulouse.fr}}
\end{document}